\documentclass[a4paper,reqno]{amsart}
\usepackage[left=3cm,right=3cm,top=3cm,bottom=3cm,headsep=0.7cm]{geometry}

\usepackage[english]{babel}
\usepackage{amssymb, amsmath, amsthm, bbm, mathtools} 
\usepackage{graphicx}
\usepackage{caption}
\usepackage{subfigure,color}
\usepackage[font={small,up}]{caption}
\newcommand{\res}{\mathop{\hbox{\vrule height 7pt width .5pt depth 0pt
\vrule height .5pt width 6pt depth 0pt}}\nolimits}

\numberwithin{equation}{section} 

\usepackage{pgfplots}
\usepackage{bm}

\usepackage{color}
\usepackage{hyperref}
\definecolor{ddorange}{rgb}{1,0.5,0}
\definecolor{ddcyan}{rgb}{0,0.2,1.0}

\newcommand{\EEE}{\color{black}}

\usepackage{amsthm}
\makeatletter
\def\th@plain{%
  \thm@notefont{}
  \itshape 
}
\def\th@definition{%
  \thm@notefont{}
  \normalfont 
}
\makeatother
\newtheorem{thm}{Theorem}[section]
\newtheorem{prop}[thm]{Proposition}
\newtheorem{lem}[thm]{Lemma}

\theoremstyle{definition}
\newtheorem{defn}[thm]{Definition}

\newtheorem{oss}[thm]{Remark}

\newcommand{\Z}{\mathbb{Z}}

\newcommand{\N}{\mathbb{N}}
\newcommand{\R}{\mathbb{R}}

\renewcommand{\epsilon}{\varepsilon}

\DeclareMathOperator{\dist}{dist}

\newcommand{\sm}{\setminus}
\newcommand{\wt}{\widetilde}
\newcommand{\hd}{\mathcal{H}^{d-1}}
\newcommand{\Rd}{\mathbb{R}^d}
\newcommand{\Ld}{\mathcal{L}^d}
\newcommand{\dx}{\,\mathrm{d}x}
\newcommand{\NI}{^{\mathrm{NI}}}
\newcommand{\NIM}{^{\mathrm{NI},M}}
\newcommand{\etl}{_\varepsilon^{\beta, l}}

\newcommand{\dod}{{\partial_D \Omega}}
\newcommand{\don}{{\partial_N \Omega}}
\newcommand{\dom}{{\partial \Omega}}
\newcommand{\Mdd}{{\R^{d\times d}_{sym}}}
\newcommand{\ltm}{^{\lambda, \theta}}
\newcommand{\ltms}{_{\lambda, \theta}}
\newcommand{\Dir}{^{\mathrm{Dir}}}

\newcommand{\wstar}{\stackrel{*}\rightharpoonup}

\newcommand{\Addresses}{{
  \bigskip
  \footnotesize

(V.~Crismale)  \textsc{CMAP, \'Ecole Polytechnique, CNRS, 91128 Palaiseau Cedex, France.}\par\nopagebreak
  \textit{E-mail address}, V.~Crismale: \texttt{vito.crismale@polytechnique.edu}

\medskip

(G.~Scilla) \textsc{Dipartimento di Matematica ed Applicazioni ``R. Caccioppoli'', Universit\`{a} di Napoli Federico~II, Via Cintia Monte Sant'Angelo, 80126 Napoli, Italy.}\par\nopagebreak
  \textit{E-mail address}, G.~Scilla: \texttt{giovanni.scilla@unina.it}

\medskip

(F.~Solombrino) \textsc{Dipartimento di Matematica ed Applicazioni ``R. Caccioppoli'', Universit\`{a} di Napoli Federico~II, Via Cintia Monte Sant'Angelo, 80126 Napoli, Italy.}\par\nopagebreak
  \textit{E-mail address}, F.~Solombrino: \texttt{francesco.solombrino@unina.it}
}}

\def\e{\epsilon}
\usepackage{enumitem}
\setlist[description]{nosep}

\author{V. Crismale, G. Scilla and F. Solombrino}

\title{A derivation of Griffith functionals from discrete finite-difference models}
\date{}

\begin{document}

\begin{abstract}
We analyze a finite-difference approximation of a functional of Ambrosio-Tortorelli type in brittle fracture, in the discrete-to-continuum limit. In a suitable regime between the competing scales, namely if the discretization step $\delta$ is smaller than the ellipticity parameter $\epsilon$, we show the $\Gamma$\hbox{-}convergence of the model to the Griffith functional, containing only a term enforcing Dirichlet boundary conditions and no $L^p$ fidelity term. Restricting to two dimensions, we also address the case in which a (linearized) constraint of non-interpenetration of matter is added in the limit functional, in the spirit of a recent work by Chambolle, Conti and Francfort. 
\end{abstract}


\maketitle

\setcounter{tocdepth}{1}  
\tableofcontents

\section{Introduction}


In this paper we provide a variational approximation by discrete finite-difference energies of functionals of the form
\begin{equation}
\lambda\int_{\Omega\backslash K} |\mathcal{E}u(x)|^2\,\mathrm{d}x + \mu\int_{\Omega\backslash K} |{\rm div}\,u(x)|^2\,\mathrm{d}x  + \mathcal{H}^{d-1}(K), 
\label{Griff1}
\end{equation}
where $\Omega$ is a bounded subset of $\R^d$, $K\subseteq\Omega$ is closed, $u\in C^1(\Omega\backslash K;\R^d)$, $\mathcal{E}u$ denotes the symmetric part of the gradient of $u$, ${\rm div}\,u$ is the divergence of $u$ and $\mathcal{H}^{d-1}$ is the $(d-1)$-dimensional Hausdorff measure.  Functionals as in \eqref{Griff1} are widely used in the variational modeling of fracture mechanics for linearly elastic materials, in the framework of Griffith's theory of brittle fracture (see, e.g. \cite{Griffith}). Here $\Omega$ stands for the reference configuration and $u$ represents the displacement field of the body. The total energy \eqref{Griff1} is composed by a bulk energy in $\Omega\backslash K$, where the material is supposed to be linearly elastic, and a surface term accounting for the energy necessary to produce the fracture, proportional to the area of the crack surface $K$. A rigorous weak formulation of the problem \eqref{Griff1}, which is usually complemented by the assignment of boundary Dirichlet datum, has been provided only in very recent years \cite{DM2013,CC18Comp}. In the appropriate functional setting, $u$ is a (vector-valued) generalized special function of bounded deformation, for which the symmetrized gradient $\mathcal{E}u$ and the divergence ${\rm div}\,u$ are defined almost everywhere in an approximate sense (see~\cite{DM2013}), and the set $K$ is replaced by the $(d-1)$-rectifiable set $J_u$, the jump set of $u$.

However, the numerical treatment of functionals \eqref{Griff1} presents relevant difficulties mainly connected to the presence of the surface term $\mathcal{H}^{d-1}(J_u)$. Such difficulties already appear in the case of antiplane shear (see, e.g., \cite{BFM}) where the energy \eqref{Griff1} reduces to the
Mumford-Shah-type functional
\begin{equation}
\int_\Omega |\nabla u|^2\, \mathrm{d}x + \mathcal{H}^{d-1}(J_u), 
\label{MS}
\end{equation}
for a scalar-valued displacement $u\in SBV(\Omega)$, the space of special functions of bounded variation.
In view of the aforementioned numerical issues, a particular attention has been devoted over the last three decades to provide suitable discrete approximations, by means of both finite-difference and finite-elements, of the functional \eqref{MS}. 

A first approach, based on earlier models in Image Segmentation, has been proposed by Chambolle~\cite{Chambolle1} in dimension $d=1,2$; there, the discrete model depends on finite differences through a truncated quadratic potential. In the case $d=2$, the surface term of the variational limit is described by an anisotropic function $\varphi(\nu_u)$ of the normal $\nu_u$ to $J_u$ depending on the geometry of the underlying lattice. As a matter of fact, this anisotropy can be avoided by considering alternate finite-elements of different local approximations of the Mumford-Shah functional, as showed, still in dimension two, by Chambolle and Dal Maso~\cite{ChambDM}. We refer to \cite{BelCos94} (cf.\ also \cite{Bou99}) and to \cite{Gob98} for some other approximations using finite-elements and continuous finite-difference approximations of \eqref{MS}, respectively. 

A different strategy consists in replacing the Mumford-Shah functional by an elliptic approximation (with parameter $\epsilon>0$) in the spirit of Ambrosio-Tortorelli~\cite{AT1,AT2}, and then by discretizing these elliptic functionals by means of either finite-difference or finite-elements with mesh-size $\delta$, independent of $\epsilon$. For a suitable fine mesh, with size $\delta=\delta(\epsilon)$ small enough, these numerical approximations $\Gamma$-converge, as $\epsilon\to0$, to the Mumford-Shah functional. 

This suggests that a remarkable problem to be addressed is the so called ``quantitative analysis'': i.e., the study of the limit behavior of these approximations as $\delta$ and $\epsilon$ simultaneously tend to 0.
Following on the footsteps of the approximation of the Modica-Mortola functional proposed by Braides and Yip~\cite{BY}, this analysis has been recently developed by Bach, Braides and Zeppieri in \cite{BBZ} for \eqref{MS}. They characterize the limit behavior of the energies
\begin{equation*}
\sum_{\substack{\alpha,\beta\in \Omega\cap\delta\Z^d\\|\alpha-\beta|=\delta}}\delta^d (v(\alpha))^2\left|\frac{u(\alpha)-u(\beta)}{\delta}\right|^2 + \sum_{\alpha\in \Omega\cap\delta\Z^d}\delta^d\frac{(v(\alpha)-1)^2}{\epsilon} + \frac{1}{2}\sum_{\substack{\alpha,\beta\in \Omega\cap\delta\Z^d\\|\alpha-\beta|=\delta}}\epsilon\delta^d \left|\frac{v(\alpha)-v(\beta)}{\delta}\right|^2,
\end{equation*}
showing the variational convergence to the 
functional \eqref{MS} in the regime $\delta\!<\!<\epsilon$. 
Other scalings of the parameters are also studied: in the regime $\delta\sim\epsilon$, the surface energy is described by a function $\varphi(\nu_u)$ solution to a discrete optimal-profile problem, while if $\delta\!>\!>\epsilon$, the limit energy is the Dirichlet functional.
Recently, approximations of \eqref{MS} (thus without anisotropy in the limit) have been obtained even when $\delta \sim \epsilon$, by employing discretizations on random lattices. In particular, \cite{BCR} analyzes the random version of the discrete energies in \cite{BBZ}, basing on \cite{Ruf} (cf.\ also \cite{CDSZ19}).  

Coming back to the problem of providing discrete approximations of the Griffith functional, we mention the finite-elements approximation in \cite{Neg03} and focus on the discrete-to-continuum analysis performed by Alicando, Focardi and Gelli~\cite{AFG}. They considered, in the spirit of  \cite{Chambolle1} and in the planar setting $d=2$, discrete energies of the form
\begin{equation}
\sum_{\xi\in\Z^d}\rho(\xi)\sum_{\alpha\in R_\delta^\xi}\delta^{d-1}f\left(\delta(|\langle D_\delta^\xi u(\alpha),\xi\rangle|^2+\theta|{\rm div}_\delta^\xi u(\alpha)|^2)\right),
\label{AFGenergies}
\end{equation}
defined on a portion $R_\delta^\xi$ of $\Omega\cap\delta\Z^d$, where $\rho$ is a positive kernel, $\theta$ is a positive constant, $f(t):=\min\{t,1\}$, $D_\delta^\xi u(x)$ denotes the difference quotient $\frac{1}{\delta}(u(x+\delta\xi)-u(x))$ and ${\rm div}_\delta^\xi \, u$ is a suitable discretization of the divergence which takes into account three-point-interactions in the directions $\xi$ and $\xi^\perp$ (the vector orthogonal to $\xi$). In order to obtain compactness of sequences of competitors with equibounded energy, they require that $\rho(\xi)>0$ for $\xi\in\{\pm e_1, \pm e_2, \pm (e_1\pm e_2)\}$, which amounts to consider nearest-neighbors (NN) and next-to-nearest neighbors (NNN) interactions in the energies. Furthermore, an $L^\infty$ bound has to be imposed, which is quite unnatural in Fracture Mechanics. Differently from \cite{BBZ}, the characterization of the limit energy cannot be achieved with the reduction to a 1-dimensional case by means of slicing techniques (see, e.g., \cite{BG, Chambolle1999, Gob98}), due to the presence of the divergence term. Hence, a different strategy has to be used, involving the construction of suitable interpolants (see \cite[Proposition~4.1]{AFG}). As it happened in \cite{Chambolle1}, the surface term in the limit energy is still reminiscent of the underlying lattice, and only a continuous version of \eqref{AFGenergies} allows to obtain 
$\mathcal{H}^{d-1}(J_u)$ as surface energy. Furthermore, a possible extension of the model to dimension $d=3$, still involving NN and NNN interactions is proposed, but no compactness result is provided.\\
\\
\emph{Our results:} This leads us to the motivation of our paper, which complements the results of both \cite{BBZ} and \cite{AFG}.
On the one hand, 
we provide a discrete Ambrosio-Tortorelli approximation to the Griffith functional both in dimension $d=2$ and $d=3$, of the form 
\begin{equation}
\begin{split}
&\frac{1}{2}\sum_{\xi\in S_d}\sum_{\alpha\in R_\delta^\xi(\Omega)}\delta^{d-2}(v(\alpha))^2\left|D_{\delta,\xi}u(\alpha)\right|^2 + \frac{1}{2^d}\sum_{\alpha\in R_\delta^{{\rm div}}(\Omega)}\delta^{d-2}(v(\alpha))^2\left|{\rm Div}_{\delta}u(\alpha)\right|^2\\
&+\frac12\sum_{\alpha\in \Omega\cap\delta\Z^d}\delta^d\Bigg(\frac{1}{\epsilon}(v(\alpha)-1)^2 + \epsilon\sum_{\substack{k=1\\\alpha+\delta e_k\in \Omega\cap\delta\Z^d}}^d\left(\frac{v(\alpha+\delta e_k)-v(\alpha)}{\delta}\right)^2\Bigg)\,,
\end{split}
\label{model}
\end{equation}
where $S_d$ is a set of lattice directions (depending on the dimension $d$), $D_{\delta,\xi}u$ and ${\rm Div}_{\delta}\,u$ are suitable discretizations of the symmetrized gradient and of the divergence of the vector-valued $u$, and the latter term is a discrete Modica-Mortola functional. Notice that ${\rm Div}_{\delta}\,u$ takes into account $(d+1)$-point-interactions on a complete set of orthogonal directions (see \eqref{discrediv2}). Then we prove, as main result (Theorem~\ref{teo:main}), that \eqref{model} $\Gamma$-converges as $\epsilon\to0$ to the Griffith's functional under the assumption that $\delta\!<\!<\epsilon$. 

On the other hand, we conclude the analysis started in \cite{AFG} for the finite-difference approximation of \eqref{Griff1} in dimension $d=3$, although with a different approach, by both  rigorously proving a compactness result under more general assumptions, and recovering an isotropic surface energy in the limit.
We also stress the fact that the extension of the two-dimensional model to the case $d=3$ is not just a minor modification but requires the introduction of additional interactions in the elastic term of the energies by specifying the set of directions $S_3$ (see \eqref{setdirec}); namely, we need to take into account also next-to-next-nearest neighbors (NNNN) interactions, corresponding to lattice vectors $\xi\in\{\pm(e_1\pm e_2\pm e_3)\}$.

The aforementioned compactness result, which is the content of Proposition~\ref{prop:compactness}, determines the functional space domain of the limit: we benefit from the recent results \cite{CC18Comp, CriFri19} and prove that sequences $(u_\epsilon,v_\epsilon)$ with equibounded energies \eqref{model} converge (up to subsequences) to a limit pair $(u,v)\in GSBD^2_\infty(\Omega)\times\{1\}$. We refer the reader to Section~\ref{sec:prelim} for a precise definition of this function space, where also the value $\infty$ is allowed.
We underline that our compactness result, valid under the weaker assumption that $\frac{\delta}{\epsilon}$ be bounded, cannot be obtained in our view \EEE through any slicing procedure (as it happened, on the contrary, in \cite{BBZ}) and also refines the compactness lemma \cite[Lemma~1]{Chambolle1999} to deal with the vector-valued case. \EEE Indeed, while in the scalar-valued case controlling the total variation along $d$ independent slices of $u_\e$ is enough to  provide $BV$-compactness, no analogue procedure is at the moment known  in $GSBD$ (whose definition \cite[Definition 4.1]{DM2013} in principle requires a uniform control of the symmetrized slices on a dense set of directions in the unit sphere, cf.\ also \cite[Remark~4.15]{DM2013}). Such issue prevents us to get a uniform bound in $GSBD$ from a control on the slices corresponding to the  directions of the lattice vectors, that could be easily obtained from the discrete functional as in \cite{BBZ, Chambolle1999}. We notice that the situation is different with respect to the $BD$ case, where it is enough to control the slices on a finite set of directions, see \cite[Proposition~3.2]{ACDM}. \EEE 

In fact, we are able to prove that a continuous Ambrosio-Tortorelli functional, defined on the standard piecewise affine interpolations $\bar u_\epsilon$ of the $u_\epsilon$ and on  suitable piecewise constant interpolations $\tilde v_{min, \epsilon}$ of the $v_\epsilon$ (different than the standard ones),  bounds from below the discrete energies \eqref{model}. To this aim, taking the additional (NNNN) interactions is crucial in dimension $d=3$ . In addition, we do not need to add any $L^p$ fidelity term to the discrete energies, since compactness in $GSBD^2_\infty$ does not require such limitations and is also able to handle the fact that $u$ may take value $\infty$.

The proof of the $\Gamma$-liminf inequality is subdivided into two steps. The lower semicontinuity of the elastic part of the limit energies (see Lemma~\ref{le:technical1} and Proposition~\ref{prop:lowerbound}) can be obtained by combining slicing arguments on suitable interpolations of $u_\epsilon$ and $v_\epsilon$ with a splitting into sublattices of $\delta \Z^d$, which are frequently used techniques to work 
with discrete energies with both short and long-range interactions (see, e.g., \cite{AFG,BG}). 
 It must be noticed, at this point, that both the first two summands in \eqref{model} give a contribution to the second term in \eqref{Griff1}. As, within the proof technique described above, both are assumed to be nonnegative, the constants $\mu$ and $\lambda$ appearing there are related by $2\mu=\lambda+2\theta$ with $\theta \geq 0$, as it also happened, for instance, in the statement of \cite[Theorem 7.1]{AFG}. Hence, our main result is stated in terms of the two indendent parameters $\lambda$ and $\theta$ and is valid for materials whose Poisson ratio (due to the inequality $2\mu \geq \lambda$) does not exceed the value $\frac13$. \EEE

The lower bound for the surface term, instead, requires a more refined blow-up procedure (Proposition~\ref{lowbound}) and this is the very first technical point where we need to assume that $\frac{\delta}{\epsilon}\to0$, in order to recover the optimal constant. Indeed a slicing argument under the weaker assumption that $\delta/\epsilon$ be bounded would provide a lower bound with a wrong constant.  We remark that, also in this proof, similar arguments as in Proposition~\ref{prop:compactness} have to be used, in order to get compactness of a rescaled version of the $u_\epsilon$. Moreover,
additional care is needed in order to deal with the fact that our limit displacements may assume the value infinity (see e.g.\ Step~2 in Proposition~\ref{prop:lowerbound}). 

The construction of a recovery sequence (Proposition~\ref{prop:upperbound}) relies on the density result for $GSBD^2$ functions \cite[Theorem~1.1]{CC}, recalled here with Theorem~\ref{thm:density}. The upper bound for the elastic term is obtained by first reducing the discrete energies to continuous ones by means of a classical translation argument (see, e.g. \cite[Proposition~4.4]{AFG}) and then by exploiting the upper estimates coming from the approximations of $\int|\langle (\mathcal{E}u)\xi,\xi\rangle|^2\,\mathrm{d}x$ and $\int({\rm div}\,u)^2\,\mathrm{d}x$ outside an infinitesimal neighborhood of the jump set of the target function $u$. The limsup inequality for the surface term is developed as in \cite[Proposition~4.2]{BBZ}, by also employing the one-dimensional solution to the Ambrosio-Tortorelli optimal profile problem.

We conclude our analysis by investigating the compatibility of our two-dimensional model with the constraint of non-interpenetration. The answer is positive under the assumptions of \cite{CCF18ARMA} but, in order to obtain the desired upper bound, we need to require the stronger scaling $\frac{\delta}{\epsilon^2}\to0$ between the parameters.

As a final remark, we mention that our results also give a partial insight on the case $\delta \sim \epsilon$. Indeed, the constructions in Sections \ref{sec:lowerbound} and \ref{sec:upperbound} can also be used to show that, whenever the ratio $\delta/\epsilon$ stays bounded, the $\Gamma$-limit of the energy \eqref{model} can be controlled from above and from below by functionals of the kind \eqref{Griff1}, with different constants appearing in the surface term. However, a precise characterization of the limit energy in this case has to face additional issues. 
The analysis performed in \cite{BBZ} for the scalar-valued case, indeed, relies indeed on two major ingredients. First of all, the limit energy is characterized as an abstract integral surface energy by means of the global method for relaxation introduced in \cite{BFMA}. This could be also done in our setting, by exploiting a recent integral representation result for energies on spaces of functions of bounded deformation \cite{CFS} (see also \cite{CFI} in the planar setting). However, a crucial step in this procedure consists in proving that a separation of bulk and
surface contributions takes place in the limit. In \cite{BBZ} this is done by means of an explicit construction which, however, is confined to $2$ dimensions and  strongly exploits  the $SBV$-setting. A more general point of view, also suitable for higher dimensions,  is for instance used in \cite[Proposition~4.11]{BCR} with the help of a weighted coarea formula. This is unfortunately also a tool which is not available when dealing with $(G)SBD$ functions. The investigation of  these issues has therefore to be deferred to further contributions. \EEE
%
\\
\\
\emph{Outline of the paper:}  The paper is organized as follows. In Section~\ref{sec:prelim} we fix the basic notation and collect some definitions and results on the function spaces we will deal with. In Section~\ref{sec:discrmodel} we introduce our discrete model and state the main results of the paper. Section~\ref{sec:compprelboun} contains the compactness result of Proposition~\ref{prop:compactness}. Section~\ref{sec:lowerbound} is devoted to the liminf inequality, proved with Proposition~\ref{lowbound}, while Section~\ref{sec:upperbound} deals with the upper inequality (Proposition~\ref{prop:upperbound}). Eventually, in Section~\ref{sec:nonimp} we analyze the compatibility of the two-dimensional model with a non-interpenetration constraint.

\section{Preliminaries}\label{sec:prelim}

\subsection{Notation}
The symbol $\langle\cdot,\cdot\rangle$ denotes the scalar product in $\mathbb{R}^d$, while $|\cdot|$ stands for the Euclidean norm in any dimension. For any $x,y\in\mathbb{R}^d$, $[x,y]$ is the segment with endpoints $x$ and $y$. The symbol $\Omega$ will always denote an open, bounded subset of $\mathbb{R}^d$. The Lebesgue measure in $\mathbb{R}^d$ and the $s$-dimensional Hausdorff measure are written as $\mathcal{L}^d$ and $\mathcal{H}^s$, respectively. We will often use the notation $|A|$ 
for the Lebesgue measure of a Borel set $A$. The symbols $\lesssim$ and $\gtrsim$ denote the boundedness modulo a constant. 

For any locally compact subset $B  \subset \R^d$ (i.e.\ any point in $B$ has a neighborhood contained in a compact subset of $B$),
the space of bounded $\R^m$-valued Radon measures on $B$ [respectively, the space of $\R^m$-valued Radon measures on $B$] is denoted by $\mathcal{M}_b(B;\R^m)$ [resp., by $\mathcal{M}(B;\R^m)$]. If $m=1$, we write $\mathcal{M}_b(B)$ for $\mathcal{M}_b(B;\R)$, $\mathcal{M}(B)$ for $\mathcal{M}(B;\R)$, and $\mathcal{M}^+_b(B)$ for the subspace of positive measures of $\mathcal{M}_b(B)$. For every $\mu \in \mathcal{M}_b(B;\R^m)$, its total variation is denoted by $|\mu|(B)$. We write $\{e_1,\dots,e_d\}$ for the canonical basis of $\R^d$.
\subsection{$GBD$, $GSBD$, and $GSBD^2_\infty$ functions}

We recall here some basic definitions and results on generalized functions with bounded deformation, as introduced in \cite{DM2013}. Throughout the paper we will use standard notations for the spaces $SBV$ and $SBD$, referring the reader to \cite{AFP} and \cite{ACDM, BCDM, Temam}, respectively, for a detailed treatment on the topics.\\

Let $\xi\in\R^d\backslash\{0\}$ and $\Pi^\xi=\{y\in\R^d:\, \langle\xi,y\rangle=0\}$. If $y\in\Pi^\xi$ and $\Omega\subset\R^d$ we set $\Omega_{\xi,y}:=\{t\in\R:\, y+t\xi\in \Omega\}$ and $\Omega_\xi:=\{y\in \Pi^\xi:\, \Omega_{\xi,y}\neq\emptyset\}$. Given $u:\Omega\to\R^d$, $d\geq2$, we define $u^{\xi,y}: \Omega_{\xi,y}\to\R$ by 
\begin{equation}
u^{\xi,y}(t):=\langle u(y+t\xi),\xi\rangle\,, 
\label{section1}
\end{equation}
while if $h: \Omega\to\R$, the symbol $h^{\xi,y}$ will denote the restriction of $h$ to the set $\Omega_{\xi,y}$; namely,
\begin{equation}
h^{\xi,y}(t):= h(y+t\xi)\,.
\label{section2}
\end{equation}

\begin{defn}
An $\mathcal L^{d}$-measurable function $u:\Omega\to \R^{d}$ belongs to $GBD(\Omega)$ if there exists a positive bounded Radon measure $\lambda_u$ such that, for all $\tau \in C^{1}(\R^{d})$ with $-\frac12 \le \tau \le \frac12$ and $0\le \tau'\le 1$, and all $\xi \in S^{d-1}$, the distributional derivative $D_\xi (\tau(\langle u,\xi\rangle))$ is a bounded Radon measure on $\Omega$ whose total variation satisfies
$$
\left|D_\xi (\tau(\langle u,\xi\rangle))\right|(B)\le \lambda_u(B)
$$
for every Borel subset $B$ of $\Omega$. 
\end{defn}

If $u\in GBD(\Omega)$ and $\xi\in\R^d\backslash\{0\}$ then, in view of \cite[Theorem~9.1, Theorem~8.1]{DM2013}, the following properties hold:
\begin{enumerate}
\item[{\rm(a)}] $\dot{u}^{\xi,y}(t)=\langle\mathcal{E}u(y+t\xi)\xi,\xi\rangle$ for a.e. $t\in \Omega_y^\xi$;\\
\item[{\rm(b)}] $J_{u^{\xi,y}}=(J_u^\xi)_y^\xi$ for $\mathcal{H}^{n-1}$-a.e. $y\in\Pi^\xi$, where
\begin{equation}
J_u^\xi:=\{x\in J_u:\, \langle u^+(x)-u^-(x),\xi\rangle\neq0\}\,;
\label{jumpset}
\end{equation}
\end{enumerate}

\begin{defn}
A function $u \in GBD(\Omega)$ belongs to the subset $GSBD(\Omega)$ of special functions of bounded deformation if in addition for every $\xi \in S^{d-1}$ and $\mathcal H^{d-1}$-a.e.\ $y \in \Pi^\xi$, the function $u^{\xi,y}$ belongs to $SBV_{\mathrm{loc}}(\Omega^\xi_y)$.
\end{defn}

By \cite[Remark 4.5]{DM2013} one has the inclusions $BD(\Omega)\subset GBD(\Omega)$ and $SBD(\Omega)\subset GSBD(\Omega)$, which are in general strict. Some relevant properties of functions with bounded deformation can be generalized to this weak setting: in particular, in \cite[Theorem 6.2 and Theorem 9.1]{DM2013} it is shown that the jump set $J_u$ of a $GBD$-function is $\mathcal H^{d-1}$-rectifiable and that $GBD$-functions have an approximate symmetric differential $\mathcal{E}u(x)$ at $\mathcal L^{d}$-a.e.\ $x\in \Omega$, respectively. The space $GSBD^2(\Omega)$ is defined through:
$$
GSBD^2 (\Omega):= \{u \in GSBD(\Omega): \mathcal{E}u \in L^2 (\Omega; \mathbb R_{\mathrm{sym}}^{d\times d})\,,\,\mathcal H^{d-1}(J_u) < +\infty\}\,.
$$

Every function in $GSBD^2(\Omega)$ is approximated by bounded $SBV$ functions with more regular jump set, as stated by the following result (\cite[Theorem~1.1]{CC}). In order to deal with the Dirichlet boundary value problem (in fact we will impose a Dirichlet boundary datum $u_0 \in H^1(\R^d;\R^d)$ on a subset $\partial_D \Omega \subset \partial \Omega$), we report a version adapted for boundary data (cf.\ \cite[Section 5]{CC}).
 For technical reasons, we suppose that  $\dom=\dod\cup \don\cup N$ with $\dod$ and $\don$ relatively open, $\dod \cap \don =\emptyset$, $\hd(N)=0$, 
$\dod \neq \emptyset$,  $\partial(\dod)=\partial(\don)$, and 
that there exist a small $\overline \delta$ and $x_0\in \Rd$ such that for every $\delta \in (0,\overline \delta)$ 
\begin{equation}\label{0807170103}
O_{\delta,x_0}(\dod) \subset \Omega\,,
\end{equation}
where $O_{\delta,x_0}(x):=x_0+(1-\delta)(x-x_0)$.

  In the following, we denote by ${\rm tr}(u)$ the trace of $u$ on $\partial \Omega$ which is well defined for functions in $GSBD^2(\Omega)$ if $\Omega$ is Lipschitz (see \cite[Section~5]{DM2013}). 

\begin{thm}
Let $\Omega\subset\mathbb{R}^d$ be a bounded open Lipschitz set, and $u\in GSBD^2(\Omega;\R^d)$. Then there exists a sequence $u_n$ such that\\
$(i)$ $u_n\in SBV^2(\Omega;\R^d)\cap L^\infty(\Omega;\R^d)$;\\
$(ii)$ each $J_{u_n}$ is closed and included in a finite union of closed connected pieces of $C^1$-hypersurfaces;\\
$(iii)$ $u_n\in W^{1,\infty}({\Omega}\backslash J_{u_n};\R^d)$, and
\begin{align}
u_n\to u \mbox{ in measure on $\Omega$},\label{convme}\\
\mathcal{E}u_n \to \mathcal{E}u \mbox{ in $L^2(\Omega;\R^{d\times d}_{sym})$,}\label{convgrad}\\
\mathcal{H}^{d-1}(J_{u_n}\triangle J_u)\to0.\label{convjump}
\end{align}
\label{thm:density}
 Moreover, if $\partial_D \Omega\subset \partial \Omega$ satisfies \eqref{0807170103} and $u_0 \in H^1(\R^d;\R^d)$, then one can ensure that each $u_n$ satisfies $u_n=u_0$ in a neighborhood $U_n \subset \Omega$ of $\partial_D \Omega$, provided that \eqref{convjump} is replaced by
\begin{equation}\label{1412190922}
\lim_{n \to \infty} \mathcal{H}^{d-1}(J_{u_n}) = \mathcal{H}^{d-1}(J_u) + \mathcal{H}^{d-1}(\lbrace {\rm tr}(u) \neq {\rm tr}(u_0)\rbrace \cap \partial_D \Omega)\,.
\end{equation}
\end{thm} 
A further approximation result, by Cortesani and Toader \cite[Theorem~3.9]{CorToa99},  
allows us to approximate $GSBD^2(\Omega)$ functions with the so-called ``piecewise smooth''  $SBV$-functions, denoted $\mathcal{W}(\Omega;\Rd)$, characterized by the three properties 
\begin{equation}\label{1412191008}
\begin{cases}
u\in SBV(\Omega;\Rd)\cap W^{m,\infty}(\Omega\sm J_u;\Rd) \,\text{for every }m\in \N\,,\\
\hd(\overline{J}_u \sm J_u ) = 0\,,\\
\overline{J}_u \text{ is the intersection of $\Omega$ with a finite union of ${(d{-}1)}$-dimensional simplexes}\,.
\end{cases}
\end{equation}
As observed in \cite[Remark~4.3]{ChaCri19}, we may even approximate through functions $u$ such that, besides \eqref{1412191008}, also $\overline J_u \subset \Omega$ holds and the $(d{-}1)$-dimensional simplexes in the decomposition of $\overline J_u$ may be taken pairwise disjoint with $J_u \cap \Pi_i \cap \Pi_j=\emptyset$ for any two different hyperplanes $\Pi_i$, $\Pi_j$.
Furthermore, in the assumption under which \eqref{1412190922} holds true, we may also ensure that $u=u_0$ in a neighborhood of $\dom$.
We will employ these properties in Section~\ref{sec:upperbound}.

We recall the following general $GSBD^2$ compactness result from \cite{CC18Comp}.  In the following, when we deal with sets of finite perimeter, such as $A^\infty_u$, we identify the set with its subset of points with density 1, with respect to $d$-dimensional Lebesgue measure (cf.\ \cite[Definition~3.60]{AFP}), while we denote explicitly their essential boundary with the symbol $\partial^*$. 
\begin{thm}[$GSBD^2$ compactness]\label{th: GSDBcompactness}
 Let $\Omega \subset \R$ be an open, bounded set,  and let $(u_n)_n \subset  GSBD^2(\Omega)$ be a sequence satisfying
$$ \sup\nolimits_{n\in \N} \big( \Vert \mathcal{E}u_n \Vert_{L^2(\Omega)} + \mathcal{H}^{d-1}(J_{u_n})\big) < + \infty.$$
Then there exists a subsequence, still denoted by $u_n$, such that the set  $A^\infty_u := \lbrace x\in \Omega: \, |u_n(x)| \to +\infty \rbrace$ has finite perimeter, and  there exists  $u \in GSBD^2(\Omega)$ such that 
\begin{align}\label{eq: GSBD comp}
{\rm (i)} & \ \ u_n \to u \  \ \ \  \mbox{ in measure on $\Omega\sm A^\infty_u$}, \notag \\ 
{\rm (ii)} & \ \ \mathcal{E}u_n \rightharpoonup\mathcal{E}u \ \ \ \text{ in } L^2(\Omega \setminus A^\infty_u; \Mdd),\notag \\
{\rm (iii)} & \ \ \liminf_{n \to \infty} \mathcal{H}^{d-1}(J_{u_n}) \ge \mathcal{H}^{d-1}(J_u \cup  (\partial^*A^\infty_u \cap\Omega)  ).
\end{align}
\end{thm}

\paragraph*{\bf $GSBD^2_\infty$ functions.}

Inspired by the previous compactness result, in \cite{CriFri19}  a space of $GSBD^2$ functions which may also attain  a limit  value $\infty$ has been introduced, as we recall.  The space  $\bar{\R}^d := \R^d \cup \lbrace \infty \rbrace$ (with its sum  given by $a + \infty = \infty$ for any $a \in \bar{\R}^d$) is in a natural bijection with $ \mathbb{S}^d  =\lbrace \xi \in\R^{d+1}:\,|\xi| =1 \rbrace$ through the stereographic projection of $\mathbb{S}^{d}$ to $\bar{\R}^d$: for $\xi \neq e_{d+1}$,
$\phi(\xi) = \frac{1}{1-\xi_{d+1}}(\xi_1,\ldots,\xi_d),$
  $\phi(e_{d+1}) = \infty$. Let $\psi\colon \bar{\R}^d\to \mathbb{S}^{d}$  denote the inverse. Note that 
\begin{equation}\label{3005191230}
d_{\bar{\R}^d}(x,y):= |\psi(x) - \psi(y)|\quad \text{for }x,y \in \bar{\R}^d\end{equation} 
induces a bounded metric on $\bar{\R}^d$. Then
\begin{align}\label{eq: compact extension}
GSBD^2_\infty(\Omega) := \Big\{ &u \colon \Omega \to \bar{\R}^d \text{ measurable } \colon   A^\infty_u  := \lbrace u = \infty \rbrace \text{ satisfies } \mathcal{H}^{d-1}(\partial^* A^\infty_u)< +\infty, \notag \\
&  \ \ \ \ \ \ \ \ \ \  \ \ \ \ \ \widetilde{u}_t := u \chi_{\Omega \setminus A^\infty_u} + t \chi_{A^\infty_u} \in GSBD^2(\Omega) \ \text{ for all $t \in \R^d$} \Big\}\,. 
\end{align}
Symbolically, we will also write
$u = u \chi_{\Omega \setminus A^\infty_u} + \infty \chi_{A^\infty_u}\,.$
 Moreover, for any $u \in GSBD^2_\infty(\Omega)$
\begin{align}\label{eq: general jump}
\mathcal{E}u = 0 \text{ in } A^\infty_u \quad
 \text{and} \quad J_u = J_{u \chi_{\Omega \setminus A^\infty_u}} \cup (\partial^*A^\infty_u \cap \Omega)\,.
\end{align}
 In particular, 
\begin{align}\label{eq:same}
\mathcal{E}u = \mathcal{E}\widetilde{u}_t \ \ \text{$\mathcal{L}^d$-a.e.\ on  $\Omega$} \ \ \ \text{ and  } \ \ \ J_u = J_{\widetilde{u}_t} \ \ \text{$\mathcal{H}^{d-1}$-a.e.} \ \ \ \text{ for almost all $t \in \R$}\,,
\end{align} where $\widetilde{u}_t$ is the function from \eqref{eq: compact extension}.  Hereby, we also get a natural definition of a normal $\nu_u$ to the jump set $J_u$, and the slicing properties described for $GSBD^2$ still hold in $\Omega\sm A^\infty_u$.   Finally, we point out that all definitions are consistent with the usual ones if $u \in GSBD^2(\Omega)$; i.e.,\ if $A^\infty_u= \emptyset$.  Since $GSBD^2(\Omega)$ is a vector space, we observe that the sum of two functions in $GSBD^2_\infty(\Omega)$ lies again in this space.     
A  metric on $GSBD^2_\infty(\Omega)$  is given by
\begin{equation}\label{eq:metricd}
d(u,v) = \int_\Omega  d_{\bar{\R}^d}(u(x),v(x)) \, \dx\,,
\end{equation} 
 where $d_{\bar{\R}^d}$ is the distance in \eqref{3005191230}. In Sections~\ref{sec:compprelboun} and \ref{sec:lowerbound}, when we work in an extended domain $\wt\Omega$, we will still write $d(u,v)$ for  $\int_{\wt\Omega}  d_{\bar{\R}^d}(u(x),v(x)) \, \dx$. 
  We say that a sequence $(u_n)_n \subset GSBD^2_\infty(\Omega)$ \emph{converges weakly} to $u \in GSBD^2_\infty(\Omega)$ if 
\begin{align}\label{eq: weak gsbd convergence}
 \sup\nolimits_{n\in \N} \big( \Vert \mathcal{E}u_n \Vert_{L^2(\Omega)} + \mathcal{H}^{d-1}(J_{u_n})\big) < + \infty \ \ \ \text{and} \ \ \    d(u_n,u) \to 0 \text{ for } n \to \infty\,.
 \end{align}

\subsection{Some lemmas}

For $a<b$, we introduce the space $PC_\delta(a,b)$ of piecewise-constant functions on partitions of $(a,b)\subset \R$ with size $\delta$ \EEE; namely,
\begin{equation*}
\begin{split}
PC_\delta(a,b):=\biggl\{v:(a,b)\to\mathbb{R}:\,&\mbox{ there exists a partition $\{x_i\}_{i=0}^N$ of $[a,b]$ such that}\\
 & \quad |x_{i+1}-x_i|=\delta \mbox{ and } v(x)=v(x_i) \mbox{ on }[x_i,x_{i+1})\biggr\}.
\end{split}
\end{equation*}
For every $v\in PC_\delta(a,b)$, we denote by $\hat{v}$ the corresponding piecewise-affine interpolation on the nodes of the same partition, defined as
\begin{equation}
\hat{v}(x):=v(x_i)+\frac{v(x_{i+1})-v(x_i)}{x_{i+1}-x_i}(x-x_i),\quad x\in[x_i,x_{i+1})\,.
\label{pc-aff}
\end{equation}

\begin{lem}\label{Lemma1D}
Let $(v_\epsilon)_\epsilon$ be a sequence such that $v_\epsilon\in PC_\delta(a,b)$, $v_\epsilon\geq0$, and let $(\hat{v}_\epsilon)_\epsilon$ be the sequence of the corresponding piecewise-affine interpolations defined as in \eqref{pc-aff}. Assume that there exists $C>0$ such that
\begin{equation}
\frac{1}{\epsilon}\int_a^b(v_\epsilon(t)-1)^2\,\mathrm{d}t + \epsilon \int_a^b(\dot{\hat{v}}_\epsilon(t))^2\,\mathrm{d}t \leq C\,.
\label{Cbound}
\end{equation}
Then, setting
\begin{equation}
I:=\left\{s\in(a,b):\, \exists\, s_\epsilon\to s \mbox{ such that }\mathop{\lim\inf}_{\epsilon\to0}v_\epsilon(s_\epsilon)=0\right\}\,,
\label{setI}
\end{equation}
we have:\\
$(a)$ for every fixed constant $N_C>0$ depending only on $C$, it holds that
\begin{equation*}
\#I\leq N_C\,;
\end{equation*}
$(b)$ for every $A$ open such that $A\subset\!\subset\!(a,b)\backslash I$, there exists $\eta_A>0$ such that
\begin{equation*}
\mathop{\lim\inf}_{\epsilon\to0} \inf_{s\in A}v_\epsilon(s)\geq\eta_A\,.
\end{equation*}
\end{lem}

\proof

The assertion $(b)$ immediately follows from $(a)$. As for the proof of $(a)$, let us fix $N_C:=\lfloor 4C\rfloor$ and, arguing by contradiction, we assume that $\#I=N_C+1$ and $I=\{s^1,s^2,\dots, s^{N_C+1}\}$. For every such index $i$, we denote by $(s_\epsilon^i)_\epsilon$ the sequence defined by \eqref{setI} such that $s_\epsilon^i\to s^i$ and 
\begin{equation}
\displaystyle\mathop{\lim\inf}_{\epsilon\to0}v_\epsilon(s_\epsilon^i)=0\,.
\label{liminf1}
\end{equation}
Since by \eqref{Cbound} $v_\epsilon\to1$ a.e. in $(a,b)$, we can find a sequence $(t_\epsilon^i)_\epsilon$ such that
\begin{itemize}
\item[(i)] $s_\epsilon^i<t_\epsilon^i<s_\epsilon^{i+1}$\,;
\item[(ii)] $t_\epsilon^i\to s^i$\,;
\item[(iii)] $\displaystyle\mathop{\lim\inf}_{\epsilon\to0}v_\epsilon(t_\epsilon^i)=1$\,.
\end{itemize}
Moreover, we may assume that the subsequences of $s_\epsilon^i$ and $t_\epsilon^i$ realizing the liminf in \eqref{liminf1} and (iii), respectively, have infinite terms of the sequences of the indices in common. Now, let $\hat{s}^i_\epsilon$ and $\hat{t}^i_\epsilon$ be the greatest nodes of the partition that are less or equal than $s^i_\epsilon$ and $t^i_\epsilon$, respectively. Since $\delta\to0$ as $\epsilon\to0$, we have that $|s^i_\epsilon-\hat{s}^i_\epsilon|\to0$ and $|t^i_\epsilon-\hat{t}^i_\epsilon|\to0$, which, combined with the fact that $\hat{v}_\epsilon(\hat{s}^i_\epsilon)={v}_\epsilon(\hat{s}^i_\epsilon)$, $\hat{v}_\epsilon(\hat{t}^i_\epsilon)={v}_\epsilon(\hat{t}^i_\epsilon)$, with \eqref{liminf1} and (iii) give
\begin{align*}
\mathop{\lim\inf}_{\epsilon\to0} \hat{v}_\epsilon(\hat{s}^i_\epsilon)=0\,,\\
\mathop{\lim\inf}_{\epsilon\to0} \hat{v}_\epsilon(\hat{t}^i_\epsilon)=1\,.
\end{align*}

Now, for every $i$ and $\epsilon$, let $\tilde{t}_\epsilon^i$ be the first node of the partition such that $\tilde{t}_\epsilon^i\geq \hat{s}_\epsilon^i$ and $v_\epsilon(\tilde{t}_\epsilon^i)\geq\frac{1}{2}$, and let $\tau_\epsilon^i$ be the first point in $(\hat{s}_\epsilon^i,\tilde{t}_\epsilon^i)$ such that $\hat{v}_\epsilon(\tau_\epsilon^i)=\frac{1}{2}$, whose existence is ensured by the Mean Value Theorem. We then have
\begin{equation*}
v_\epsilon(s)\leq\frac{1}{2},\quad \forall s\in(\hat{s}^i_\epsilon,\tau_\epsilon^i)\,,
\end{equation*}
whence
\begin{equation}
|v_\epsilon(s)-1|\geq\frac{1}{2},\quad \forall s\in(\hat{s}^i_\epsilon,\tau_\epsilon^i)\,.
\label{geq}
\end{equation}
Now, by Young's inequality and \eqref{Cbound},
\begin{equation*}
\begin{split}
C&\geq\mathop{\lim\inf}_{\epsilon\to0}\sum_{i=1}^{N_C+1} \int_{\hat{s}^i_\epsilon}^{\tau_\epsilon^i} |v_\epsilon(t)-1||\dot{\hat{v}}_\epsilon(t)|\,\mathrm{d}t\geq \sum_{i=1}^{N_C+1}\mathop{\lim\inf}_{\epsilon\to0}\int_{\hat{s}^i_\epsilon}^{\tau_\epsilon^i} |v_\epsilon(t)-1||\dot{\hat{v}}_\epsilon(t)|\,\mathrm{d}t\\
&\stackrel{\eqref{geq}}{\geq} \frac{1}{2}  \sum_{i=1}^{N_C+1}\mathop{\lim\inf}_{\epsilon\to0}\int_{\hat{s}^i_\epsilon}^{\tilde{t}_\epsilon^i} |\dot{\hat{v}}_\epsilon(t)|\,\mathrm{d}t \geq \frac{1}{4}(N_C+1)\,,
\end{split}
\end{equation*}
which gives a contradiction.
\endproof

\begin{lem}\label{lemma2}
Let $\Omega'\subset\!\subset\!\Omega$ and $(f_\epsilon)_\epsilon$, $(g_\epsilon)_\epsilon$ be sequences of real-valued measurable functions such that
\begin{enumerate}
\item[(i)] $f_\epsilon\to f$ a.e. in $\Omega$,\, $0\leq f_\epsilon(x)\leq M$;
\item[(ii)] $g_\epsilon\rightharpoonup g$ in $L^2(\Omega')$,
\end{enumerate}
for some measurable $f$ and $g$. Then,
\begin{equation*}
f_\epsilon g_\epsilon \rightharpoonup fg \mbox{ in $L^2(\Omega')$}\,.
\end{equation*}
In particular,
\begin{equation*}
\mathop{\lim\inf}_{\epsilon\to0}\int_\Omega (f_\epsilon g_\epsilon)^2(x)\mathrm{d} x\geq \int_{\Omega'} (fg)^2(x)\mathrm{d} x\,.
\end{equation*}
\end{lem}

For $\delta>0$, and for \EEE any measurable function $u:\Omega\subset\R^d\to\R^d$ and $y\in\R^d\backslash\{0\}$ we define the translations
\begin{equation}
T_y^{\delta}u(x):=u\left(\delta y +\delta\left\lfloor\frac{x}{\delta}\right\rfloor\right)\,,
\label{transldiscr}
\end{equation}
where $\lfloor z \rfloor:=\sum_{i=1}^d\left\lfloor{\langle z,e_i\rangle}\right\rfloor e_i$ and, for every $t\in\R$, $\lfloor t \rfloor$ denotes the integer part of $t$. We have that $T_y^{\delta}u$ is constant on each $d$-cube $\alpha+\delta (0,1]^d$, $\alpha\in\delta\Z^d$. Moreover, the following result holds (see, e.g., \cite[Lemma~2.11]{AFG}).
\begin{lem}\label{lemmatransl}
Let $u_\delta\to u$ in $L^1(\Omega;\R^d)$ as $\delta\to0$. Then for every $\Omega'\!\subset\!\subset \Omega$ it holds
\begin{enumerate}
\item[{\rm(i)}]
\begin{equation}
\lim_{\delta\to0}\int_{[0,1]^d}\|T_y^{\delta} u_\delta - u\|_{L^1(\Omega',\R^d)}\,\mathrm{d}y=0\,;
\label{transconverg}
\end{equation}
\item[{\rm (ii)}] if $C_\delta\subset [0,1]^d$ is a family of sets such that $\displaystyle\mathop{\lim\inf}_{\delta\to0}|C_\delta|>0$, then there exists a sequence $y_\delta\in C_\delta$ such that $T_{y_\delta}^{\delta}u_\delta\to u$ in $L^1(\Omega';\R^d)$.
\end{enumerate}
\end{lem}
Let $d\in\{2,3\}$. We set 
\begin{equation}
\begin{split}
S_d:=\{e_i:\, i=1,\dots,d\}&\cup\{e_i+e_j,\, e_i-e_j:\, 1\leq i<j\leq d\}\\
&\cup\{\{e_i\pm e_j\pm e_k\}:\, 1\leq i<j<k\leq d\}\,,
\end{split}
\label{setdirec}
\end{equation}
and consider a kernel function $\sigma:\Z^d\to[0,+\infty)$ such that 
\begin{subequations}\label{eqs:0912191317}
\begin{equation}\label{0912191318}
\sigma(\xi)=\sigma(|\xi|)
\end{equation}
and $\sigma(\xi)\neq0$ for every $\xi\in S_d$; we will often use the shortcut \begin{equation}\label{0912191319}
\sigma_r:=\sigma(r) \text{ when }|\xi|=r\,.
\end{equation}

\begin{lem}
Let $d\in\{2,3\}$ and $M$ be a $d\times d$ symmetric matrix.
Then, defining $S_d$ and $\sigma$ as before, it holds that
\begin{equation}
\sum_{\xi\in S_d}\frac{\sigma_{|\xi|}}{|\xi|^4}|\langle M\xi,\xi\rangle|^2=c_{1,\sigma,d}\sum_{i=1}^d M_{ii}^2 + 2c_{2,\sigma,d}\sum_{1\leq i<j\leq d}M_{ij}^2+c_{3,\sigma,d}\left(\sum_{i=1}^d M_{ii}\right)^2\,,
\label{matrix1}
\end{equation}
where
\begin{equation*}
\begin{split}
&c_{1,\sigma,d}:=\left(\sigma_1+\frac{\sigma_{\sqrt{2}}}{2}(d-2)\right), \quad c_{2,\sigma,d}:=\left(\sigma_{\sqrt{2}}+\frac{8\sigma_{\sqrt{3}}}{9}(d-1)(d-2)\right),\\
&c_{3,\sigma,d}:=\left(\frac{\sigma_{\sqrt{2}}}{2}+\frac{4\sigma_{\sqrt{3}}}{9}(d-1)(d-2)\right).
\end{split}
\end{equation*}
\end{lem}
\proof
We can rewrite the sum on left hand side of \eqref{matrix1} as (recall that $\{e_1, \dots, e_d\}$ denote the canonical basis of $\R^d$)
\begin{equation*}
\begin{split}
&\sigma_1\sum_{i=1}^d M_{ii}^2 + \frac{\sigma_{\sqrt{2}}}{4} \sum_{1\leq i<j\leq d}|\langle M(e_i\pm e_j),e_i\pm e_j\rangle|^2 \\
&+ \frac{\sigma_{\sqrt{3}}}{9} \sum_{1\leq i<j<k\leq d}|\langle M(e_i\pm e_j\pm e_k),e_i\pm e_j\pm e_k\rangle|^2\\
&= \sigma_1 \sum_{i=1}^d M_{ii}^2 + \frac{\sigma_{\sqrt{2}}}{4} \sum_{1\leq i<j\leq d}(M_{ii}+M_{jj}\pm2M_{ij})^2 \\
&+ \frac{\sigma_{\sqrt{3}}}{9} \sum_{1\leq i<j<k\leq d}(M_{ii}+M_{jj}+M_{kk}\pm2M_{ij}\pm 2M_{ik}\pm 2M_{jk})^2 \\
&= \sigma_1 \sum_{i=1}^d M_{ii}^2 + 2\sigma_{\sqrt{2}}\sum_{1\leq i<j\leq d}M_{ij}^2+\frac{\sigma_{\sqrt{2}}(d-2)}{2}\sum_{i=1}^d M_{ii}^2 +\frac{\sigma_{\sqrt{2}}}{2}\left(\sum_{i=1}^d M_{ii}\right)^2\\
&+\frac{16\sigma_{\sqrt{3}}}{9}(d-1)(d-2)\sum_{1\leq i<j\leq d}M_{ij}^2 +\frac{4\sigma_{\sqrt{3}}}{9}(d-1)(d-2)\left(\sum_{i=1}^d M_{ii}\right)^2,
\end{split}
\end{equation*}
which coincides with the right hand side of \eqref{matrix1}. 
\endproof

\begin{oss}
Notice that setting $\tilde{c}_{\sigma,d}:=\min\{c_{1,\sigma,d},c_{2,\sigma,d}\}$, from \eqref{matrix1} we may deduce the bound
\begin{equation}
\tilde{c}_{\sigma,d}|M|^2\leq \sum_{\xi\in S_d}\frac{\sigma_{|\xi|}}{|\xi|^4}|\langle M\xi,\xi\rangle|^2\,.
\label{matrix2}
\end{equation}
Moreover, choosing in \eqref{matrix1}
\begin{equation}\label{0912191307}
\begin{cases}
\sigma_1=\sigma_{\sqrt{2}}=1 & \mbox{ if $d=2$\,,}\\
\sigma_1=\frac{3}{4}, \sigma_{\sqrt{2}}=\frac{1}{2}, \sigma_{\sqrt{3}}=\frac{9}{32}, & \mbox{ if $d=3$\,,}
\end{cases}
\end{equation}
we obtain the identity
\begin{equation*}
\sum_{\xi\in S_d}\frac{\sigma_{|\xi|}}{|\xi|^4}|\langle M\xi,\xi\rangle|^2 = |M|^2 + \frac{1}{2}  |\mathrm{Tr}M|^2\,.
\end{equation*}
\label{rem:summation}
\end{oss}
\end{subequations}

\section{Discrete models and approximation results}\label{sec:discrmodel}

Let $d\in\{2,3\}$, $\Omega\subset \Rd$ an open, bounded, Lipschitz set, with $\dom$ satisfying \eqref{0807170103} and the related assumptions, and let $u_0\in H^1(\Rd;\Rd)$. For any $\delta>0$, we consider the scaled lattice $\delta\mathbb{Z}^d$ and set $\Omega_\delta:=\Omega\cap\delta\Z^d$. 
 We introduce  suitable discretizations for both the symmetrized gradient and the divergence. For $\xi\in\mathbb{R}^d\backslash\{0\}$, $\delta>0$, and $u:\Omega\to\mathbb{R}^d$ measurable we define
\begin{equation}
\begin{split}
D_\delta^\xi u(x)&:= \left\langle u(x+\delta\xi)-u(x),\frac{\xi}{|\xi|^2}\right\rangle\,,\\
|D_{\delta,\xi}u(x)|^2 &:= |D_\delta^\xi u(x)|^2 + |D_\delta^{-\xi} u(x)|^2\,.
\end{split}
\label{diffquot1}
\end{equation}
For a scalar function $v:\Omega\to\mathbb{R}$, we will often adopt the notation 
\begin{equation}
\Delta_\delta^\xi v(x):= v(x+\delta\xi)-v(x)\,.
\label{diffquot2}
\end{equation}
Moreover, for any $\{\psi_1,\dots,\psi_d\}$ orthogonal basis of $\R^d$, we set
\begin{equation}
{\rm div}_\delta^{\psi_1,\dots,\psi_d} u(x):=\sum_{i=1}^d D_\delta^{\psi_i}u(x)\,.
\label{discrediv1}
\end{equation}
Then
we define 
\begin{equation}
\begin{split}
|{\rm Div}_{\delta}\,u(x)|^2:= \sum_{(k_1,\dots,k_d)\in\{-1,1\}^d}|{\rm div}_\delta^{k_1e_1,k_2e_2,\dots,k_de_d} u(x)|^2\,.
\end{split}
\label{discrediv2}
\end{equation}
In order to impose a non-interpenetration constraint in the limit fracture energy, we treat differently in the approximation the positive and negative part of the discrete divergence.
 We set, for $u\colon \Omega\to \Rd$ measurable,
\begin{subequations}\label{eqs:discredivpm}
\begin{equation}\label{discredivpm}
({\rm div}_\delta^\pm)^{\psi_1,\dots,\psi_d} u(x):=\Big(\sum_{i=1}^d D_\delta^{\psi_i}u(x)\Big)^\pm\,,
\end{equation}
\begin{equation}\label{discrediv2pm}
\begin{split}
|{\rm Div}_{\delta}^\pm u(x)|^2:= \sum_{(k_1,\dots,k_d)\in\{-1,1\}^d}|({\rm div}_\delta^{\pm})^{k_1e_1,k_2e_2,\dots,k_de_d} u(x)|^2\,,
\end{split}
\end{equation}
\end{subequations}

For $u\colon \Omega\to \Rd$, $v\colon \Omega\to \R$ measurable,  
$\xi\in\mathbb{Z}^d\backslash\{0\}$, $\sigma_{|\xi|}$ fixed from \eqref{eqs:0912191317},  we consider the functionals $F_\epsilon^\xi$, $F_\varepsilon$, ${F}^{\rm div}_\epsilon$ 
defined as 
\begin{subequations}\label{eqs:energiesF}
\begin{equation}
F_\epsilon^\xi(u,v):=\frac{1}{2}\sum_{\alpha\in R_\delta^\xi(\Omega)}\delta^{d-2}(v(\alpha))^2\left|D_{\delta,\xi}u(\alpha)\right|^2\,, \quad F_\epsilon(u,v):=\sum_{\xi\in S_d}\sigma_{|\xi|}F_\epsilon^\xi(u,v)\,,
\label{energiesF}
\end{equation}
\begin{equation}
{F}^{\rm div}_\epsilon(u,v):=\frac{1}{2^d}\sum_{\alpha\in R_\delta^{{\rm div}}(\Omega)}\delta^{d-2}(v(\alpha))^2\left|{\rm Div}_{\delta}u(\alpha)\right|^2\,,
\label{energiesF2}
\end{equation}
\end{subequations}
where 
\begin{equation}
R_\delta^\xi(\Omega):= \biggl\{\alpha\in\delta\mathbb{Z}^d:\,  [\alpha-\delta\xi,\alpha+\delta\xi]\subset\Omega\biggr\},\quad R_\delta^{{\rm div}}(\Omega):=\bigcap_{i=1}^dR_\delta^{e_i}(\Omega)\,,
\label{range}
\end{equation}
and ${F}^{{\rm div}^+}_\epsilon$, ${F}^{{\rm div}^-}_\epsilon$, $F_\varepsilon^{\mathrm{div},\mathrm{NI}}$ given by
\begin{subequations}\label{eqs:energiesF2pm}
\begin{equation}\label{energiesF2p}
{F}^{{\rm div}^+}_\epsilon(u,v):=\frac{1}{2^d}\hspace{-1em}\sum_{\alpha\in R_\delta^{{\rm div}}(\Omega)}\hspace{-1em}\delta^{d-2}(v(\alpha))^2\left|{\rm Div}_{\delta}^+ u(\alpha)\right|^2\,,
\end{equation}
\begin{equation}\label{energiesF2m}
{F}^{{\rm div}^-}_\epsilon(u):=\frac{1}{2^d} \hspace{-1em}\sum_{\alpha\in R_\delta^{{\rm div}}(\Omega)}\hspace{-1em}\delta^{d-2}\left|{\rm Div}_{\delta}^- u(\alpha)\right|^2\hspace{-0.5em}\,,\quad F_\varepsilon^{\mathrm{div},\mathrm{NI}}(u,v):= {F}^{{\rm div}^+}_\epsilon(u,v) + {F}^{{\rm div}^-}_\epsilon(u)\,.
\end{equation}
\end{subequations}
Notice that ${F}^{{\rm div}^-}_\epsilon$ does not include any contribution in $v$. Moreover, we introduce the discrete Modica-Mortola-type functional  
\begin{equation}
G_\epsilon(v):=\frac12\sum_{\alpha\in \Omega_\delta}\delta^d\left(\frac{1}{\epsilon}(v(\alpha)-1)^2 + \epsilon\sum_{\substack{k=1\\\alpha+\delta e_k\in \Omega_\delta}}^d\left(\frac{v(\alpha+\delta e_k)-v(\alpha)}{\delta}\right)^2\right)\,.
\label{energiesG}
\end{equation}
It will be useful to introduce also a localized version of the functionals defined above. For every $A\subset\Omega$ open bounded set, the symbols $F_\epsilon^\xi(u,v,A)$, ${F}^{\rm div}_\epsilon(u,v,A)$ and $G_\epsilon(v,A)$ denote the energies as in \eqref{energiesF}, \eqref{energiesF2} and \eqref{energiesG}, respectively, where the sums are restricted to $\alpha\in R_\delta^\xi(A)$ defined as in \eqref{range} with $A$ in place of $\Omega$.

For $\lambda$, $\theta >0$,  let $E_\epsilon\ltm$ and $(E_{\lambda, \theta}^{\mathrm{NI}})_\varepsilon$ be defined on $ L^1(\Omega;\mathbb{R}^d)\times L^1(\Omega;\mathbb{R})$
by
\begin{equation*}
E_\epsilon\ltm(u,v):=
 \lambda \, F_\epsilon(u,v)+\theta \, {F}^{\rm div}_\epsilon(u,v)+ G_\epsilon(v)
\label{eq:mainenergies}
\end{equation*}
and 
\begin{equation*}
(E_{\lambda, \theta}^{\mathrm{NI}})_\varepsilon(u,v):=
\lambda\, F_\epsilon(u,v)+\theta\,{F}^{\mathrm{div}, \mathrm{NI}}_\epsilon(u,v) +G_\epsilon(v)\,, 
\end{equation*}

Let us define the class of vector-valued piecewise constant functions on $\Omega$
\begin{equation*}
\mathcal{A}_\delta(\Omega;\mathbb{R}^d):=\biggl\{u:\Omega\to\mathbb{R}^d:\, u(x)\equiv u(\alpha) \mbox{ for every }x\in(\alpha+[0,\delta)^d)\cap\Omega\mbox{ for any }\alpha\in\delta\mathbb{Z}^d\biggr\}\,,
\end{equation*}
and, analogously, the class of real-valued piecewise constant functions $\mathcal{A}_\delta(\Omega;\mathbb{R})$; in order to deal with the Dirichlet boundary value problem, we set
\begin{equation*}
\begin{split}
\mathcal{A}_\delta^{\mathrm{Dir}}(\Omega;\mathbb{R}^d):=\biggl\{\mathcal{A}_\delta(\Omega;\mathbb{R}^d) \colon & u\equiv u_0(\alpha) \mbox{ in }\alpha+[0,\delta)^d \cap\Omega\mbox{ for any }\alpha\in\delta\mathbb{Z}^d\\&\text{ such that }(\alpha+[0,\delta)^d)\cap\dod \neq \emptyset\biggr\}
\end{split}
\end{equation*}
and $\mathcal{A}_\delta^{\mathrm{Dir}}(\Omega;\mathbb{R})$ for real-valued functions, with $u_0$ replaced by the constant function 1.

We introduce the energy functionals $(E_{\lambda,\theta}^{\mathrm{Dir}})_\varepsilon$ and, for every $M>0$, $(E_{\lambda, \theta}\NIM)_\epsilon$ defined for $u$ and $v$ measurable by
\begin{equation*}
(E_{\lambda,\theta}^{\mathrm{Dir}})_\varepsilon(u,v):=
\begin{cases}
\displaystyle E_\varepsilon\ltm(u,v), & \mbox{ if }(u,v)\in\mathcal{A}_\delta^\mathrm{Dir}(\Omega;\mathbb{R}^d)\times\mathcal{A}_\delta^\mathrm{Dir}(\Omega;\mathbb{R})\,,\\
+\infty,&\mbox{ otherwise, }
\end{cases}
\label{eq:mainenergies'}
\end{equation*}
and
\begin{equation}
(E_{\lambda, \theta}\NIM)_\epsilon(u,v){:=}
\begin{cases}
(E_{\lambda, \theta}^{\mathrm{NI}})_\varepsilon(u,v), & \mbox{ if }(u,v)\in\mathcal{A}_\delta(\Omega;\mathbb{R}^d){\times}\mathcal{A}_\delta(\Omega;\mathbb{R})\text{ and }\|u\|_{L^\infty}\leq M\,,\\
+\infty,&\mbox{ otherwise. }
\end{cases}
\label{eq:mainenergiesNI}
\end{equation}

For fixed $\lambda$, $\theta>0$ we consider the Griffith functional $\mathcal{G}\ltms$ defined on $GSBD^2_\infty(\Omega)$ (recall \eqref{eq: general jump}) by
\begin{equation*}
\mathcal{G}\ltms(u):=\lambda\int_\Omega |\mathcal{E}u(x)|^2\,\mathrm{d}x + \left(\frac{\lambda}{2}+\theta\right)\int_\Omega |{\rm div}\,u(x)|^2\,\mathrm{d}x  + \mathcal{H}^{d-1}(J_u\cap \Omega)\,,
\end{equation*}
and its Dirichlet version
\begin{equation*}
\begin{split}
\mathcal{G}\ltms^{\mathrm{Dir}}(u):=\lambda& \int_\Omega |\mathcal{E}u(x)|^2\,\mathrm{d}x + \left(\frac{\lambda}{2}+\theta\right)\int_\Omega |{\rm div}\,u(x)|^2\,\mathrm{d}x  \\& + \mathcal{H}^{d-1}\big((J_{u}\cap \Omega) \cup (\{\mathrm{tr}(u) \neq \mathrm{tr}(u_0)\}\cap \dod)\big)\,.
\end{split}
\end{equation*}
Notice that a more compact expression of the jump part is obtained by considering a set $\wt \Omega \supset \Omega$ with
\begin{equation}\label{eq:tildeOmega}
\wt \Omega \cap \dom =\dod,
\end{equation}
and by extending $u$ to a function $u' \in GSBD^2_\infty(\wt \Omega)$ defined as \begin{equation}\label{eq:u'}
u'=
\begin{dcases}
u \quad & \text{in }\Omega,\\
u_0 \quad & \text{in }\wt \Omega\sm \Omega:
\end{dcases}
\end{equation} 
then 
\begin{equation}\label{1112192224}
J_{u'}=(J_{u}\cap \Omega) \cup (\{\mathrm{tr}(u) \neq \mathrm{tr}(u_0)\}\cap \dod)\,.
\end{equation}
We also set
\begin{equation*}\label{parGriff}
{\widetilde{\mathcal{G}}}\ltms^\mathrm{Dir}(u,v):= 
\begin{dcases}
{\mathcal{G}}\ltms^\mathrm{Dir}(u),\quad & \text{if }u\in GSBD^2(\Omega) \text{ and }v=1 \text{ a.e.\ in }\Omega,\\
+\infty \quad & \text{otherwise}
\end{dcases}
\end{equation*}
and, for every $M>0$,
\begin{equation*}\label{parGriffgenNI}
{\mathcal{G}}_{\lambda,\theta}\NIM(u,v):= 
\begin{dcases}
{\mathcal{G}}_{\lambda,\theta}(u), &\text{if } [u]\cdot \nu \geq 0\ \hd\text{-a.e.\ on }J_u,\, \|u\|_{L^\infty}\leq M,\,v=1 \text{ a.e.\ in } \Omega,\\
+\infty,  & \text{otherwise.}
\end{dcases}
\end{equation*}
Notice that ${\mathcal{G}}\ltms^\mathrm{Dir}(u)={\widetilde{\mathcal{G}}}\ltms^\mathrm{Dir}(\widetilde{u}_t,1)$ for $\Ld$-a.e.\ $t\in \Rd$, by \eqref{eq:same}.
Moreover, ${\mathcal{G}}_{\lambda,\theta}\NIM$ displays a non-interpenetration constraint, not present in ${{\mathcal{G}}}\ltms^\mathrm{Dir}$. We define it directly accounting for an $L^\infty$ bound for $|u|$ at level $M$, for technical reasons. Finally, we do not take into account the role of boundary conditions for the functional with non-interpenetration constraint, since we employ results from \cite{CCF18ARMA} (cf.\ Lemma~\ref{le:CCF18}), where the boundary value problem was not explicitly addressed.

We are now ready to state the main results of the paper. In the following we assume that $u_0$, $\lambda$, $\theta$ are fixed and that $\lim_{\epsilon\to 0} \frac{\delta}{\epsilon}=0$.
\begin{thm}\label{teo:main}
Under the assumptions above, it holds that:
\begin{itemize}
\item[(i)] as $\varepsilon \to 0$, $(E_{\lambda,\theta}^{\mathrm{Dir}})_\varepsilon$ $\Gamma$-converges with respect to the topology of the convergence in measure to ${\widetilde{\mathcal{G}}}\ltms^\mathrm{Dir}$;
\item[(ii)] for $((u_\epsilon, v_\epsilon))_\varepsilon$  such that $\sup_\epsilon (E_{\lambda,\theta}^{\mathrm{Dir}})_\varepsilon(u_\epsilon, v_\epsilon) <+\infty$, there exists $u \in GSBD^2_\infty(\Omega)$ such that $d(u_\varepsilon,u)\to 0$, $v_\varepsilon \to 1$, 
and
\begin{equation}\label{sciGriffith}
{\mathcal{G}}\ltms^\mathrm{Dir}(u,v) \leq \liminf_{\epsilon\to 0}(E_{\lambda,\theta}^{\mathrm{Dir}})_\varepsilon\EEE(u_\epsilon, v_\epsilon)\,.
\end{equation}
\end{itemize}
\end{thm}

We remark that any sequence of minimizers $(u_\varepsilon, v_\varepsilon)_\varepsilon$ for $(E_{\lambda,\theta}^{\mathrm{Dir}})_\varepsilon$ satisfies, up to a subsequence, $d(u_\varepsilon, u)\to 0$, for $u \in GSBD^2_\infty(\Omega)$ such that any $(\wt u_t,1)$ minimizes $\widetilde{\mathcal{G}}\ltms^\mathrm{Dir}$ (recall $\wt u_t=u \chi_{\Omega \setminus A^\infty_u} + t \chi_{A^\infty_u}$). In particular, $u_\varepsilon$ converges to $u$ a.e.\ in $\Omega\sm A^\infty_u$ and the bulk energies of $u_\varepsilon$ vanish in $A^\infty_u$ (cf.\ \cite[Theorem~5.8]{CC}). 

\begin{thm}\label{thm:noninter}
Under the assumptions above, for every $M >0$ it holds that:
\begin{itemize}
\item[(i)] ${\mathcal{G}}_{\lambda,\theta}\NIM \leq \Gamma\text{-}\liminf_{\varepsilon\to 0} (E_{\lambda, \theta}\NIM)_\epsilon$;
\item[(ii)] every $((u_\varepsilon, v_\varepsilon))_\varepsilon$ such that $\sup_{\varepsilon}(E_{\lambda, \theta}\NIM)_\epsilon(u_\varepsilon, v_\varepsilon)<+\infty$ converges, up to a subsequence, in $L^1(\Omega;\Rd) \times L^1(\Omega)$ to $(u,1)$ for $u \in SBD^2(\Omega)$;
\item[(iii)] if $d=2$ and $\lim_{\varepsilon\to 0}\frac{\delta}{\varepsilon^2}=0$, then ${\mathcal{G}}_{\lambda,\theta}\NIM \geq \Gamma\text{-}\limsup_{\varepsilon\to 0} (E_{\lambda, \theta}\NIM)_\epsilon$,
\end{itemize}
where the $\Gamma$-$\liminf$ and $\Gamma$-$\limsup$ above are with respect to the strong $L^1(\Omega;\Rd) \times L^1(\Omega)$ topology.
\end{thm}

In Sections~\ref{sec:compprelboun} and \ref{sec:lowerbound} we actually work in the enlarged configuration $\wt \Omega \subset \Rd$ satisfying \eqref{eq:tildeOmega} and with functions $u_\varepsilon$, $v_\varepsilon$ in $\mathcal{A}_\delta^{\mathrm{Dir}}(\wt\Omega;\mathbb{R}^d)$, $\mathcal{A}_\delta^{\mathrm{Dir}}(\wt\Omega;\mathbb{R})$, where
\begin{equation*}
\begin{split}
\mathcal{A}_\delta^{\mathrm{Dir}}(\wt\Omega;\mathbb{R}^d):=\biggl\{\mathcal{A}_\delta(\Omega;\mathbb{R}^d) \colon & u\equiv u_0(\alpha) \mbox{ in }\alpha+[0,\delta)^d \cap\Omega\mbox{ for any }\alpha\in\delta\mathbb{Z}^d\\&\text{ such that }(\alpha+[0,\delta)^d)\cap\wt\Omega\sm \Omega \neq \emptyset\biggr\}
\end{split}
\end{equation*}
and $\mathcal{A}_\delta^{\mathrm{Dir}}(\wt\Omega;\mathbb{R})$ is defined similarly, for $1$ in place of $u_0$ in $\wt\Omega\sm \Omega$.
In particular, if $u_\varepsilon \to \bar{u}$ in $GSBD^2_\infty(\wt\Omega)$ for some $\bar{u}$, then $\bar{u}=u_0$ in $\wt \Omega\sm \Omega$. 
Let us also fix once and for all $\lambda$, $\theta>0$.

\section{Compactness}\label{sec:compprelboun}

In this section we prove a compactness result (Proposition~\ref{prop:compactness}) for the discrete approximations of the Griffith energy, that holds under the assumption that $\frac{\delta}{\epsilon}$ be bounded. We show that sequences $(u_\epsilon,v_\epsilon)_\epsilon$ with equibounded energy $E_\varepsilon\ltm$ are approximated, in the sense of the convergence in measure, by sequences with bounded continuous Griffith energy (for which compactness is known from Theorem~\ref{th: GSDBcompactness}).

\begin{prop}\label{prop:compactness}
Let $\frac{\delta}{\epsilon}$ be bounded as $\epsilon\to0$. Let $(u_\epsilon,v_\epsilon)_\epsilon\subset L^1(\wt\Omega;\mathbb{R}^d)\times L^1(\wt\Omega;\mathbb{R})$ be such that $u_\epsilon\in\mathcal{A}_\delta^{\mathrm{Dir}}(\wt\Omega;\mathbb{R}^d)$, $v_\epsilon\in\mathcal{A}_\delta\Dir(\Omega;\mathbb{R})$ with
\begin{equation}
\sup_\epsilon E_\epsilon\ltm(u_\epsilon,v_\epsilon)<+\infty\,. 
\label{equiboundcomp}
\end{equation}
Then there exist functions $\overline{u}_\varepsilon \in SBD^2(\wt\Omega;\R^d)$ such that
\begin{equation}\label{2210191823}
u_\varepsilon-\overline{u}_\varepsilon \to 0 \quad\mathcal{L}^d\text{-a.e.\ in }\wt\Omega
\end{equation} 
and
\begin{equation}\label{compactbound}
\sup_{\epsilon>0} \Big\{\int_{\wt\Omega} |\mathcal{E}\bar{u}_\epsilon(x)|^2\,\mathrm{d}x + \mathcal{H}^{d-1}(J_{\bar{u}_\epsilon})\Big\}<+\infty\,.
\end{equation}
Moreover, if $\|u_\varepsilon\|_{L^\infty}\leq M$, then $\|\bar{u}_\varepsilon\|_{L^\infty}\leq M$. 
\end{prop}

\proof

We introduce a suitable triangulation $\mathcal{T}^d_\epsilon$ of $\wt\Omega$, based on the Freudenthal partition $\Sigma_d$ of the $d$-cube (see Fig.~\ref{fig:decomposition}). 
\begin{figure}[htbp]
\centering
\def\svgwidth{150pt}
\begingroup%
  \makeatletter%
  \providecommand\color[2][]{%
    \errmessage{(Inkscape) Color is used for the text in Inkscape, but the package 'color.sty' is not loaded}%
    \renewcommand\color[2][]{}%
  }%
  \providecommand\transparent[1]{%
    \errmessage{(Inkscape) Transparency is used (non-zero) for the text in Inkscape, but the package 'transparent.sty' is not loaded}%
    \renewcommand\transparent[1]{}%
  }%
  \providecommand\rotatebox[2]{#2}%
  \newcommand*\fsize{\dimexpr\f@size pt\relax}%
  \newcommand*\lineheight[1]{\fontsize{\fsize}{#1\fsize}\selectfont}%
  \ifx\svgwidth\undefined%
    \setlength{\unitlength}{192.75bp}%
    \ifx\svgscale\undefined%
      \relax%
    \else%
      \setlength{\unitlength}{\unitlength * \real{\svgscale}}%
    \fi%
  \else%
    \setlength{\unitlength}{\svgwidth}%
  \fi%
  \global\let\svgwidth\undefined%
  \global\let\svgscale\undefined%
  \makeatother%
  \begin{picture}(1,1.08560311)%
    \lineheight{1}%
    \setlength\tabcolsep{0pt}%
    \put(0,0){\includegraphics[width=\unitlength,page=1]{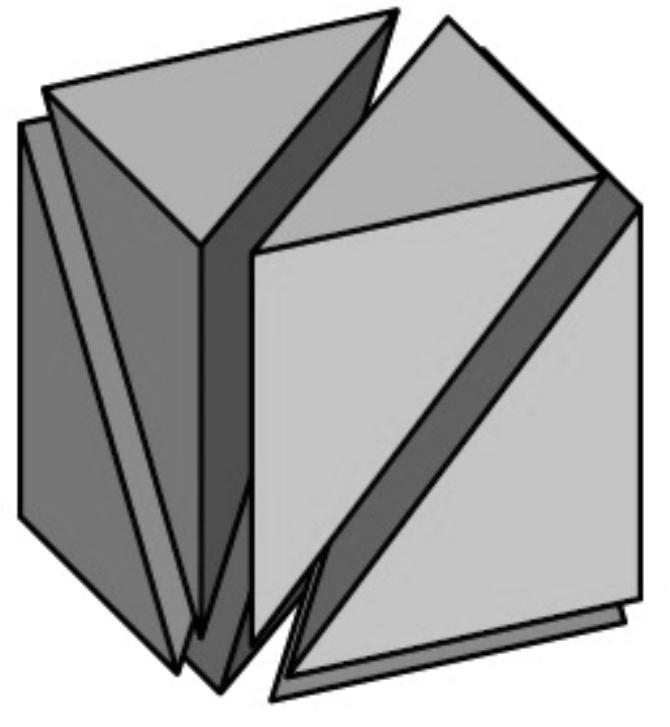}}%
    \put(0.30200091,-0.02840008){\color[rgb]{0,0,0}\makebox(0,0)[lt]{\lineheight{1.25}\smash{\begin{tabular}[t]{l}$\small(0,0,0)$\end{tabular}}}}%
    \put(0.58420168,1.08480314){\color[rgb]{0,0,0}\makebox(0,0)[lt]{\lineheight{1.25}\smash{\begin{tabular}[t]{l}$\small(1,1,1)$\end{tabular}}}}%
  \end{picture}%
\endgroup%

\caption{The Freudenthal decomposition $\Sigma_3$.}\label{fig:decomposition}
\end{figure}
It is defined as the set of all $d$-simplexes $T$ obtained through minimal chains of ordered vertices connecting the origin to the vertex $(1,1,\dots,1)$. They are $d!$ congruent simplexes and each has volume $1/d!$. In the case $d=2$, we choose
\begin{equation*}
\Sigma_2:=\{T_{1,2}, T_{2,2}\}=\{{\rm conv}\{0,e_1,e_1+e_2\}, {\rm conv}\{0,e_2,e_1+e_2\}\}\,,
\end{equation*}
while if $d=3$, the decomposition is given by
\begin{equation*}
\Sigma_3:=\{T_{1,3}, T_{2,3}, T_{3,3}, T_{4,3}, T_{5,3}, T_{6,3}\}\,,
\end{equation*}
where
\begin{equation*}
\begin{split}
&T_{1,3}= {\rm conv}\{0,e_1,e_1+e_2, e_1+e_2+e_3\},\quad T_{2,3}= {\rm conv}\{0,e_1,e_1+e_3, e_1+e_2+e_3\}\\
&T_{3,3}= {\rm conv}\{0,e_2,e_1+e_2, e_1+e_2+e_3\},\quad T_{4,3}= {\rm conv}\{0,e_2,e_2+e_3, e_1+e_2+e_3\}\\
&T_{5,3}= {\rm conv}\{0,e_3,e_1+e_3, e_1+e_2+e_3\},\quad T_{6,3}= {\rm conv}\{0,e_3,e_2+e_3, e_1+e_2+e_3\}\,.
\end{split}
\end{equation*}
For every simplex $T\in\Sigma_d$, we denote by $D_{T}$ the set of the edges directions for $T$, which contains $d(d+1)/2$ linearly independent vectors of $S_d$. 
For any vector $\xi\in\R^d$,  we denote by $\ell_j^{\xi,T}$ the coordinates of $\xi\otimes\xi$ in the basis $\{\tilde{\nu}_j\otimes\tilde{\nu}_j:\, \tilde{\xi}_j\in D_T\}$ of $\R^{d\times d}_{sym}$, where $\tilde{\nu}_j:= \tilde{\xi}_j/|\tilde{\xi}_j|$. 

Finally, we define the triangulation of $\wt\Omega$ induced by the partition $\Sigma_d$ as
\begin{equation*}
\mathcal{T}^d_\epsilon:=\{\alpha+\delta T\colon T\in\Sigma_d,\, \alpha\in\delta\mathbb{Z}^d\cap \wt\Omega\}\,.
\end{equation*}
We then denote by $\hat{u}_\epsilon=(\hat{u}_\epsilon^1,\dots,\hat{u}_\epsilon^d)$ and $\hat{v}_\epsilon$ the piecewise-affine interpolations of $u_\epsilon$ and $v_\epsilon$ on $\mathcal{T}^d_\epsilon$, respectively.
We also consider the piecewise constant functions 
\begin{equation}
\tilde{v}_{min,\epsilon}(x):=\min\{v_\epsilon(\beta),\,\, \beta\in\alpha+\delta([0,1]^d\cap\Z^d)\},\quad \mbox{if }x\in \alpha+[0,\delta)^d\,.
\label{pc4}
\end{equation}

The result will be an immediate consequence of the following crucial claim and of \cite[Proposition~A.1, Remark~A.2]{AFG}, which hold true for any distance inducing the convergence in measure on bounded sets (in particular, for the metric $d(u,v)$ defined in \eqref{eq:metricd}).\\
\\
\noindent
{\bf Claim:} There exists a set $K_\epsilon\subset\wt\Omega$, with
\begin{equation}
\mathcal{H}^{d-1}(\partial^* K_\epsilon)\leq C,\quad |K_\epsilon|\to0
\end{equation}
such that, setting $\bar{u}_\epsilon:= \hat{u}_\epsilon(1-\chi_{K_\epsilon})$, we have that $\overline{u}_\varepsilon$ satisfy \eqref{compactbound}. 
We subdivide the proof of this fact into 
two steps.\\
\noindent
{\bf Step 1:} The preliminary remark is that from the equi-boundedness of the energies \eqref{equiboundcomp} we can get
\begin{align}
&\int_{\wt\Omega} (\tilde{v}_{min,\epsilon}(x))^2 |\mathcal{E}\hat{u}_\epsilon(x)|^2\, \mathrm{d}x\leq C\label{equib1}\\
&\int_{\wt\Omega} |\hat{v}_\epsilon(x)-1||\nabla\hat{v}_\epsilon(x)|\, \mathrm{d}x\leq C\,.\label{equib2}
\end{align}
Let $\eta>0$ be fixed, and consider $\wt\Omega^\eta:=\{x\in\wt\Omega:\, {\rm dist}(x,\mathbb{R}^d \sm \wt\Omega)>\eta\}$. 
Since $\hat{u}_\epsilon$ is the affine interpolation of $u_\epsilon$ on each simplex of partition $\Sigma_d$, we have that
\begin{equation}
\langle(\mathcal{E}\hat{u}_\epsilon){\nu},{\nu}\rangle=\frac{\langle u_\epsilon(s_i)-u_\epsilon(s_j),\nu\rangle}{|s_i-s_j|}
\label{identity}
\end{equation}
for every pair $s_i,s_j$ of vertices of $[0,1]^d$, with $\nu=\frac{s_i-s_j}{|s_i-s_j|}$.

In order to prove \eqref{equib1}, a simple computation based on \eqref{matrix2}, \eqref{pc4} and \eqref{identity} shows that 
\begin{equation*}
\begin{split}
&\int_{\wt\Omega^\eta}(\tilde{v}_{min,\epsilon}(x))^2|\mathcal{E}\hat{u}_\epsilon(x)|^2\,\mathrm{d}x\\
&\leq\sum_{\xi\in S_d}\frac{\sigma_{|\xi|}}{\tilde{c}_{\sigma,d}}\sum_{\alpha\in\delta\mathbb{Z}^d\cap \wt\Omega}\int_{\alpha+\delta[0,1]^d}(\tilde{v}_{min,\epsilon}(x))^2|\langle\mathcal{E}\hat{u}_\epsilon(x)\xi,\xi\rangle|^2\,\mathrm{d}x\\
&=\sum_{\xi\in S_d}\frac{\sigma_{|\xi|}}{\tilde{c}_{\sigma,d}}\sum_{\alpha\in\delta\mathbb{Z}^d\cap \wt\Omega}\sum_{T\in\Sigma_d}\int_{\alpha+\delta T}(\tilde{v}_{min,\epsilon}(x))^2|\langle\mathcal{E}\hat{u}_\epsilon(x)\xi,\xi\rangle|^2\,\mathrm{d}x\\
&=\sum_{\xi\in S_d}\frac{\sigma_{|\xi|}}{\tilde{c}_{\sigma,d}}\sum_{\alpha\in\delta\mathbb{Z}^d\cap \wt\Omega}\sum_{T\in\Sigma_d}\int_{\alpha+\delta T}(\tilde{v}_{min,\epsilon}(x))^2\left|\sum_j\ell_j^{\xi,T}\langle\mathcal{E}\hat{u}_\epsilon(x)\tilde{\nu}_j,\tilde{\nu}_j\rangle\right|^2\,\mathrm{d}x\\
&\leq \sum_{\xi\in S_d}\frac{\sigma_{|\xi|}}{\tilde{c}_{\sigma,d}}\sum_{\alpha\in\delta\mathbb{Z}^d\cap \wt\Omega}\sum_{T\in\Sigma_d}\frac{\delta^{d-2}}{d!}\left|\sum_j\ell_j^{\xi,T}(v_\epsilon(\alpha+\delta s_j))D_\delta^{\tilde{\xi}_j} u_\epsilon(\alpha+\delta s_j)\right|^2\,,
\end{split}
\end{equation*}
where $s_j, s_j+\tilde{\xi}_j$ represent the only two vertices of $T$ whose difference is $\tilde{\xi}_j$. Thus, by simple inequalities we infer that
\begin{equation*}
\int_{\wt\Omega^\eta}(\tilde{v}_{min,\epsilon}(x))^2|\mathcal{E}\hat{u}_\epsilon(x)|^2\,\mathrm{d}x\lesssim F_\epsilon(u_\epsilon,v_\epsilon)\,,
\end{equation*}
whence the assertion easily follows from \eqref{equiboundcomp} and by the arbitrariness of $\eta$.
For what concerns \eqref{equib2}, we notice that $\hat{v}_\epsilon(x)$ can be rewritten on each simplex $\alpha+\delta T$, with vertices $\alpha+\delta \tilde{\xi}_i$, $i=0,1,\dots,d$ (we use here the convention $\alpha+\delta \tilde{\xi}_0:=\alpha$), as
\begin{equation}
\hat{v}_\epsilon(x)=\sum_{i=0}^dp_i(x)v_\epsilon(\alpha+\delta \tilde{\xi}_i),\quad \mbox{ for every $x\in \alpha+\delta T$,}
\label{rewriting}
\end{equation}
for some affine functions $p_i(x),\, i=0,1,\dots,d$ such that $\displaystyle\sum_{i=0}^dp_i(x)=1$. 

We first prove that
\begin{equation}
\int_{\wt\Omega^\eta} \frac{(\hat{v}_\epsilon(x)-1)^2}{\epsilon}+\epsilon|\nabla\hat{v}_\epsilon(x)|^2\,\mathrm{d}x\lesssim G_{\epsilon}(v_\epsilon,\wt\Omega^\eta)
\label{(4.10)}
\end{equation}
for $\delta$ small. Indeed, on the one hand, since $\hat{v}_\epsilon$ is the piecewise affine interpolation of $v_\epsilon$ on each simplex of the decomposition, we deduce that
\begin{equation*}
|\nabla \hat{v}_\epsilon(x)|^2=\frac{1}{\delta^2}\sum_{i=1}^d\left(\Delta_{\delta}^{\tilde{\xi}_i}v_\epsilon(\alpha)\right)^2,\quad \mbox{ for every $x\in\alpha+\delta T$,}
\end{equation*}
so that, by means of elementary inequalities, for $\delta$ sufficiently small we have that 
\begin{equation*}
\int_{\wt\Omega^\eta}\epsilon|\nabla\hat{v}_\epsilon(x)|^2\,\mathrm{d}x\leq C\sum_{\alpha\in\delta\mathbb{Z}^d\cap \wt\Omega} \epsilon\delta^{d-2}\left[\sum_{i=1}^d\left(\Delta_{\delta}^{e_i}v_\epsilon(\alpha)\right)^2\right]\,.
\end{equation*}
On the other hand, rewriting $\hat{v}_\epsilon(x)$ as in \eqref{rewriting} on each symplex $\alpha+\delta T$ for every $\alpha\in\delta\mathbb{Z}^d\cap \wt\Omega$, with the convexity of $z\to(z-1)^2$ we obtain
\begin{align*}
\int_{\alpha+\delta T}\frac{(\hat{v}_\epsilon(x)-1)^2}{\epsilon}\,\mathrm{d}x &= \frac{1}{\epsilon} \int_{\alpha+\delta T}\left(\sum_{i=1}^dp_i(x)v_\epsilon(\alpha+\delta \tilde{\xi}_i)-1\right)^2\,\mathrm{d}x\\
&\leq \frac{1}{\epsilon}\biggl(\sum_{i=0}^d(v_\epsilon(\alpha+\delta \tilde{\xi}_i)-1)^2\int_{\alpha+\delta T} p_i(x)\,\mathrm{d}x\biggr)\\
&=\frac{1}{3d!}\delta^d\sum_{i=0}^d\frac{(v_\epsilon(\alpha+\delta \tilde{\xi}_i)-1)^2}{\epsilon}\,.
\end{align*}
Hence, summing up on all simplices $\alpha+\delta T\in\mathcal{T}^d_\epsilon$ we finally get, for $\delta$ small enough,
\begin{equation}
\sum_{\alpha\in\delta\mathbb{Z}^d\cap \wt\Omega}\delta^d\frac{(v_\epsilon(\alpha)-1)^2}{\epsilon}\,\mathrm{d}x\geq\int_{\wt\Omega^\eta}\frac{(\hat{v}_\epsilon-1)^2}{\epsilon}\,\mathrm{d}x\,.
\label{inteq2}
\end{equation}

Now, as a consequence of \eqref{(4.10)}, \eqref{equiboundcomp} and the Cauchy-Schwarz inequality we deduce that
\begin{equation*}
C \geq \frac{1}{2}\int_{\Omega^\eta} \frac{(\hat{v}_\epsilon(x)-1)^2}{\epsilon}+\epsilon|\nabla\hat{v}_\epsilon(x)|^2\,\mathrm{d}x\geq\int_{\wt\Omega^\eta} |\hat{v}_\epsilon(x)-1||\nabla\hat{v}_\epsilon(x)|\, \mathrm{d}x\,,
\end{equation*}
whence \eqref{equib2} follows by the arbitrariness of $\eta$.\\
\noindent
{\bf Step 2:} We can start with the construction of the set $K_\epsilon$. As a consequence of the coarea formula and \eqref{equib2}, we then have
\begin{equation}
C\geq\int_{\wt\Omega} |\hat{v}_\epsilon(x)-1||\nabla\hat{v}_\epsilon(x)|\, \mathrm{d}x\geq \int_0^1 (1-s) \mathcal{H}^{d-1}(\partial^*\{\hat{v}_\epsilon<s\}\cap\wt\Omega)\,\mathrm{d}s\,,
\label{coarea1}
\end{equation}
whence, by the mean-value theorem, there exists $\bar{s}\in(0,1)$, say $\bar{s}=\frac{1}{4}$, such that
\begin{equation}
\int_0^1 (1-s) \mathcal{H}^{d-1}(\partial^*\{\hat{v}_\epsilon<s\}\cap\Omega)\,\mathrm{d}s\geq \frac{3}{4} \mathcal{H}^{d-1}(\partial^*K_\epsilon^1)\,,
\label{meanv}
\end{equation}
where we have set 
\begin{equation*}
K_\epsilon^1:=\left\{x\in\wt\Omega:\, \hat{v}_\epsilon(x)\leq\frac{1}{4}\right\}\,.
\end{equation*}
Thus, with \eqref{coarea1} and \eqref{meanv} we deduce that
\begin{equation}
\mathcal{H}^{d-1}(\partial^*K_\epsilon^1)\leq C\,.
\label{per1}
\end{equation}
Furthermore, again by the equi-boundedness of the energies and \eqref{inteq2}, we have
\begin{equation}
|K_\epsilon^1|\leq \left(\frac{1}{\epsilon}\int_{\wt\Omega} (\hat{v}_\epsilon(x)-1)^2\,\mathrm{d}x\right)\epsilon\leq C\epsilon\to0\,.
\label{area1}
\end{equation}

Now, with $\kappa>0$ 
fixed, we consider the set
\begin{equation*}
\mathcal{I}_{\epsilon,\delta}^\kappa:=\left\{\alpha\in\delta\mathbb{Z}^d\cap\wt\Omega:\, \max_{\xi\in S_d}\{|v_\epsilon(\alpha)-v_\epsilon(\alpha\pm\delta \xi)|\}\geq\kappa\right\},
\end{equation*}
and, denoting by $Q_\alpha$ the cube $\alpha+[0,\delta)^d$, we correspondingly define
\begin{equation*}
K_\epsilon^\kappa:= \bigcup_{\alpha\in \mathcal{I}_{\epsilon,\delta}^\kappa}Q_\alpha\,.
\end{equation*}
Notice that, if $\alpha\in \mathcal{I}_{\epsilon,\delta}^\kappa$, then by the triangle inequality there exists $\beta\in \alpha+\delta([-1,1]^d\cap\Z^d)$ such that
\begin{equation}
\max\{|v_\epsilon(\beta)-v_\epsilon(\beta+\delta e_j)|,|v_\epsilon(\beta)-v_\epsilon(\beta-\delta e_j)|\}\geq\frac{\kappa}{d}\,,\quad j=1,\dots,d\,.
\label{condbeta}
\end{equation}
Since different $\alpha',\alpha''\in \mathcal{I}_{\epsilon,\delta}^\sigma$ may share the same $\beta$ complying with \eqref{condbeta} if and only if $\alpha'-\alpha''\in\delta([-2,2]^d\cap\Z^d)$, then
\begin{equation}
\#\{\beta:\, \mbox{\eqref{condbeta} holds} \}\geq\frac{\#(\mathcal{I}_{\epsilon,\delta}^\kappa)}{\#([-2,2]^d\cap\Z^d)}\,.
\label{stimbeta}
\end{equation}
From \eqref{equiboundcomp}, the definition of $\mathcal{I}_{\epsilon,\delta}^\kappa$, and \eqref{stimbeta} we then infer that
\begin{equation}
\begin{split}
C\geq G_\epsilon(v_\epsilon)&\geq \sum_{\alpha\in\mathcal{I}_{\epsilon,\delta}^\kappa}\epsilon \delta^d \left(\sum_{j=1}^d\left|\frac{v_\epsilon(\alpha)-v_\epsilon(\alpha\pm\delta e_j)}{\delta}\right|^2\right)\\
 &\geq \frac{\epsilon\kappa^2\delta^{d-2}}{d^2\#([-2,2]^d\cap\Z^d)}\#(\mathcal{I}_{\epsilon,\delta}^\kappa)\,,
\end{split}
\label{boundsmallset}
\end{equation}
whence
\begin{equation*}
\#(\mathcal{I}_{\epsilon,\delta}^\kappa)\leq \frac{C}{\kappa^2\delta^{d-2}\epsilon}\,.
\end{equation*}
Consequently, taking into account the boundedness of the ratio $\frac{\delta}{\epsilon}$, we have
\begin{equation}
\begin{split}
\mathcal{H}^{d-1}(\partial^*K_\epsilon^\kappa)&\leq \sum_{\alpha\in\mathcal{I}_{\epsilon,\delta}^\kappa}\mathcal{H}^{d-1}(\partial^*Q_\alpha)=2d\delta^{d-1} \#(\mathcal{I}_{\epsilon,\delta}^\kappa)\leq \frac{C}{\sigma^2}\frac{\delta}{\epsilon}<+\infty\,,\\
|K_\epsilon^\kappa|&\leq \sum_{\alpha\in\mathcal{I}_{\epsilon,\delta}^\kappa}|Q_\alpha|=\delta^d\#(\mathcal{I}_{\epsilon,\delta}^\kappa)\leq \left(\frac{C}{\kappa^2}\frac{\delta}{\epsilon}\right)\delta\to0\,.
\end{split}
\label{perarea2}
\end{equation}
Hence, setting
\begin{equation*}
K_{\epsilon,\kappa}:= K_\epsilon^1\cup K_\epsilon^\kappa\,,
\end{equation*}
with \eqref{per1}, \eqref{area1} and \eqref{perarea2} we find that
\begin{equation*}
\mathcal{H}^{d-1}(\partial^* K_{\epsilon,\kappa})\leq C \mbox{ and } |K_{\epsilon,\kappa}|\to0\,.
\end{equation*}

It will be sufficient to show that, for every fixed $\kappa>0$,
\begin{equation}
\wt\Omega\sm K_{\epsilon,\kappa}\subseteq\left\{x\in\Omega:\, \tilde{v}_{min,\epsilon}(x)\geq\frac{1}{4}-2\kappa\right\}.
\label{inclusion}
\end{equation}
Indeed, choosing, e.g., $\kappa=\frac{1}{16}$ and setting $K_\epsilon:=K_{\epsilon,\frac{1}{16}}$, \eqref{inclusion} and \eqref{equib1} allow us to deduce a uniform bound for $\|\mathcal{E}\hat{u}_\epsilon\|_{L^2}$ outside the set $K_\epsilon$; namely, 
\begin{equation}
\int_{\wt\Omega\backslash K_\epsilon}|\mathcal{E}\hat{u}_\epsilon(x)|^2\, \mathrm{d}x \leq 64 \int_{\wt\Omega} (\tilde{v}_{min,\epsilon}(x))^2 |\mathcal{E}\hat{u}_\epsilon(x)|^2\, \mathrm{d}x\leq C\,.
\label{extbound}
\end{equation}
In order to prove \eqref{inclusion}, let $x\in\wt\Omega\backslash K_{\epsilon,\kappa}$ and $\alpha\in\delta\mathbb{Z}^d\cap\wt\Omega$ be such that $x\in Q_\alpha$. Since $\hat{v}_\epsilon(x)\geq\frac{1}{4}$, it must be
\begin{equation*}
\max\biggl\{v_\epsilon(\alpha), \{v_\epsilon(\alpha\pm\delta \xi)\}_{\xi\in S_d}\biggr\}\geq \frac{1}{4}\,.
\end{equation*}
Now, $\alpha\not\in\mathcal{I}_{\epsilon,\delta}^\kappa$, so that $v_\epsilon(\alpha)\geq \frac{1}{4}-\kappa$ and, by triangle inequality, $\tilde{v}_{min,\epsilon}(x)\geq \frac{1}{4}-2\kappa$ as desired.

Finally, setting $\bar{u}_\epsilon:= \hat{u}_\epsilon(1-\chi_{K_\epsilon})$, we notice that $J_{\bar{u}_\epsilon}=\partial^*K_\epsilon$ so that, taking into account \eqref{extbound} and $\mathcal{H}^{d-1}(\partial^* K_\epsilon)<+\infty$ we obtain \eqref{compactbound}. By the way, it is immediate to see that $\|\bar{u}_\varepsilon\|_{L^\infty} \leq \|u_\varepsilon\|_{L^\infty}$.
  This concludes the proof of Claim and then of the theorem.
\endproof

\section{Semicontinuity properties for the Griffith energy}\label{sec:lowerbound}
This section is devoted to prove the semicontinuity inequality \eqref{sciGriffith} in Theorem~\ref{teo:main}, assuming the convergence of $u_\varepsilon$ to $u$ guaranteed in Section~\ref{sec:compprelboun} on sequences with bounded approximating energies.
In particular, we deduce the lower limit inequality for the $\Gamma$-convergence approximation of the classic Griffith energy, with Dirichlet boundary conditions.

As in Section~\ref{sec:compprelboun}, we work with the extended set $\wt \Omega\subset \Rd$, $d\in\{2,3\}$, and functions in $\mathcal{A}^\mathrm{Dir}_\delta(\wt\Omega;\Rd)$, $\mathcal{A}^\mathrm{Dir}_\delta(\wt\Omega;\R)$.
As observed in Section~\ref{sec:discrmodel}, if $u_\varepsilon \in  \mathcal{A}_\delta\Dir(\wt\Omega;\Rd)$ are such that $u_\varepsilon \to \bar{u}$ a.e.\ in $\wt \Omega$, then $\bar{u}=u_0$ in $\wt \Omega\sm\Omega$.
Then (recall the definition of $u'$ \eqref{eq:u'} and \eqref{1112192224}), prove the lower limit inequality for $(E\ltms\Dir)_\varepsilon$ is equivalent to prove the lower inequality for the energies $(\wt E\ltms\Dir)_\varepsilon$ defined in the very same way of $(E\ltms\Dir)_\varepsilon$, but with all the integrals and corresponding notation considered in $\wt\Omega$ in place of $\Omega$.
To ease the reading, in the following we keep the same notation of Section~\ref{sec:discrmodel} for the functionals, just referring to the set $\wt\Omega$ in place of $\Omega$ in integrals, in sets of nodes, and in $\mathcal{A}_\delta\Dir(\wt\Omega;\Rd)$, $\mathcal{A}_\delta\Dir(\wt\Omega;\R)$.

We estimate separately from below the terms $F_\varepsilon$ and $F_\varepsilon^{\mathrm{div}}$ (Lemma~\ref{le:technical1}, Lemma~\ref{le:technical2}, and Proposition~\ref{prop:lowerbound}), and then address in Proposition~\ref{lowbound} the
 lower bound for the Modica-Mortola part $G_\varepsilon$, by a blow-up argument.
We remark that the results concerning $F_\varepsilon$ and $F_\varepsilon^{\mathrm{div}}$ hold under the only assumption that $\delta=\delta(\epsilon)$ vanishes as $\varepsilon \to 0$. In contrast, we use the assumption $\lim_{\epsilon\to0}\frac{\delta}{\epsilon}=0$ to estimate the Modica-Mortola terms from below in Step~3 of Proposition~\ref{lowbound}.

\begin{lem}\label{le:technical1}
Let $u_\epsilon\in\mathcal{A}_\delta\Dir(\wt\Omega;\mathbb{R}^d)$, $v_\epsilon\in\mathcal{A}_\delta\Dir(\wt\Omega;\mathbb{R})$ be such that 
\begin{equation}
\sup_\epsilon (E^{\mathrm{Dir}}\ltms)_\epsilon(u_\epsilon,v_\epsilon)<+\infty\,,
\label{equiboundedness}
\end{equation}
$d(u_\epsilon, u) \to 0$, with  $u \in GSBD^2_\infty(\wt \Omega)$, and $v_\epsilon \to 1$ in $L^2(\wt \Omega)$.
Then, for every $\xi\in S_d$,
\begin{equation}
\begin{split}
\mathop{\lim\inf}_{\epsilon\to0} F^\xi_\epsilon(u_\epsilon,v_\epsilon)&\geq \frac{1}{|\xi|^4}\int_{\wt \Omega} |\langle\mathcal{E}u(x)\xi,\xi\rangle|^2\,\mathrm{d}x\,.
\end{split}
\label{bulkpart1}
\end{equation}
\end{lem}

\proof

For simplicity, we develop the proof in dimension $d=3$, although the following slicing argument would hold in any dimension $d\geq2$. Let $\xi\in S_3$ be fixed, and $\{\xi_1,\xi_2,\xi_3\}$ be an orthogonal basis of $\R^3$ such that $\xi_i\in\Z^3$ for every $i=1,2,3$ and $\xi_1=\xi$. Setting $Q_\xi:=\sum_{i=1}^3[0,1)\xi_i$, we note that $M_\xi:=|Q_\xi|={\rm det}(\xi_1,\xi_2,\xi_3)$ and $M_\xi\in\Z$. If we denote by $z_l$ the points of $\Pi^\xi$ such that
\begin{equation*}
\{z_l:\, l=1,\dots,M_\xi\}:=\Z^3\cap Q_\xi,
\end{equation*}
we can split $\Z^3$ into the union of disjoint copies of $Z^\xi:=\bigoplus_{i=1}^3\Z\xi_i$ as
\begin{equation*}
\Z^3=\bigcup_{l=1}^{M_\xi}Z^{\xi,l}:=\bigcup_{l=1}^{M_\xi}(z_l+Z^\xi)
\end{equation*}
(see the proof of \cite[Theorem~4.1]{BG} and Figure~\ref{fig:lattices}, in the sample case of $\xi=e_1+e_2+e_3$). 
\begin{figure}[htbp]
\begin{minipage}[c]{0.45\linewidth}
\includegraphics[width=\linewidth]{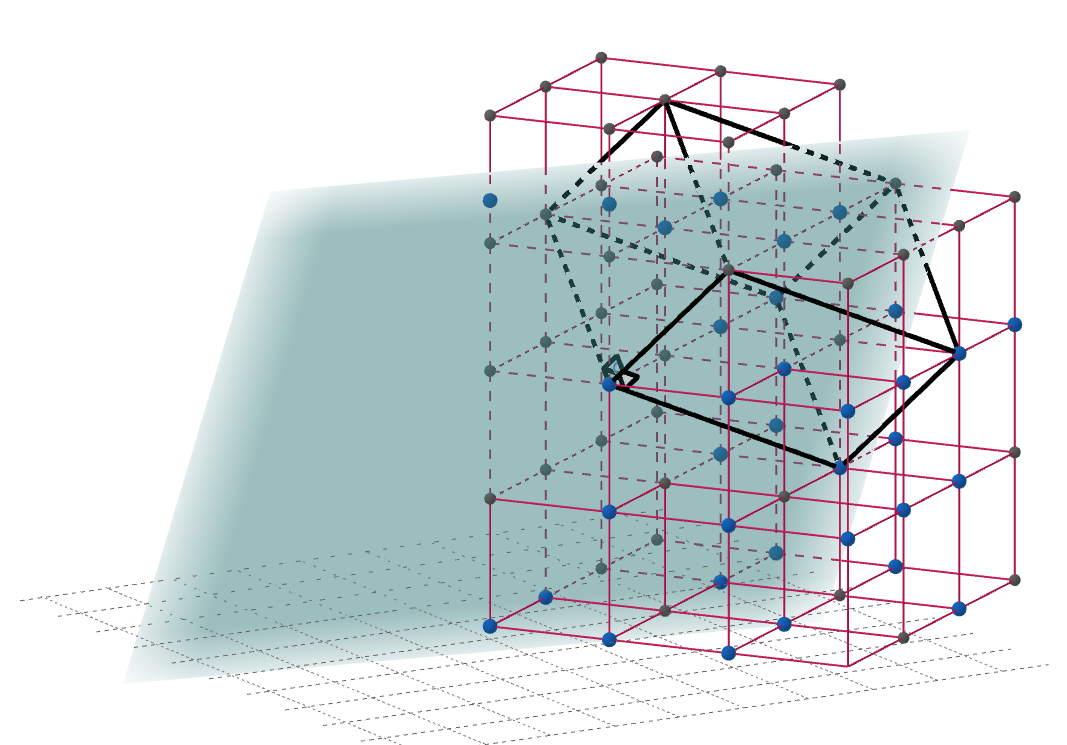}
\end{minipage}
\begin{minipage}[c]{0.45\linewidth}
\includegraphics[width=\linewidth]{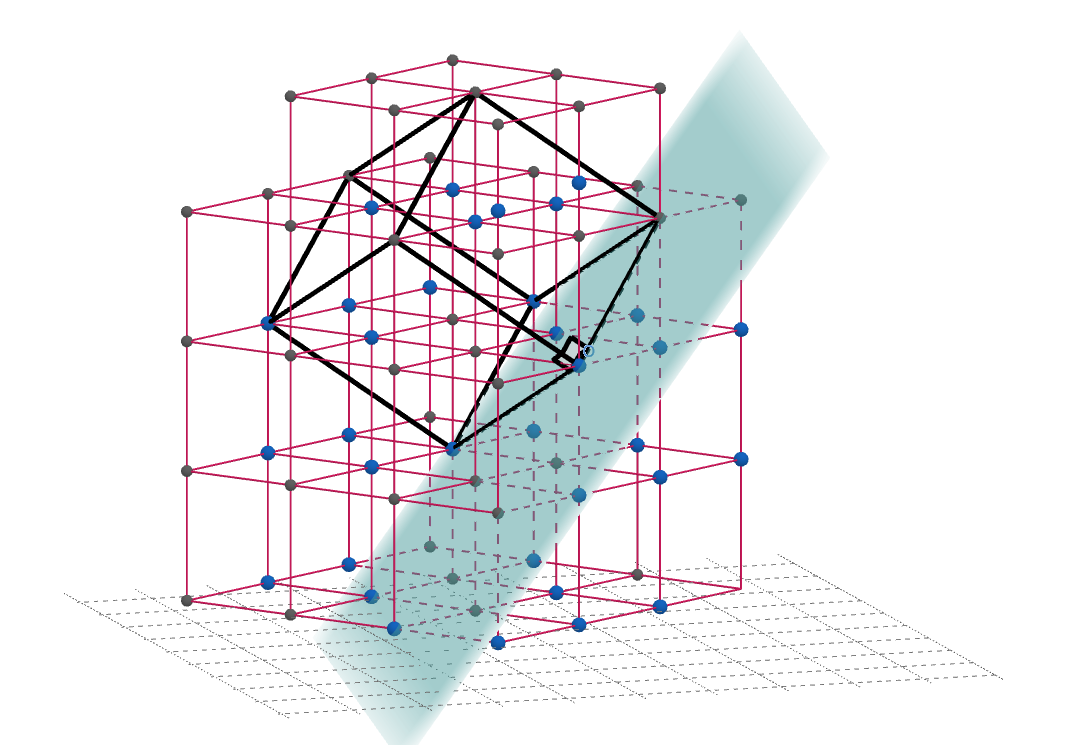}
\end{minipage}
\caption{The lattice corresponding to $\xi=e_1+e_2+e_3$, with the plane $\Pi^\xi$, from two different points of view. Notice that the main sidelengths are $\sqrt{3}$, $\sqrt{2}$, $\sqrt{6}$, so that $M_\xi=6$.}\label{fig:lattices}
\end{figure}

We claim that
\begin{equation}
\mathop{\lim\inf}_{\epsilon\to0}\sum_{\alpha\in Z_{\delta}^{l}(\wt\Omega)} \delta (v_\epsilon(\alpha))^2 (\langle u_\epsilon(\alpha+\delta\zeta)-u_\epsilon(\alpha),\zeta\rangle)^2\geq \frac{1}{M_\xi}\int_{\wt\Omega}(\langle\mathcal{E}u(x)\xi,\xi\rangle)^2\,\mathrm{d}x
\label{secclaim}
\end{equation}
for $\zeta=\pm\xi$ and for every $l=1,\dots,M_\xi$, where $Z_{\delta}^{l}(\wt\Omega):=R_\delta^{\xi}(\wt\Omega)\cap\delta Z^{\xi,l}$. The conclusion \eqref{bulkpart1} will follow up to multiplying by $\frac{1}{|\xi|^4}$ both the sides of \eqref{secclaim} and summing up over  the sublattices.

In order to prove \eqref{secclaim}, we introduce two other  piecewise constant interpolations  $\tilde{u}_\epsilon$ and $\tilde{v}_\epsilon$ of $u_\epsilon$ and $v_\epsilon$, respectively. For  $\alpha \in Z_{\delta}^l(\wt\Omega)$ and $Q_\xi$ as before, we set 
\begin{equation}
\tilde{u}_\epsilon(x):=u_\epsilon(\alpha)\,, \qquad \tilde{v}_\epsilon(x):=v_\epsilon(\alpha),\quad x\in\alpha+\delta Q_\xi\,.
\label{firstinterp1}
\end{equation}
The triangular inequality implies that $\tilde v_\e \to 1$ in $L^1(\wt \Omega)$. We also have that $d(\tilde u_\e, u)\to 0$. This follows from the fact that $u_\epsilon - \tilde{u}_\epsilon \to 0$ in measure. To see this, set $\tilde{g}^\zeta_\e=\arctan(\langle\tilde{u}_\epsilon , \zeta\rangle)$,  $g^\zeta_\e=\arctan(\langle u_\epsilon , \zeta\rangle)$, $\zeta\in\{e_1,e_2,e_3\}$. We have by definition of the interpolants that
\[
\begin{split}
&\int_{\alpha +\delta Q_\xi}|\tilde{g}^\zeta_\e(x)-g^\zeta_\e(x)|\,\mathrm{d}x =\int_{\alpha +\delta Q_\xi}|\tilde{g}^\zeta_\e(\alpha)-g^\zeta_\e(x)|\,\mathrm{d}x\\
&\le \sum_{i=1}^{N_\xi}\int_{\alpha +\delta Q_\xi}|g^\zeta_\e(x-\delta \psi_i)-g^\zeta_\e(x)|\,\mathrm{d}x \,,
\end{split}
\]
where $N_\xi$ is finite depending on $\xi$ and $\psi_i$  are the vectors connecting $\alpha$ with the $N_\xi$ remaining integer vertices in $Q_\xi$. We now observe two facts: 1) from Proposition~\ref{prop:compactness} we have that there exist $\overline{u}_\epsilon$ with $\overline{u}_\epsilon-u_\epsilon \to 0$ in measure and $\overline{u}_\epsilon \to u$ weakly in $GSBD^2_\infty(\wt \Omega)$;
2) arguing for any fixed $\zeta=e_i$ as in \cite[proof of Theorem~1.1, Compactness]{CC18Comp} we have that $\arctan(\langle\overline{u}_\epsilon,\zeta\rangle)$ is compact in $L^1(\wt \Omega)$ (in fact, in $\wt \Omega \sm A^\infty_u$, $\arctan(\langle\overline{u}_\epsilon,\zeta\rangle)\to\arctan(\langle u , \zeta\rangle)$ for any $\zeta \in \mathbb{S}^2$, and, in $A^\infty_u$, $|\arctan(\langle\overline{u}_\epsilon , \zeta\rangle)|\to \frac{\pi}{2}$ for $\mathcal{H}^2$-a.e.\ $\zeta \in \mathbb{S}^{2}$, but the limit exists for any $\zeta$). 
Then $g^\zeta_\epsilon$ is compact in $L^1(\wt\Omega)$ so that, summing up on all $\alpha$'s   in $Z_{\delta}(\wt\Omega)$ and using the Fr\'echet-Kolomogorov criterion, we get $\tilde{g}^\zeta_\epsilon - g^\zeta_\epsilon \to 0$ in  $L^1(\wt\Omega)$. Hence, the claim is proved.

We define $\wt\Omega^\eta$ as the set of $x \in \tilde \Omega$ whose distance from $\partial \wt \Omega$ is at least $\eta$. Setting $\wt\Omega^l_\delta:=\bigcup_{\alpha \in Z_{\delta}^l}(\alpha+\delta Q_\xi)$, we clearly have that  $\wt\Omega^\eta \subseteq \wt\Omega^l_\delta$ for $\delta$ small enough. Furthermore (we argue for $\zeta=\xi$ in \eqref{secclaim}, the case $\zeta=-\xi$ is analogous) 
\begin{equation}\label{eq: rewr}
\begin{split}
&\sum_{\alpha\in Z_{\delta}^{l}(\wt\Omega)} \delta (v_\epsilon(\alpha))^2 (\langle u_\epsilon(\alpha+\delta \xi)-u_\epsilon(\alpha),\xi\rangle)^2 \\
&=\frac{1}{\delta^2 M_\xi}\int_{\wt\Omega^l_\delta} (\tilde{v}_\epsilon(x))^2 (\langle \tilde{u}_\epsilon(x+\delta\xi)-\tilde{u}_\epsilon(x),\xi\rangle)^2\,\mathrm{d}x\\
&=\frac{1}{M_\xi}\int_{\pi^\xi(\wt\Omega^l_\delta)}\int_{(\wt\Omega^l_\delta)_{\xi,y}}(\tilde{v}_\epsilon^{\xi,y}(t))^2 \left( \frac{\tilde{u}_\epsilon^{\xi,y}(t+\delta)-\tilde{u}_\epsilon^{\xi,y}(t)}{\delta}\right)^2\,\mathrm{d}t\,\mathrm{d}\mathcal{H}^2(y)\,.
\end{split}
\end{equation}
Observe that $\tilde{u}_\epsilon^{\xi,y} \in PC_\delta((\wt\Omega^l_\delta)_{\xi,y})$, where $PC_\delta$ here denotes the space of  piecewise constant functions on intervals of size $\delta$. We now define $\hat u_{\epsilon,\xi,y}$ as the piecewise linear interpolation of $\tilde{u}_\epsilon^{\xi,y}$ on $(\wt\Omega^l_\delta)_{\xi,y}$. We remark that $\hat u_{\epsilon,\xi,y}$ has nothing to do with the slices $\hat{u}_\e^{\xi, y}$ of the affine function $\hat u_\e$ used in Proposition \ref{prop:compactness}, hence the different notation. Now, \eqref {eq: rewr} can be rewritten as
\begin{equation*}
\begin{split}
&\sum_{\alpha\in Z_{\delta}^{l}(\wt\Omega)} \delta (v_\epsilon(\alpha))^2 (\langle u_\epsilon(\alpha+\delta \xi)-u_\epsilon(\alpha),\xi\rangle)^2 \\
&= \frac{1}{M_\xi}\int_{\pi^\xi(\wt\Omega^l_\delta)}\int_{(\wt\Omega^l_\delta)_{\xi,y}}(\tilde{v}_\epsilon^{\xi,y}(t))^2 (\dot{\hat u}_{\epsilon, \xi , y}(t))^2\,\mathrm{d}t\,\mathrm{d}\mathcal{H}^2(y)\\
&\geq \frac{1}{M_\xi}\int_{\pi^\xi(\wt\Omega^\eta)}\int_{\wt\Omega^\eta_{\xi,y}}(\tilde{v}_\epsilon^{\xi,y}(t))^2 (\dot{\hat u}_{\epsilon, \xi , y}(t))^2\,\mathrm{d}t\,\mathrm{d}\mathcal{H}^2(y)
\end{split}
\end{equation*}
and we are left to prove that, for $\mathcal{H}^{2}$-a.e. $y\in\Pi^\xi$,
\begin{equation}\label{1212191044}
\int_{({\wt\Omega}^\eta\sm (A^\infty_u))_{\xi,y}}|\dot{u}^{\xi,y}(t)|^2\,\mathrm{d}t  \leq  \mathop{\lim\inf}_{\epsilon\to0}\int_{{\wt\Omega}^\eta_{\xi,y}}(\tilde{v}_\epsilon^{\xi,y}(t))^2|\dot{\hat u}_{\epsilon, \xi, y}(t)|^2\,\mathrm{d}t\,.
\end{equation}
Indeed, if the above holds, \eqref{bulkpart1} will follow as a consequence of Fatou's lemma by integrating the above estimate over $\Pi^\xi$ and observing that, since $u\in GSBD^2_\infty({\wt\Omega})$, then $\dot{u}^{\xi,y}(t)=\langle\mathcal{E}u(y+t\xi)\xi,\xi\rangle$ for a.e. $t\in {(\wt\Omega\sm A^\infty_u)}_{\xi,y}$ and $\mathcal E u=0$ in $A^\infty_u$. (Notice that we have also to use the arbitrariness of $\eta>0$.)

In the following we argue for $\wt\Omega$ in place of $\wt\Omega_\eta$, in order to simplify the notation, since we know that $d(\tilde {u}_\epsilon,u)\to 0$. Nevertheless, all the inequalities may be localized on $\wt\Omega_\eta$.
Since $\tilde {u}_\epsilon \to u$ in measure in $\wt\Omega \sm A^\infty_u$, by Fubini's Theorem (see~\cite[(5.5)]{CC}) 
we have that $\tilde{u}_\epsilon^{\xi,y}\to u^{\xi,y}$ in measure in ${(\wt\Omega \sm A^\infty_u)}_{\xi,y}$ for $\mathcal{H}^2$-a.e. $y\in\Pi^\xi$. The same holds then for the piecewise affine functions $\hat u_{\epsilon,\xi,y}$. Summarizing, we have for $\mathcal{H}^2$-a.e. $y\in\Pi^\xi$:
\begin{equation}
\begin{split}
\hat u_{\epsilon,\xi,y}\to u^{\xi,y} \mbox{ in measure in }  {(\wt\Omega \sm A^\infty_u)}_{\xi,y}\,,  \qquad 
 \tilde v_\epsilon^{\xi,y}\to 1
\mbox{ in $L^1({\wt\Omega}_{\xi,y})$,}
\label{conve}
\end{split}
\end{equation}
where the second one follows by Fubini's Theorem.
For fixed $y \in \Pi^\xi$ such that \eqref{conve} holds and the $\liminf$ in \eqref{1212191044} is finite, denoting by $\hat{v}_{\epsilon, \xi,y}$ the piecewise affine interpolations of $\tilde v^{\xi,y}_\epsilon$, from \eqref{equiboundedness} and the triangular inequality we deduce that 
\begin{equation*}
\frac{1}{\epsilon}\int_{{\wt\Omega}_{\xi,y}}(\tilde v_\epsilon^{\xi,y}(t)-1)^2\,\mathrm{d}t + \epsilon \int_{{\wt\Omega}_{\xi,y}}\dot{\hat{v}}_{\epsilon, \xi,y}(t)^2\,\mathrm{d}t \leq C(y)\,.
\end{equation*}
In view of Lemma~\ref{Lemma1D}, there exists a finite set $I^{\xi,y}\subset{\wt\Omega}_{\xi,y}$ such that for every $A^{\xi,y}$ open, with $A^{\xi,y}\subset\!\subset\!{\wt\Omega}_{\xi,y}\sm I^{\xi,y}$, there exists $\kappa>0$ such that
\begin{equation*}
\mathop{\lim\inf}_{\epsilon\to0} \inf_{s\in A^{\xi,y}}\tilde v_\epsilon^{\xi,y}(s)\geq\kappa\,.
\end{equation*}
In particular, we may assume that there exists $\kappa'>0$ such that, for $\epsilon$ small enough,
\begin{equation*}
\tilde v_\epsilon^{\xi,y}(s)\geq\kappa', \quad s\in A^{\xi,y},
\end{equation*}
so that
\begin{equation}
\kappa'\sup_\epsilon\int_{A^{\xi,y}}|\dot{\hat{u}}_{\epsilon, \xi,y}(t)|^2\,\mathrm{d}t\leq\sup_\epsilon\int_{{\wt\Omega}_{\xi,y}}(\tilde v_{\epsilon}^{\xi,y}(t))^2|\dot{\hat{u}}_{\epsilon, \xi,y}(t)|^2\,\mathrm{d}t<+\infty\,.
\label{interstima}
\end{equation}
Up to considering separately its connected components, we may assume that $A^{\xi,y}$ be connected and contained in one of the finitely many connected components of $\wt\Omega_{\xi,y}\sm I^{\xi,y}$ (it is not restrictive to assume $\wt\Omega$ connected).
Arguing as in \cite[part below (3.21)]{CC18Comp}, we have that by the regularity of $\hat{u}_{\epsilon, \xi,y}$, \eqref{conve}, and \eqref{interstima} one of the following two alternative possibilities hold: 
\begin{enumerate}
\item either  $|\hat{u}_{\epsilon, \xi,y}(x)|$ converge to $+\infty$ for some $x \in A^{\xi,y}$ and then $|\hat{u}_{\epsilon, \xi,y}| \to +\infty$ on $A^{\xi,y}$ and  $A^{\xi,y} \subset (A^\infty_u)_{\xi,y}$;
\item or $(\hat{u}_{\epsilon, \xi,y})_\epsilon$ is bounded in $H^1(A^{\xi,y})$ and then 
\begin{equation*}
u^{\xi,y}\in H^1(A^{\xi,y}) \text{ and } \hat{u}_{\epsilon, \xi,y}\rightharpoonup u^{\xi,y} \text{ in }H^1(A^{\xi,y}) \,.
\end{equation*} 
\end{enumerate}
In particular, ${\wt\Omega}_{\xi,y} \sm I^{\xi,y}$ is made up of a finite union of intervals, where either $\hat{u}_{\epsilon, \xi,y}$ converge in $H^1_{\rm loc}$ or $\hat{u}_{\epsilon, \xi,y}\to +\infty$. Therefore we may partition $\wt\Omega_{\xi,y}$ as $\wt\Omega_{\xi,y}^1 \cup \wt\Omega_{\xi,y}^2 \cup I^{\xi,y}$, where $\wt\Omega_{\xi,y}^1$, $\wt\Omega_{\xi,y}^2$ are finite unions of open intervals with boundary contained in $I^{\xi,y}$, such that $\hat{u}_{\epsilon, \xi,y} \to +\infty$ in $\wt\Omega_{\xi,y}^2$ and  $\hat{u}_{\epsilon, \xi,y} \to u^{\xi,y}$ in $H^1(A^{\xi,y})$ for every $A^{\xi,y} \subset \subset \wt\Omega_{\xi,y}^1$.

With \eqref{conve} and Lemma~\ref{lemma2} we obtain that for every 
$A^{\xi,y}\subset\subset {\wt\Omega}_{\xi,y}^1$  
\begin{equation}
\int_{A^{\xi,y}}|\dot{u}^{\xi,y}(t)|^2\,\mathrm{d}t\leq \mathop{\lim\inf}_{\epsilon\to0}\int_{{\wt\Omega}_{\xi,y}}(\tilde v_\epsilon^{\xi,y}(t))^2|\dot{\hat{u}}_{\epsilon, \xi,y}(t)|^2\,\mathrm{d}t\,.
\label{interstima1}
\end{equation}
Notice that \eqref{interstima1} holds for any arbitrary open set 
$A^{\xi,y}\subset \wt\Omega_{\xi,y}^1$, so that (since $\wt\Omega_{\xi,y}^2 \subset (A^\infty_u)_{\xi,y}$)
\begin{equation*}
 \int_{(\wt\Omega \sm A^\infty_u)_{\xi,y}}|\dot{u}^{\xi,y}(t)|^2\,\mathrm{d}t  \leq  \int_{\wt\Omega_{\xi,y}^1}|\dot{u}^{\xi,y}(t)|^2\,\mathrm{d}t\leq \mathop{\lim\inf}_{\epsilon\to0}\int_{\wt\Omega_{\xi,y}}(\tilde v_\epsilon^{\xi,y}(t))^2|\dot{\hat{u}}_{\epsilon, \xi,y}(t)|^2\,\mathrm{d}t\,.
\label{interstima2}
\end{equation*}
As observed before, the above estimate may be localized in $\wt\Omega_\eta$, obtaining \eqref{1212191044} and thus concluding the proof.
\endproof

 For every 
$\xi\in\R^d\backslash\{0\}$, $u\in L^1(\wt\Omega;\mathbb{R}^d)$, $v\in L^1(\wt\Omega)$, we define 
\begin{equation}\label{eq:H12}
\begin{split}
H^\xi(u,v)&:=\int_{\wt\Omega}({v}(x))^2 \,\Big|\Big\langle \mathcal{E}u(x)\frac{\xi}{|\xi|},\frac{\xi}{|\xi|}\Big\rangle\Big|^2 \mathrm{d}x\,.
\end{split} 
\end{equation} 
Setting $Z_{\delta}(\wt\Omega):=R^{\rm div}_\delta(\wt\Omega)\cap\delta Z$, where $R^{\rm div}_\delta(\wt\Omega)$ was defined in \eqref{range} (for the domain $\Omega$, here we consider as always the analogous one for $\wt\Omega$), and $Z:=2\Z^d$, and 
\begin{align}
&Q_{2m}:=\{x\in\mathbb{R}^d:\, |\langle x,e_i\rangle|\leq m,\,i=1,\dots,d\},\label{quadratone}\\
&Q_{2m,i,\pm}:=\{x\in Q_{2m}:\, \pm \langle x,e_i\rangle\geq0\}\label{mezquadratone}
\end{align}
\begin{figure}[htbp]
\centering
\def\svgwidth{150pt}
\begingroup%
  \makeatletter%
  \providecommand\color[2][]{%
    \errmessage{(Inkscape) Color is used for the text in Inkscape, but the package 'color.sty' is not loaded}%
    \renewcommand\color[2][]{}%
  }%
  \providecommand\transparent[1]{%
    \errmessage{(Inkscape) Transparency is used (non-zero) for the text in Inkscape, but the package 'transparent.sty' is not loaded}%
    \renewcommand\transparent[1]{}%
  }%
  \providecommand\rotatebox[2]{#2}%
  \newcommand*\fsize{\dimexpr\f@size pt\relax}%
  \newcommand*\lineheight[1]{\fontsize{\fsize}{#1\fsize}\selectfont}%
  \ifx\svgwidth\undefined%
    \setlength{\unitlength}{471.28060913bp}%
    \ifx\svgscale\undefined%
      \relax%
    \else%
      \setlength{\unitlength}{\unitlength * \real{\svgscale}}%
    \fi%
  \else%
    \setlength{\unitlength}{\svgwidth}%
  \fi%
  \global\let\svgwidth\undefined%
  \global\let\svgscale\undefined%
  \makeatother%
  \begin{picture}(1,1.03461722)%
    \lineheight{1}%
    \setlength\tabcolsep{0pt}%
    \put(0,0){\includegraphics[width=\unitlength,page=1]{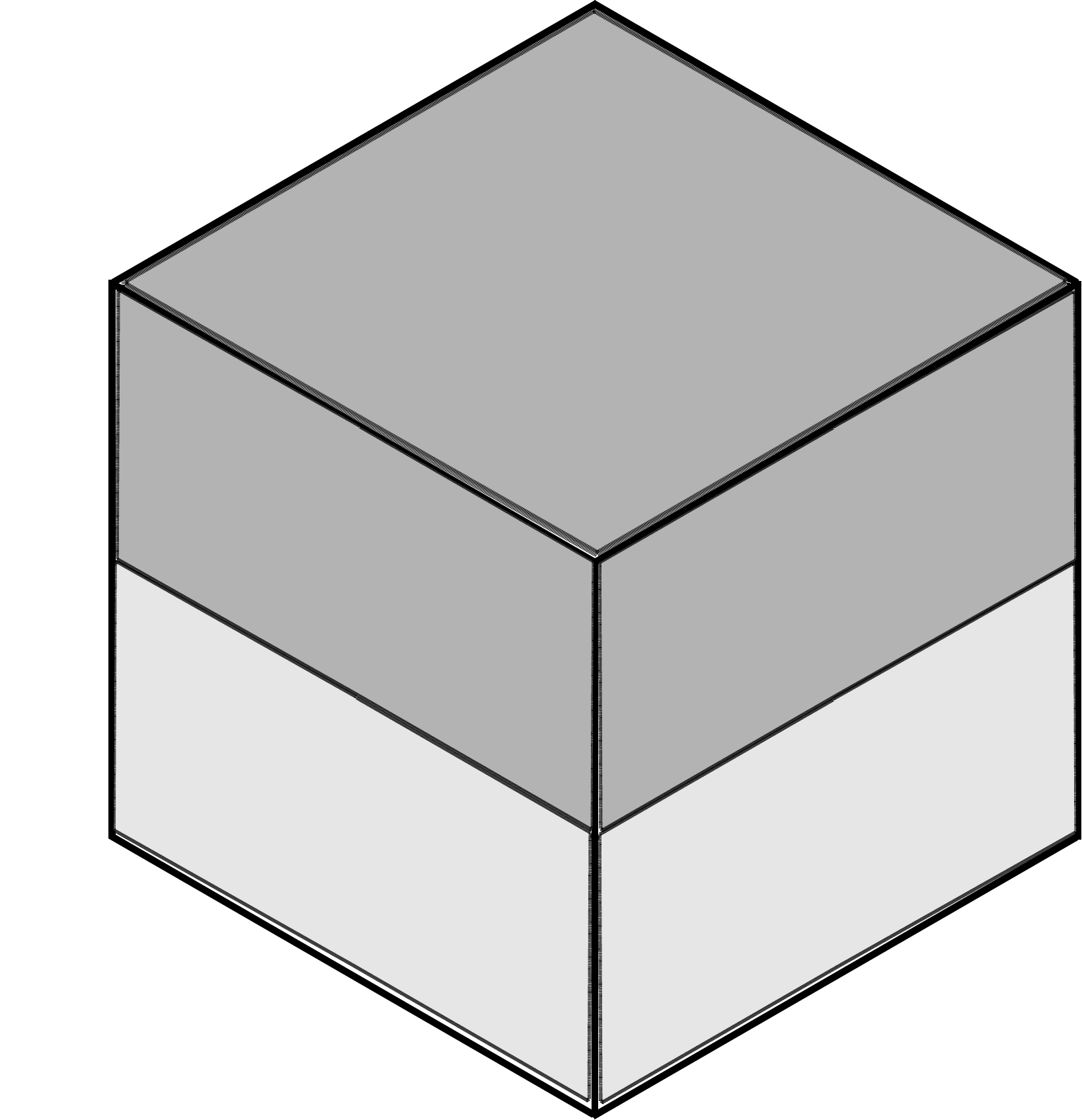}}%
    \put(0.1777066,0.29528386){\color[rgb]{0,0,0}\makebox(0,0)[lt]{\lineheight{1.25}\smash{\begin{tabular}[t]{l}$Q_{2m,i,+}$\end{tabular}}}}%
    \put(0,0){\includegraphics[width=\unitlength,page=2]{halfcubes.pdf}}%
    \put(0.76904978,0.0402809){\color[rgb]{0,0,0}\makebox(0,0)[lt]{\lineheight{1.25}\smash{\begin{tabular}[t]{l}$Q_{2m,i,-}$\end{tabular}}}}%
    \put(0,0){\includegraphics[width=\unitlength,page=3]{halfcubes.pdf}}%
    \put(0.02382549,0.57226991){\color[rgb]{0,0,0}\makebox(0,0)[lt]{\lineheight{1.25}\smash{\begin{tabular}[t]{l}$e_i$\end{tabular}}}}%
  \end{picture}%
\endgroup%

\caption{The half-cubes $Q_{2m,i,\pm}$.}\label{fig:halfcubes}
\end{figure} 
(see Fig.~\ref{fig:halfcubes}), we introduce the class of real-valued piecewise constant functions on the cells $\alpha+\delta Q_{2m}$ defined as
\begin{equation*}
\mathcal{A}_{2m\delta}(\wt\Omega;\mathbb{R}):=\biggl\{v\colon\Omega\to\mathbb{R}\colon v(x)\equiv v(\alpha) \mbox{ for every }x\in(\alpha+\delta Q_{2m})\cap\wt\Omega\mbox{ for any }\alpha\in Z_{\delta}(\wt\Omega)\biggr\}\,.
\end{equation*}

\begin{lem}\label{le:technical2}
Let $u\in GSBD^2_\infty(\wt\Omega)$ and $(w_\epsilon)_\epsilon$, $(v_\epsilon)_\epsilon$ be sequences such that $v_\epsilon\in\mathcal{A}_{2\delta}(\wt\Omega;\mathbb{R})$, $d(w_\epsilon, u)\to 0$,
\begin{align}
&\sup_{\epsilon>0}\left\{\sum_{i=1}^{d}H^{e_i}(w_\epsilon,v_\epsilon)\right\}<+\infty\,, \label{equibound}\\
&w_\epsilon^{e_i,y}\in H^1({\wt\Omega}_{e_i,y})\mbox{ for a.e. $y\in\Pi^{e_i}$, $i=1,\dots,d$,}\label{secregularity}\\
\sum_{\alpha\in Z_{\delta}({\wt\Omega})}&\delta^d\left(\frac{1}{\epsilon}(v_\epsilon(\alpha)-1)^2 + \epsilon\left(\frac{v_\epsilon(\alpha+2\delta e_i)-v_\epsilon(\alpha)}{\delta}\right)^2\right)\leq C,\, i=1,\dots,d\,.\label{bounds2}
\end{align}
Then
\begin{equation}\label{eq:liminfdiv}
\liminf_{\epsilon\to 0} \int_{\wt\Omega} (v_\epsilon(x))^2 (\mathrm{div}\, w_\epsilon(x))^2 \dx\geq \int_{\wt\Omega} (\mathrm{div}\, u(x))^2 \dx\,.
\end{equation}
\end{lem}

\proof

Notice that, under the assumption \eqref{equibound}, from the identity
\begin{equation}
\sum_{i=1}^d\langle Ae_i,e_i\rangle={\rm tr}(A)
\label{matrixidentity}
\end{equation}
we infer that $\sup_{\epsilon>0}\int_{\wt\Omega} (v_\epsilon(x))^2 (\mathrm{div}\, w_\epsilon(x))^2 \dx< +\infty$. We then show that
\begin{equation}
{v}_\epsilon\, \mathrm{div} \, w_\epsilon\rightharpoonup {\rm div}\,u\quad \mbox{ in }L^2({\wt\Omega}\sm A_u^\infty)\,,\label{debole3}
\end{equation} 
from which \eqref{eq:liminfdiv} immediately follows, recalling that $\mathcal E u=0$ in $A^\infty_u$. 
Note that by Egorov's Theorem, with fixed $\eta>0$ there exists ${\wt\Omega}_\eta\subset{{\wt\Omega}\sm A^\infty_u}$ such that $|{({\wt\Omega}\sm A^\infty_u)}\sm {\wt\Omega}_\eta|<\eta$ and $v_\epsilon>1-\eta$ on ${\wt\Omega}_\eta$ for $\epsilon$ small enough. 

Now, under assumptions \eqref{equibound}-\eqref{bounds2}, an analogous slicing argument as for the proof of Lemma~\ref{le:technical1} applied to $w_\epsilon^{e_i,y}$ shows that
\begin{equation}
\int_{\wt{\Omega}\sm A_u^\infty} (\langle (\mathcal{E}u(x))e_i,e_i\rangle-g(x))^2\,\mathrm{d}x \leq \mathop{\lim\inf}_{\epsilon\to0} \int_{\wt{\Omega}\sm A_u^\infty}({v}_\epsilon \langle (\mathcal{E}w_\epsilon(x))e_i,e_i\rangle-g(x))^2\,\mathrm{d}x
\label{debole11bis}
\end{equation}
for every $g\in L^2(\wt{\Omega} \sm A_u^\infty)$ and every $i=1,\dots,d$. The proof of \eqref{debole11bis} can be developed in the case $g=0$, the general case following by approximation of $g\in L^2(\wt{\Omega}\sm A_u^\infty)$ with piecewise constant functions on a Lipschitz partition of $\widetilde{\Omega}$. 

From \eqref{debole11bis} we then get
\begin{align}
&\langle (\mathcal{E}w_\epsilon)e_i,e_i\rangle \chi_{\wt{\Omega}_\eta} \rightharpoonup \langle (\mathcal{E}u)e_i,e_i\rangle\chi_{\wt{\Omega}_\eta}\quad \mbox{ in }L^2({\wt{\Omega}\sm A^\infty_u }),\mbox{ for every $i=1,\dots,d$}\,,\label{debole1bis}
\end{align} 
whence, by the identity \eqref{matrixidentity} we obtain
\begin{equation*}
{\rm div \, }w_\epsilon \chi_{\wt{\Omega}_\eta}\rightharpoonup {\rm div\,}u \chi_{\wt{\Omega}_\eta}\quad \mbox{ in }L^2({\wt{\Omega}\sm A^\infty_u})\,.
\end{equation*}
Finally, since $|{(\wt{\Omega}\sm A^\infty_u)} \sm \wt{\Omega}_\eta|<\eta$, letting $\eta\to0$ and by the absolute continuity of the integral we obtain
\begin{equation}
{\rm div\,}w_\epsilon \rightharpoonup {\rm div\,}u \quad \mbox{ in }L^2({\wt{\Omega}\sm A^\infty_u})\,.
\label{debole3bis}
\end{equation}
The assertion \eqref{debole3} now follows from \eqref{debole3bis} and Lemma~\ref{lemma2} since $v_\epsilon\leq1$ and $v_\epsilon\to1$ a.e. in ${\wt{\Omega}}$.
\endproof

As a consequence of Lemma~\ref{le:technical2}, we deduce now the optimal lower bound for the functionals $F_\epsilon^{\rm div}(u,v)$ as defined in Section~\ref{sec:discrmodel}.

\begin{prop}\label{prop:lowerbound}
Let $u_\epsilon\in\mathcal{A}_\delta({\wt{\Omega}};\mathbb{R}^d)$, $v_\epsilon\in\mathcal{A}_\delta({\wt{\Omega}};\mathbb{R})$ be such that 
\begin{equation}
\sup_\epsilon (E^{\mathrm{Dir}}\ltms)_\epsilon(u_\epsilon,v_\epsilon)<+\infty\,,
\label{equiboundedness2}
\end{equation}
$d(u_\epsilon, u) \to 0$, with  $u \in GSBD^2_\infty(\wt \Omega)$, and $v_\epsilon \to 1$ in $L^2(\wt \Omega)$.
Then
\begin{equation}
\begin{split}
\mathop{\lim\inf}_{\epsilon\to0} F_\epsilon^{\rm div}(u_\epsilon,v_\epsilon)&\geq \int_{\wt \Omega} |{\rm div}\,u(x)|^2\,\mathrm{d}x\,.
\end{split}
\label{bulkpart2}
\end{equation}
\end{prop}

\proof
We prove \eqref{bulkpart2} for $d=3$, the case $d=2$ being analogous. 
Notice that $\Z^3$ admits the following decomposition:
\begin{equation*}
\begin{split}
\Z^3=\bigcup_{l=1}^8 Z^{l}:= & Z\cup\bigcup\biggl\{ Z+\xi:\, \xi\in\left\{\{e_i\}_{i=1,2,3},\{e_i+e_j\}_{1\leq i<j\leq3},e_1+e_2+e_3\right\}\biggr\}\,.
\end{split}
\end{equation*}
Correspondingly,  recalling that  $Z_{\delta}^{l}(\wt\Omega)=R_\delta^{\rm div}(\wt\Omega)\cap\delta Z^{l}$ and  setting 
\begin{equation*}
{F}_\epsilon^{{\rm div},l}(u,v): =  \frac{1}{8}\sum_{\alpha\in Z_{\delta}^{l}(\wt\Omega)}\delta(v(\alpha))^2\left|{\rm Div}_{\delta}u(\alpha)\right|^2
\end{equation*} 
we can rewrite the energies as
${F}_\epsilon^{{\rm div}}(u,v)=\sum_{l=1}^8{F}_\epsilon^{{\rm div},l}(u,v)$,
so that
\begin{equation}
\mathop{\lim\inf}_{\epsilon\to0} F^{\rm div}_\epsilon(u_\epsilon,v_\epsilon)\geq\sum_{l=1}^8\mathop{\lim\inf}_{\epsilon\to0}{F}_\epsilon^{{\rm div},l}(u_\epsilon,v_\epsilon)\,.
\label{decompo2}
\end{equation}
With fixed $\eta>0$ and $\wt\Omega_\eta$ defined as in the proof of Lemma~\ref{le:technical2}
we argue for $l=1$ and claim that
\begin{equation}
\mathop{\lim\inf}_{\epsilon\to0} {F}_\epsilon^{{\rm div},1}(u_\epsilon,v_\epsilon)\geq \frac{1}{8} \int_{\wt\Omega_\eta}({\rm div}\,u)^2\,\mathrm{d}x\,.
\label{claime}
\end{equation}

For this, we start by defining two other  piecewise constant interpolations  $\tilde{u}_\epsilon$ and $\tilde{v}_\epsilon$ of $u_\epsilon$ and $v_\epsilon$, respectively.  
For  $\alpha \in Z_{\delta}(\wt\Omega)$ and $Q_2$ as in \eqref{quadratone}, we set 
\begin{equation}
\tilde{u}_\epsilon(x):=u_\epsilon(\alpha)\,, \qquad \tilde{v}_\epsilon(x):=v_\epsilon(\alpha)\,,\quad x\in\alpha+\delta Q_2\,.
\label{firstinterp}
\end{equation}
It is immediate to check that $\tilde v_\epsilon \to 1$ in $L^1(\wt\Omega)$, and, more in general, that 
\eqref{bounds2} are satisfied. 
Indeed, for every $\alpha\in \delta\mathbb{Z}^3$ and $i=1,2,3$, by triangle inequality we have 
\begin{equation*}
\begin{split}
|v_\epsilon(\alpha{+}2\delta e_i){-}v_\epsilon(\alpha)|^2\leq 2\biggl(|v_\epsilon(\alpha{+}2\delta e_i){-}v_\epsilon(\alpha{+}\delta e_i)|^2 {+} |v_\epsilon(\alpha{+}\delta e_i){-}v_\epsilon(\alpha)|^2\biggr)\,.
\end{split}
\label{estim1}
\end{equation*}
We also have that $d(\tilde{u}_\epsilon, u)\to 0$. 
This follows arguing as in Lemma~\ref{le:technical1}.  

We introduce further interpolations of $u_\epsilon$, whose components $z_\epsilon^i$, $i=1,2,3$ are piecewise affine, defined as
\begin{equation}
z_\epsilon^i(x):=
\begin{cases}
u_\epsilon^i(\alpha)+ \frac{1}{\delta}D_\delta^{e_i} u_\epsilon(\alpha) (x_i-\alpha_i)\,, & \mbox{ if } x\in(\alpha+\delta Q_{2,i,+})\cap\wt\Omega\,,\\
\\
u_\epsilon^i(\alpha)+ \frac{1}{\delta}D_\delta^{-e_i} u_\epsilon(\alpha) (x_i-\alpha_i)\,, & \mbox{ if } x\in(\alpha+\delta Q_{2,i,-})\cap\wt\Omega\,,
\end{cases}
\label{interp1}
\end{equation}
where $Q_{2,i,\pm}$ are as in \eqref{mezquadratone}.

Notice that, by the definition \eqref{interp1}, the first component of $z_\epsilon$ is continuous across interfaces which are orthogonal to $e_1$. Indeed, clearly no discontinuity of $z^1_\epsilon(x)$ can appear at points $x$ on the interface between $\alpha+\delta Q_{2,1,+}$ and $\alpha+\delta Q_{2,1,-}$; the only points to be checked are those $\bar x$ on the boundary between $\alpha+\delta Q_{2}$ and $(\alpha+2\delta e_1)+\delta Q_{2}$. A direct computation shows that, since $\bar x\in\partial(\alpha+\delta Q_{2,1,+})\cap\partial((\alpha+2\delta e_1)+\delta Q_{2,1,-})$, one has
\[
\lim_{x\to \bar x}z^1_\epsilon(x) = u_\epsilon^1(\alpha+\delta e_1)\,,
\] 
which proves the claim.

It follows that $z_\epsilon^{e_1,y} \in H^1(\wt\Omega_{e_1, y})$ for $\mathcal{H}^2$-almost every $y\in \Pi^{e_1}$. A similar argument shows that  $z_\epsilon^{e_i,y} \in H^1(\wt\Omega_{e_i, y})$ for $\mathcal{H}^2$-almost every $y\in \Pi^{e_i}$ for every $i=2,3$.
We now prove that $z_\epsilon \to u$ in measure on $\wt\Omega$. It will be enough to show that 
\begin{align*}
&\tilde{v}_\epsilon\left (z_\epsilon-\tilde u_\epsilon\right) \to 0 \mbox{ in $L^1(\wt\Omega)$ }.
\end{align*}
To see this, again we may argue componentwise and observe that, since 
$|\delta Q_2|=8\delta^3$, 
\[
\int_{\alpha +\delta Q_2}\left|\tilde{v}_\epsilon(z_\epsilon^i-\tilde u_\epsilon^i)\right|\,\mathrm{d}x \le 12\delta^3 |v_\epsilon(\alpha)|\left( |D_\delta^{e_i} u_\epsilon(\alpha)|+|D_\delta^{-e_i} u_\epsilon(\alpha)|\right)\,.
\]
By the Cauchy-Schwarz inequality, and using equiboundedness of the energies, we get
\begin{equation*}
\begin{split}
&\sum_{\alpha\in Z_{\delta}(\wt\Omega)}\int_{\alpha +\delta Q_2}|\tilde{v}_\epsilon(z_\epsilon^i-\tilde u_\epsilon^i)|\,\mathrm{d}x\\
&\leq 12\delta^3 \biggl(\sum_{\alpha\in Z_{\delta}(\wt\Omega)}|v_\epsilon(\alpha)|^2|D_{\delta, e_i} u_\epsilon(\alpha)|^2)\biggr)^\frac{1}{2} \biggl(\#(Z_{\delta}(\wt\Omega))\biggr)^\frac{1}{2} \le C \delta^3 \biggl(\#(Z_{\delta}(\wt\Omega))\biggr)^\frac{1}{2} \le C\delta\\
\end{split}
\end{equation*}
which entails the convergence of  $z_\epsilon \to u$ in measure on $\wt\Omega$.

For all $\psi \in S^2$ it holds  $\langle \mathcal{E}z(x)\psi, \psi \rangle= \partial_\psi \langle z(x),\psi\rangle$, where $\partial_\psi w$ stands for the directional derivative of $w$ with respect to $\psi$. Applying this to the unitary vectors $e_i$, by \eqref{interp1} we have that
\begin{equation}
\langle \mathcal{E}z_\epsilon(x)e_i, e_i\rangle=
\left\{
\begin{array}{c}
\frac1\delta D^{e_i}_{\delta} u_\epsilon(\alpha) \quad \mbox{if } x \in \alpha+\delta Q_{2,i, +}\\[5pt]
\frac1\delta D^{-e_i}_{\delta} u_\epsilon(\alpha) \quad \mbox{if } x \in \alpha+\delta Q_{2,i, -}\,.
\end{array}
\right. 
\label{gradinterp}
\end{equation}
\begin{figure}[htbp]
\centering
\def\svgwidth{150pt}
\begingroup%
  \makeatletter%
  \providecommand\color[2][]{%
    \errmessage{(Inkscape) Color is used for the text in Inkscape, but the package 'color.sty' is not loaded}%
    \renewcommand\color[2][]{}%
  }%
  \providecommand\transparent[1]{%
    \errmessage{(Inkscape) Transparency is used (non-zero) for the text in Inkscape, but the package 'transparent.sty' is not loaded}%
    \renewcommand\transparent[1]{}%
  }%
  \providecommand\rotatebox[2]{#2}%
  \newcommand*\fsize{\dimexpr\f@size pt\relax}%
  \newcommand*\lineheight[1]{\fontsize{\fsize}{#1\fsize}\selectfont}%
  \ifx\svgwidth\undefined%
    \setlength{\unitlength}{426bp}%
    \ifx\svgscale\undefined%
      \relax%
    \else%
      \setlength{\unitlength}{\unitlength * \real{\svgscale}}%
    \fi%
  \else%
    \setlength{\unitlength}{\svgwidth}%
  \fi%
  \global\let\svgwidth\undefined%
  \global\let\svgscale\undefined%
  \makeatother%
  \begin{picture}(1,1.14788732)%
    \lineheight{1}%
    \setlength\tabcolsep{0pt}%
    \put(0,0){\includegraphics[width=\unitlength,page=1]{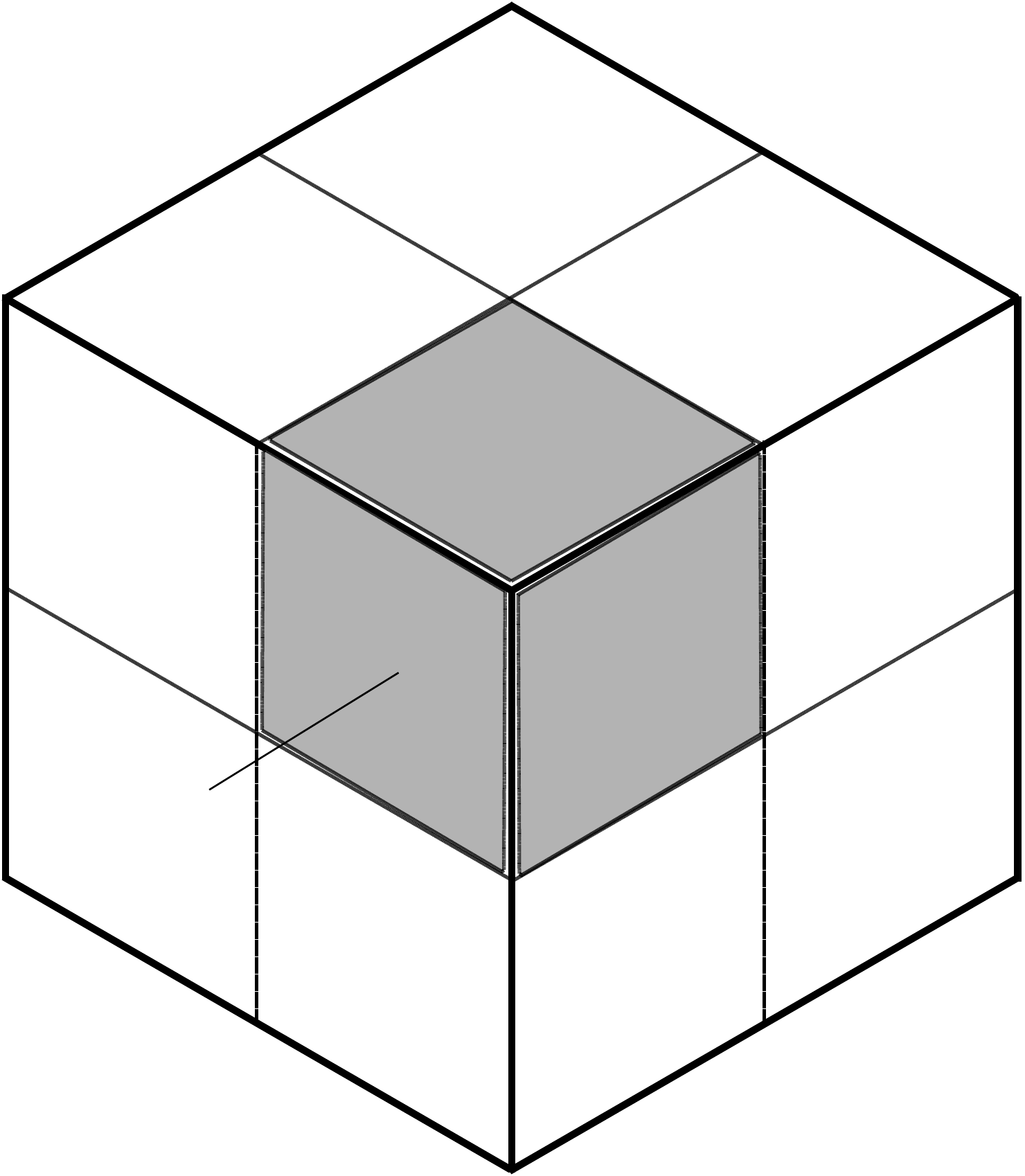}}%
    \put(0.06566305,0.33749881){\color[rgb]{0,0,0}\makebox(0,0)[lt]{\lineheight{1.25}\smash{\begin{tabular}[t]{l}\small$\mathcal{Q}^{k_1 e_1, k_2 e_2, k_3 e_3}$\end{tabular}}}}%
  \end{picture}%
\endgroup%

\caption{The cubes $\mathcal{Q}^{k_1e_1,k_2e_2,k_3e_3}$.}\label{fig:smallcubes}
\end{figure}
Then, by using the identity \eqref{matrixidentity}, we have that
\begin{equation}
({\rm div}\,z_\epsilon(x))^2= \frac{1}{\delta^2}|{\rm div}_\delta^{k_1e_1,k_2e_2,k_3e_3} u_\epsilon(\alpha)|^2 \quad\mbox{ if } x\in\delta \mathcal{Q}^{k_1e_1,k_2e_2,k_3e_3}(\alpha)
\label{interp5}
\end{equation}
for every $(k_1,k_2,k_3)\in \{-1,1\}^3$, where we have set 
\begin{equation}\label{0712192124}
\delta \mathcal{Q}^{k_1e_1,k_2e_2,k_3e_3}(\alpha):=\bigcap_{i=1,2,3}(\alpha+\delta Q_{2,i,{\rm sign}(k_i)})\,. 
\end{equation}
Since $|\delta \mathcal{Q}^{k_1e_1,k_2e_2,k_3e_3}|=\delta^3$, it holds
\begin{equation}
\begin{split}
\int_{\alpha+\delta Q_{2}}(\tilde{v}_\epsilon(x))^2({\rm div}\,z_\epsilon(x))^2\,\mathrm{d}x&=\delta(v_\epsilon(\alpha))^2\sum_{(k_1,k_2,k_3)\in\{-1,1\}^3}|{\rm div}_\delta^{k_1e_1,k_2e_2,k_3e_3} u_\epsilon(\alpha)|^2\\
&=\delta(v_\epsilon(\alpha))^2|{\rm Div}_{\delta}u_\epsilon(\alpha)|^2\,.
\end{split}
\label{formule4}
\end{equation}
Now, from the equi-boundedness of the energies \eqref{equiboundedness2}, we infer that
\begin{equation}
\sup_{\epsilon>0}\left\{H^{e_1}(z_\epsilon,\tilde{v}_\epsilon)+H^{e_2}(z_\epsilon,\tilde{v}_\epsilon)+H^{e_3}(z_\epsilon,\tilde{v}_\epsilon)\right\}<+\infty,
\label{equiboundedH}
\end{equation}
where $H^\zeta$ is defined as in \eqref{eq:H12}.
Thus, the conclusion \eqref{eq:liminfdiv} of Lemma~\ref{le:technical2} holds with $z_\epsilon$ and $\tilde{v}_\epsilon$ in place of $w_\epsilon$ and $v_\epsilon$, respectively.
Therefore, with \eqref{formule4}, it follows that
\begin{equation*}
\begin{split}
\mathop{\lim\inf}_{\epsilon\to0} {F}_\epsilon^{{\rm div},1}(u_\epsilon,v_\epsilon)&\geq \mathop{\lim\inf}_{\epsilon\to0}\left(\frac{1}{8}\int_{\wt\Omega_\eta}(\tilde{v}_\epsilon(x))^2({\rm div}\,z_\epsilon(x))^2\,\mathrm{d}x\right)\geq \frac{1}{8} \int_{\wt\Omega_\eta}({\rm div}\,u)^2\,\mathrm{d}x\,,
\end{split}
\label{stimoneprelim2}
\end{equation*}
which proves the claim \eqref{claime}.

We now observe that we have also, for every $l$ and $\eta$ small,
\begin{equation*}
{F}_\epsilon^{{\rm div},l}(u_\epsilon,v_\epsilon)\geq \frac{1}{8}\int_{\wt\Omega_\eta}(\tilde{v}_\epsilon(x))^2({\rm div}\,z_\epsilon(x))^2\,\mathrm{d}x\,.
\end{equation*}
In fact, \eqref{gradinterp}--\eqref{formule4} continue to hold, since the lattices $Z^{l}$ are just suitable translations of $Z^1\equiv Z$, while the compact subset $\wt \Omega_\eta$ of $\wt \Omega$ appears on the right-hand side.  We deduce that \eqref{claime} follows also for general ${F}_\epsilon^{{\rm div},l}$ in place of ${F}_\epsilon^{{\rm div},1}$.

By \eqref{decompo2} we eventually obtain that 
\begin{equation*}
\mathop{\lim\inf}_{\epsilon\to0} F^{\rm div}_\epsilon(u_\epsilon,v_\epsilon)\geq  \int_{\wt\Omega_\eta}({\rm div}\,u(x))^2\,\mathrm{d}x\,,
\end{equation*}
whence \eqref{bulkpart2} follows by the arbitrariness of $\eta>0$.
\endproof

 With the results proven before in this section, we are in position to prove the liminf inequality for $(E^\mathrm{Div}\ltms)_\epsilon$.

\begin{prop}\label{lowbound}
Assume that $\lim_{\varepsilon\to 0} \frac{\delta}{\varepsilon}=0$. Let $(u_\epsilon,v_\epsilon)_\epsilon\subset L^1(\wt\Omega;\mathbb{R}^d)\times L^2(\wt\Omega;\mathbb{R})$ be such that $u_\epsilon\in\mathcal{A}_\delta(\wt\Omega;\mathbb{R}^d)$, $v_\epsilon\in\mathcal{A}_\delta(\wt\Omega;\mathbb{R})$, 
\begin{equation}
\sup_\epsilon (E^\mathrm{Div}\ltms)_\epsilon(u_\epsilon,v_\epsilon)<+\infty\,,
\label{equiboundedness1}
\end{equation}
$d(u_\epsilon, u)\to 0$ for $u\in GSBD^2_\infty(\wt\Omega)$, $v_\epsilon \to 1$ in $L^2(\wt\Omega)$. 
Then 
\begin{equation}
\mathop{\lim\inf}_{\epsilon\to0} (E^\mathrm{Div}\ltms)_\epsilon(u_\epsilon,v_\epsilon)\geq {\mathcal{G}}^{\mathrm{Dir}}\ltms(u)\,.
\label{lbound}
\end{equation}
\end{prop}

\proof

Let us fix a small $\zeta \in(0,1)$. For every $\epsilon>0$, we define the discrete measures
\begin{equation*}
\begin{split}
\mu_\epsilon^\zeta:=& \frac{1}{2}\sum_{\alpha\in \wt\Omega_\delta}\delta^d\left(\frac{1}{\epsilon}(v_{ \epsilon }(\alpha)-1)^2 + \epsilon\sum_{k=1}^d\left(\frac{v_{ \epsilon }(\alpha+\delta e_k)-v_{ \epsilon }(\alpha)}{\delta}\right)^2\right) \mathbbm{1}_\alpha\\& + \frac{\zeta}{2} \sum_{\xi \in S_d} \sum_{\alpha\in R_\delta^\xi(\wt\Omega)}\delta^{d-2}(v(\alpha))^2\left|D_{\delta,\xi}u(\alpha)\right|^2 \mathbbm{1}_\alpha\,,
\end{split}
\end{equation*}
where $\mathbbm{1}_\alpha$ denotes the Dirac delta in $\alpha$. 
We observe that
\begin{equation*}
 (E^\mathrm{Div}\ltms)_\epsilon(u_\epsilon,v_\epsilon) \geq  \Big(1- \frac{\zeta}{\lambda}\Big)  \lambda F_\varepsilon(u_\varepsilon, v_\varepsilon) + \theta F_\varepsilon^{\mathrm{div}}(u_\varepsilon, v_\varepsilon) + \mu_\varepsilon^\zeta(\wt\Omega)  \,.
\end{equation*}
In view of Lemma~\ref{le:technical1} (recall Remark~\ref{rem:summation}) and Proposition~\ref{prop:lowerbound}, the general proof will be a consequence of
\begin{equation}
\mathop{\lim\inf}_{\epsilon\to0} \mu_\epsilon^\zeta(\wt\Omega)\geq \mathcal{H}^{d-1}(J_u\cap \wt \Omega)\,,
\label{claimG}
\end{equation}
by the arbitrariness of $\zeta \in (0,1)$.
Therefore we prove \eqref{claimG} in the following. 
We divide the proof into three steps: in Step~1 we see that \eqref{claimG} is guaranteed from \eqref{claimjump}; in Step~2 we show that, after a blow up procedure around a fixed $x_0$ in a set of full $\hd$-measure of $J_u$, \eqref{claimjump} would follow from \eqref{claim2}; in Step~3 we prove \eqref{claim2}. 
\medskip
\paragraph*{\textbf{Step~1.}} 
Since by \eqref{equiboundedness1} it holds that
\begin{equation*}\label{1612191800}
\sup_{\epsilon>0}\mu_\epsilon^\zeta(\wt\Omega)<+\infty\,,
\end{equation*}
we have that there exists a positive bounded Radon measure $\mu^\zeta$ such that, up to subsequences, $\mu_\epsilon^\zeta\wstar\mu^\zeta$  weakly$^*$ in $\mathcal{M}_b^+(\wt\Omega)$.
Since $J_u$ is countably rectifiable, so that $\hd \res J_u$ is $\sigma$-finite, and $\mu^\zeta \in \mathcal{M}_b^+(\wt\Omega)$, then the Radon-Nikodym derivative of $\mu^\zeta$ with respect to $\hd\res J_u$ exists (cf.\ e.g.\ \cite[Theorem~2.9]{ChaCri19}). Denoting its density by $\mu^\zeta_J \in L^1(J_u;\R^+)$, we have that $\mu^\zeta_J$ may be explicitly computed by (see e.g.\ \cite[Theorems~1.28 and 2.83]{AFP})  
\begin{equation}\label{(4.7)}
\mu^\zeta_J(x_0)=\lim_{\rho\to0^+}\frac{\mu^\zeta(Q_\rho^\nu(x_0))}{\mathcal{H}^{d-1}(Q_\rho^\nu(x_0)\cap J_u)}=\lim_{\rho\to0^+}\frac{\mu^\zeta(Q_\rho^\nu(x_0))}{\rho^{d-1}}\,,\quad\text{for } \hd\text{-a.e.\ } x_0 \in J_u\,,
\end{equation}
where $\nu:=\nu_u(x_0)$ and $Q^\nu_\rho(x_0)=x_0+ \rho\, Q^\nu$, $Q^\nu$ being the unitary cube centered in $x_0$ with two faces in planes orthogonal to $\nu$. 
Let us  set $(Q_\rho^\nu(x_0))^\pm:=x_0+\rho \, Q^{\nu,\pm}=\{x\in Q_\rho^\nu(x_0):\, \pm\langle x-x_0,\nu\rangle>0\}$ for the following discussion.

We now claim that
\begin{equation}
\mu^\zeta_J(x_0)\geq1\quad \mbox{ for $\mathcal{H}^{d-1}$-a.e. }x_0\in J_u\,.
\label{claimjump} 
\end{equation}
Once \eqref{claimjump} has been proved, the conclusion \eqref{claimG} follows by a standard argument. Indeed, by choosing an increasing sequence of cut-off functions $(\varphi_k)\subset C_c^\infty(\wt\Omega)$ such that $0\leq\varphi_k\leq1$ and $\sup_k\varphi_k=1$, we get
\begin{equation*}
\mathop{\lim\inf}_{\epsilon\to0}\mu^\zeta_\epsilon(\wt\Omega)\geq \mathop{\lim\inf}_{\epsilon\to0}\int_{\wt\Omega}\varphi_k\mathrm{d}\mu^\zeta_\epsilon = \int_{\wt\Omega} \varphi_k\mathrm{d}\mu^\zeta\geq \int_{J_u\cap{\wt\Omega}} \varphi_k\mathrm{d}\mu^\zeta_J\,, \\
\end{equation*}
whence \eqref{claimG} follows letting $k\to+\infty$ by the Monotone Convergence Theorem. 

\medskip
\paragraph*{\textbf{Step~2.}}  Since $u \in GSBD^2_\infty({\wt\Omega})$, we may subdivide $J_u$ into $J_u \cap ({\wt\Omega} \sm A^\infty_u)$ and $\partial^* A^\infty_u$. Moreover, $\widetilde{u}_t:=u\chi_{{\wt\Omega}\sm A^\infty_u} + t \chi_{A^\infty_u} \in GSBD^2({\wt\Omega})$ and $J_u=J_{\widetilde{u}_t}$ (up to a $\hd$-negligible set) for $\Ld$-a.e.\ $t \in \Rd$. Therefore, for $\hd$-a.e.\ $x_0 \in J_u \cap ({\wt\Omega} \sm A^\infty_u)$ there exist two values $u^\pm(x_0) \in \Rd$ such that 
\begin{equation}\label{(4.8)}
\mathop{{\rm ap}\lim}_{\substack{x\in(Q_\rho^\nu(x_0))^\pm\\x\to x_0}}u(x)=u^\pm(x_0)\,;
\end{equation}
 moreover, for $\hd$-a.e.\ $x_0 \in \partial^* A^\infty_u$, assuming that $\nu$ is the outer normal to ${\wt\Omega}\sm A^\infty_u$, it holds that there exists $u^-(x_0)\in \Rd$ such that
\begin{equation}\label{2110192100}
\mathop{{\rm ap}\lim}_{\substack{x\in(Q_\rho^\nu(x_0))^-\\x\to x_0}}u(x)=u^-(x_0)\,, \qquad \mathop{{\rm ap}\lim}_{\substack{x\in(Q_\rho^\nu(x_0))^+\\x\to x_0}}\tanh (|u(x)|)=1\,.
\end{equation} 
In fact, the latter identity may be seen by considering the $GSBD^2$ function $\widetilde{u}_t$ for a $t$ for which $J_u=J_{\widetilde{u}_t}$, so that $x_0 \in J_{\widetilde{u}_t}$. Thus the approximate limit of $\widetilde{u}_t$ as $x\to x_0$ in $(Q_\rho^\nu(x_0))^+$ is $t$; on the other hand, we have that $\widetilde{u}_t(x)=t$ if and only if $|u(x)|=+\infty$, so we deduce the latter identity in \eqref{2110192100}.

Let us fix $x_0 \in J_u$ such that \eqref{(4.7)} and either \eqref{(4.8)} (if $x_0 \in \wt\Omega\sm A^\infty_u$) or \eqref{2110192100} (if $x_0 \in \partial^* A^\infty_u$) hold. Notice that this corresponds to fix $x_0$ in a subset of $J_u$ of full $\hd$-measure. 
 Since $\mu^\zeta \in \mathcal{M}_b^+({\wt\Omega})$, we have that $\mu^\zeta(Q_\rho^\nu(x_0))=\mu^\zeta(\overline{Q_\rho^\nu(x_0)})$ except for a countable family of $\rho$'s. Moreover, for $\rho$ small the upper semicontinuous function $\chi_{\bar{Q}_\rho}$ has compact support in ${\wt\Omega}$. Thus, in view of \cite[Proposition~1.62(a)]{AFP} and the Besicovich Derivation Theorem (see, e.g., \cite[Theorem~2.22]{AFP}) we infer that for every $\rho_m\to0$ and every $\epsilon_j\to0$ it holds that
\begin{equation*}
\mu^\zeta_J(x_0)\geq \lim_{m\to+\infty}\mathop{\lim\sup}_{j\to+\infty}\frac{\mu^\zeta_{\epsilon_j}(Q_{\rho_m}^\nu(x_0))}{\rho_m^{d-1}}\,,
\end{equation*}
so that we need an estimate from below of $\frac{\mu^\zeta_{\epsilon_j}(Q_{\rho_m}^\nu(x_0))}{\rho_m^{d-1}}$. For this, we first note that for every $j$ and for every $m$ we can find $x_0^j\in \delta_j\mathbb{Z}^d$ and $\rho_{m,j}>0$ such that $x_0^j\to x_0$, $\rho_{m,j}\to\rho_m$ as $j\to+\infty$ and $\delta_j\mathbb{Z}^d\cap Q_{\rho_{m,j}}^\nu(x_0^j)=\delta_j\mathbb{Z}^d\cap Q_{\rho_{m}}^\nu(x_0)$. Now,  setting in correspondence to $\delta_j=\delta(\epsilon_j)$
\[
\tau_{m,j}:=\frac{\delta_j}{\rho_{m,j}}\,,\qquad  \sigma_{m,j}:=\frac{\epsilon_j}{\rho_{m,j}}\,,  
\]
we introduce the functions $u_{j,m} \in \mathcal{A}_{\tau_{m,j}}(Q^\nu;\Rd)$, $v_{j,m} \in \mathcal{A}_{\tau_{m,j}}(Q^\nu;\R)$ characterized by   the  following  ``change of variables  in the nodes'' 
\begin{equation}
u_{j,m}(\beta):=u_{\epsilon_j}(x_0^j+\rho_{m,j}\beta)\,,\quad v_{j,m}(\beta):=v_{\epsilon_j}(x_0^j+\rho_{m,j}\beta) \quad \mbox{ for every }\beta\in\frac{\delta_j}{\rho_{m,j}}\mathbb{Z}^d\cap Q^\nu.
\label{changevar}
\end{equation}
Let $G_{\sigma_{m,j}}$ and $F_{\sigma_{m,j}}$ be defined by replacing, in \eqref{energiesG} for $G_{\sigma_{m,j}}$, both $\delta_{m,j}$ with $\tau_{m,j}$ and $\epsilon_{m,j}$ with $\sigma_{m,j}$, and, in \eqref{energiesF} for  $F_{\sigma_{m,j}}$, $\delta_{m,j}$ with $\tau_{m,j}$. 
We find that
\begin{equation}\label{2210191014}
\frac{\mu^\zeta_{\epsilon_j}(Q_{\rho_m}^\nu(x_0))}{\rho_m^{d-1}}\geq \left(\frac{\rho_{m,j}}{\rho_m}\right)^{d-1} G_{ \sigma_{m,j}}(v_{j,m}, Q^\nu) + \zeta F_{\sigma_{m,j}}(u_{j,m}, v_{j,m}) \frac{\rho_{m,j}^{d-2}}{\rho_m^{d-1}}\,. 
\end{equation}
In particular we have that 
\begin{equation}\label{2210191050}
\sup_{m,j} F_{\sigma_{m,j}}(u_{j,m}, v_{j,m})< \sup_{m,j} \frac{1}{\zeta} \frac{\rho_{m}^{d-1}}{\rho_{m,j}^{d-2}} \frac{\mu^\zeta_{\epsilon_j}(Q_{\rho_m}^\nu(x_0))}{\rho_m^{d-1}} < \frac{1}{\zeta} \sup_{m,j} \frac{\mu^\zeta_{\epsilon_j}(Q_{\rho_m}^\nu(x_0))}{\rho_m^{d-1}}<+\infty\,.
\end{equation}
Notice that we used above that $\zeta>0$ is fixed, and it holds indeed that $\lim_{m,j} F_{\sigma_{m,j}}(u_{j,m}, v_{j,m})=0$.

 By \eqref{2210191014}, \eqref{2210191050}, Proposition~\ref{prop:compactness}, and Theorem~\ref{th: GSDBcompactness}, we obtain that $(u_{j,m}, v_{j,m})_{j,m}$ converges, up to a subsequence, towards a suitable couple in $GSBD^2_\infty({\wt\Omega}) \times L^2({\wt\Omega})$. 

 Moreover, setting $u_m(y):=u(x_0+\rho_m y)$ for $y \in Q^\nu$, it holds that $(u_m)_m$ converges in $L^0(Q^\nu;\Rd)$ to 
\begin{equation}\label{2210191215}
u_0(x):=
\begin{cases}
u^+(x_0), & \mbox{ if }\langle x-x_0,\nu\rangle\geq0\,,\\
u^-(x_0), & \mbox{ if }\langle x-x_0,\nu\rangle<0\,,
\end{cases} \qquad \text{if }x_0 \in J_u \cap ({\wt\Omega}\sm A^\infty_u)\,,
\end{equation}
while, if $x_0 \in \partial^* A^\infty_u$, we have that $u_m|_{Q^{\nu,-}}$ converges in $L^0({Q^{\nu,-}};\Rd)$ to $u_0^-(x):= u^-(x_0)$ in ${Q^{\nu,-}}$ and that $\tanh(|u_m|)|_{{Q^{\nu,+}}}$ converges in $L^1({Q^{\nu,+}};\Rd)$ to the constant function 1. 
Since, for fixed $m$, $u_{j,m}$, $v_{j,m}$ converge in measure to $u_m$, $v_m$ as $j\to +\infty$, by a diagonal argument we may find a sequence $m_j\to+\infty$ such that  
 the above properties hold for $u_j:=u_{j,m_j}$ as $j\to +\infty$ in place of $u_m$ as $m\to +\infty$ and $v_j:=v_{j,m_j}\to 1$ in $L^2(Q^\nu)$, $\sigma_j:=\sigma_{m_j,j}\to0$, $\tau_j:=\tau_{m_j,j}\to0$, and
 \begin{equation*}
\mu^\zeta_J(x_0)\geq \mathop{\lim\inf}_{j\to+\infty}G_{\sigma_j}(v_j,Q^\nu)\,. 
\end{equation*}

 We now collect these informations and the fact that $(u_j, v_j)_j$ converges $\Ld$-a.e., up to a subsequence (see discussion below \eqref{2210191050}). Therefore
\begin{equation}\label{2210191322}
u_j \to u_0 \in GSBD^2_\infty({\wt\Omega}) \qquad \Ld\text{-a.e.\ in }Q^\nu
\end{equation} 
and $v_j \to 1$ in $L^2(Q^\nu)$, where $u_0$ is given by \eqref{2210191215} if $x_0 \in J_u \cap ({\wt\Omega}\sm A^\infty_u)$ and by
\begin{equation}\label{2210191216}
u_0(x):=
\begin{cases}
 u_0(x) \in \widetilde{\R}^d \sm \Rd,   &\mbox{ if }\langle x-x_0\,,\nu\rangle\geq0\,,\\
  u^-(x_0), & \mbox{ if }\langle x-x_0\,,\nu\rangle<0\,,
\end{cases} \qquad \text{if }x_0 \in \partial^* A^\infty_u\,.
\end{equation}

Thus, \eqref{claimjump} (and then the result) would follow from
\begin{equation}
\mathop{\lim\inf}_{j\to+\infty}G_{\sigma_j}(v_j,Q^\nu)\geq1\,,
\label{claim2}
\end{equation}
that we show in the remaining part of the present proof.
\medskip
\paragraph*{\textbf{Step~3.}}
Up to passing to a subsequence, we may assume that the liminf in \eqref{claim2} is actually a limit. Now, we consider a suitable triangulation $\mathcal{T}_j^d$ of $Q^\nu$, as introduced in Proposition~\ref{prop:compactness}. Namely, we set
\begin{equation*}
\mathcal{T}^d_j:=\{\alpha+\delta T:\quad T\in\Sigma_d,\, \alpha\in\tau_j\mathbb{Z}^d\cap Q^\nu\}\,.
\end{equation*}
We then denote by $\hat{u}_j=(\hat{u}_j^1,\hat{u}_j^2,\dots, \hat{u}_j^d)$ and $\hat{v}_j$ the piecewise-affine interpolations of $u_j$ and $v_j$ on $\mathcal{T}^d_j$, respectively. We have that $\hat{u}_j\to u_0$ in measure on $Q^\nu$, and $\hat{v}_j\to1$ in $L^2(Q^\nu)$. 

With fixed $\eta>0$, by arguing as for the proof of \eqref{(4.10)} we can prove that for $j$ large
\begin{equation}
G_{\sigma_j}(v_j,Q^\nu)\gtrsim \int_{Q^\nu_{1-\eta}} \frac{(\hat{v}_j(x)-1)^2}{\sigma_j}+\sigma_j|\nabla\hat{v}_j(x)|^2\,\mathrm{d}x\,.
\label{(4.10bis)}
\end{equation} 

Now, we introduce the piecewise constant functions $\tilde{v}_{min,j}$ as in \eqref{pc4} and, along the lines of the proof of \eqref{equib1},
with \eqref{equiboundedness1} (here we use again that $\zeta>0$ is fixed in the definition of $\mu_\varepsilon^\zeta$, as done for \eqref{2210191050})
we have that 
\begin{equation*}
\mathop{\lim\inf}_{j\to+\infty}\int_{Q_{1-\eta}^\nu}(\tilde{v}_{min,j}(x))^2|\mathcal{E}\hat{u}_j(x)|^2\,\mathrm{d}x\leq C<+\infty\,,
\end{equation*}
whence we can assume, by taking a further (not relabeled) subsequence, that
\begin{equation}
\sup_j\int_{Q_{1-\eta}^\nu}(\tilde{v}_{min,j}(x))^2|\mathcal{E}\hat{u}_j(x)|^2\,\mathrm{d}x\leq C<+\infty\,.
\label{bbound}
\end{equation}
Recalling the notation for slicing in Section~\ref{sec:prelim}, for any fixed $\eta>0$ there exists $\gamma=\gamma(\eta)$ such that, setting $I_\eta:=(\frac{-1+\eta}{2},\frac{1-\eta}{2})$, it holds
\begin{equation*}
y+ t (\nu+\vartheta) \subset Q^\nu \qquad \text{ for all } y \in Q^\nu_{1-\eta}\cap\Pi^\nu,\, t \in I_\eta, \, \vartheta \in \nu^\perp,\, |\vartheta|<\gamma\,.
\end{equation*}
 Therefore, recalling also \eqref{2210191322}, \eqref{2210191216}, we infer that for $\hd$-a.e.\ $y\in Q^\nu_{1-\eta}\cap\Pi^\nu$ and  $\vartheta\in \nu^\perp$, $|\vartheta|<\gamma$ (with the notation for slicing from \eqref{section1}, \eqref{section2})
 \begin{equation*}
 \hat{u}_j^{{(\nu{+}\vartheta)},y}\in H^1(I_\eta)\,,\qquad \hat{u}_j^{{(\nu{+}\vartheta)},y}\to u_0^{{(\nu{+}\vartheta)},y} \quad \mathcal{L}^1\text{-a.e.\ in }I_\eta\,.
\end{equation*} 
 We now have that for $\hd$-a.e.\ $\vartheta\in \nu^\perp$, $0\neq |\vartheta|<\gamma$ 
\begin{equation}
\begin{split}
 J_{u_0^{\nu,y}}=\{0\}, &\text{ if } \langle u^+(x_0){-}u^-(x_0),\nu\rangle \neq 0\,,\\  J_{u_0^{{(\nu{+}\vartheta)},y}} = \{ 0\}, & \text{ if }  \langle u^+(x_0){-}u^-(x_0),\nu \rangle= 0 \,,
\end{split}
\label{convsec}
\end{equation}
for $\mathcal{H}^1$-a.e. $y\in Q^\nu_{1-\eta} \cap\Pi^\nu$.
In the case where $x_0 \in J_u\cap ({\wt\Omega}\sm A^\infty_u)$, \eqref{convsec} are readily obtained and the second expression holds true for every $\vartheta$. In the case $x_0 \in \partial^* A^\infty_u$, we regard the points where $u_0^{{(\nu{+}\vartheta)},y}$ (here possibly $\theta=0$) passes from a finite to an infinite value as jump points, that is we adopt the same convention as for $GSBD^2_\infty$ functions, and we work with the usual product between two numbers in $\widetilde{\R}$ and $\R$, setting $0\cdot (\pm \infty)=0$. By standard arguments (in the spirit of e.g.\ \cite[Lemma~2.7]{CC18Comp}), we can see that for $\hd$-a.e.\ $\xi$, $\lim_{t \to 0^+}|u_0^{\xi,y}(t)|=+\infty$ for $\hd$-a.e.\ $y$.

From now on we assume that $\langle u^+(x_0){-}u^-(x_0) , \nu \rangle \neq 0$, so that we may take $\vartheta=0$ to ease the reading. In the opposite case, we may argue in the very same way, just replacing the slices along the direction $\nu$ through the slices along a direction $\nu {+} \vartheta$, for some $\vartheta\in \nu^\perp$, $0\neq |\vartheta| < \gamma$, and considering, below \eqref{(4.19)}, $\pi^{\nu{+}\vartheta} \colon \Rd \to \Pi^\nu$ given by $\pi^{\nu{+}\vartheta}(x)=\{x+t(\nu{+}\vartheta) \colon t \in \R\} \cap \Pi^\nu$, in place of $\pi^\nu$.

Then, with \eqref{bbound} and Fubini's Theorem, we have
\begin{equation*}
\begin{split}
+\infty>C&\geq\int_{Q_{1-\eta}^\nu}(\tilde{v}_{min,j}(x))^2|\mathcal{E}\hat{u}_j(x)|^2\,\mathrm{d}x \\
&\geq \int_{Q_{1-\eta}^\nu\cap\Pi^\nu} \left(\int_{I_\eta}(\tilde{v}^{\nu,y}_{min,j}(t))^2(\dot{\hat{u}}^{\nu,y}(t))^2\,\mathrm{d}t\right)\,\mathrm{d}\mathcal{H}^{d-1}(y)\,,
\end{split}
\end{equation*}
whence we deduce the existence of a set $N\subset\Pi^\nu$ with $\mathcal{H}^{d-1}(N)=0$ such that
\begin{equation}
\sup_j\int_{I_\eta}(\tilde{v}^{\nu,y}_{min,j}(t))^2(\dot{\hat{u}}^{\nu,y}(t))^2\,\mathrm{d}t <+\infty
\label{bbound2}
\end{equation}
for every $y\in(Q^\nu_{1-\eta}\cap\Pi^\nu)\backslash N$. It is not restrictive to assume that $\hat{u}_j^{\nu,y}\in H^1(I_\eta)$ for every $y\in (Q^\nu_{1-\eta}\cap\Pi^\nu)\backslash N$. Now, let $I'_\eta$ be any open interval such that $0\in \overline{I'_\eta}\subset I_\eta$. If it were
\begin{equation}
\mathop{\lim\inf}_j\inf_{s\in I'_\eta}\tilde{v}^{\nu,y}_{min,j}(s)>0\,,
\label{liminfnon0}
\end{equation}
from \eqref{bbound2} we would infer that $\hat{u}_j^{\nu,y}\rightharpoonup u_0^{\nu,y}$ in $H^1(I'_\eta)$ and $J_{u_0^{\nu,y}}\cap I'_\eta=\emptyset$, which clearly would contradict \eqref{convsec}. 
Thus, the liminf in \eqref{liminfnon0} is 0, so that for every $y\in (Q^\nu_{1-\eta}\cap\Pi^\nu)\backslash N$ there exists a sequence $(s_j^y)_j\subset I_\eta$ complying with
\begin{equation}
\tilde{v}_{min,j}^{\nu,y}(s_j^y)\to0 \mbox{ as $j\to+\infty$ }.
\label{(4.16)}
\end{equation}
Now, with fixed $\kappa>0$, we claim that there exists a set $N_j^\kappa\subset\Pi^\nu$, with $\mathcal{H}^{d-1}(N_j^\kappa)\to0$ as $j\to+\infty$, such that for every $y\in(Q^\nu_{1-\eta}\cap\Pi^\nu)\backslash (N\cup N_j^\kappa)$ there exists $j_0:=j_0(y,\kappa)$ satisfying
\begin{equation*}
\hat{v}_{j}^{\nu,y}(s_j^y)\leq\frac{5}{4}\kappa\quad \mbox{ for every $j\geq j_0$ }.
\end{equation*}
For this, for every $\alpha\in\tau_j\mathbb{Z}^d\cap Q^\nu$ we set
\begin{equation*}
M_j^\alpha:=\max\left\{|v_j(\alpha)-v_j(\beta)|:\, \beta\in\tau_j\mathbb{Z}^d\cap Q^\nu,|\alpha-\beta|=\tau_j|\xi|,\, \xi\in S_d\right\} 
\end{equation*}
and
\begin{equation*}
\mathcal{I}_j^\kappa:=\left\{\alpha\in\tau_j\mathbb{Z}^d\cap Q^\nu:\, M_j^\alpha\geq\frac{\kappa}{2}\right\}.
\end{equation*}
From the equiboundedness of the energies \eqref{equiboundedness1} and an analogous argument as for the proof of \eqref{boundsmallset}, we deduce that there exists a constant $C>0$ such that
\begin{equation*}
\begin{split}
C\geq \sum_{\alpha\in\mathcal{I}_j^\kappa}\sum_{\substack{\beta\in\tau_j\mathbb{Z}^d\cap Q^\nu \\  |\alpha-\beta|=\tau_j}}\sigma_j\tau_j^d\left|\frac{v_j(\alpha)-v_j(\beta)}{\tau_j}\right|^2&\gtrsim \sum_{\alpha\in\mathcal{I}_j^\kappa}\sigma_j\tau_j^{d-2}(M_j^\alpha)^2\gtrsim \#(\mathcal{I}_j^\kappa){\kappa^2}\sigma_j\,\tau_j^{d-2} 
\end{split}
\end{equation*}
for every $j$, whence
\begin{equation}
\#(\mathcal{I}_j^\kappa)\leq\frac{c(\kappa)}{\sigma_j\tau_j^{d-2}}\,,\quad \mbox{ for every $j$ }.
\label{(4.19)}
\end{equation}
Let $\pi^\nu:\mathbb{R}^d\to\Pi^\nu$ be the orthogonal projection onto the hyperplane $\Pi^\nu$ and set
\begin{equation*}
N_j^\kappa:=\bigcup_{\alpha\in\mathcal{I}_j^\kappa}\pi^\nu(\alpha+\tau_j[0,1)^d)\,.
\end{equation*}
then, with \eqref{(4.19)} we infer that
\begin{equation}
\mathcal{H}^{d-1}(N_j^\kappa)\leq \sqrt{d}\tau_j^{d-1}\#(\mathcal{I}_j^\kappa)\leq \sqrt{d}c(\kappa)\frac{\delta_j}{\epsilon_j}\to0\,,\quad \mbox{as $j\to+\infty$ }.
\label{(4.20)}
\end{equation}
Now let $j\in\mathbb{N}$ be large, $y\in(Q^\nu_{1-\eta}\cap\Pi^\nu)\backslash (N\cup N_j^\kappa)$, and consider the corresponding sequence $(s_j^y)$ as defined in \eqref{(4.16)}. By the definition of $\tilde{v}_{min,j}$ we deduce the existence of $\alpha_0:=\alpha_0(y)\in(\tau_j\mathbb{Z}^d\cap Q^\nu)\backslash\mathcal{I}_j^\kappa$ such that $y+s_j^y\nu\in\alpha_0+\tau_j[0,1)^d$  
and 
$\tilde{v}_{min,j}(\alpha_0)=\min\{v_j(\alpha_0), \{v_j(\alpha_0\pm\tau_j\xi)\}_{\xi\in S_d}\}\to0$ as $j\to+\infty$. 

Therefore, for every $\kappa>0$ and every $y\in(Q^\nu_{1-\eta}\cap\Pi^\nu)\backslash (N\cup N_j^\kappa)$ there exists $j_0:=j_0(\kappa,y)\in\mathbb{N}$ such that $\tilde{v}_{min,j}(\alpha_0)<\frac{\kappa}{4}$ for every $j\geq j_0$. Moreover, since $\alpha_0\in(\tau_j\mathbb{Z}^d\cap Q^\nu)\backslash\mathcal{I}_j^\kappa$ we also have $v_j(\alpha_0)<\frac{3}{4}\kappa$, $v_j(\alpha_0\pm\tau_j\xi)<\frac{5}{4}\kappa$ for every $j\geq j_0$ and every $\xi\in S_d$. 
This implies, by convexity, that
\begin{equation}
\hat{v}_j(y+s_j^y\nu)<\frac{5}{4}\kappa\,,\quad\mbox{ for every $j\geq j_0$ }.
\label{mod1}
\end{equation}

Since in the previous argument $\kappa>0$ was chosen arbitrarily, from now on we may assume that $0<\kappa<\frac{4}{9}$. As we already know, up to a possible subsequence, $\hat{v}_j^{\nu,y}\to1$ a.e., so that we can find $r_j^y,t_j^y\in I_\eta$ such that $r_j^y<s_j^y<t_j^y$ and
\begin{equation}
\hat{v}_j^{\nu,y}(r_j^y)>1-\kappa\,, \quad \hat{v}_j^{\nu,y}(t_j^y)>1-\kappa\,,\quad\mbox{ for $j$ large enough.}
\label{mod2}
\end{equation}
Now, for every fixed $y\in(Q^\nu_{1-\eta}\cap\Pi^\nu)\backslash (N\cup N_j^\kappa)$, by using the Cauchy Inequality and taking into account \eqref{mod1}-\eqref{mod2} we obtain
\begin{equation*}
\begin{split}
\frac{1}{2}&\int_{I_\eta}\frac{(\hat{v}_j^{\nu,y}(x)-1)^2}{\sigma_j} +\sigma_j(\dot{\hat{v}}_j^{\nu,y}(x))^2\,\mathrm{d}x \\ 
&\geq\int_{r_j^y}^{s_j^y}(1-\hat{v}_j^{\nu,y}(x))|\dot{\hat{v}}_j^{\nu,y}(x)|\,\mathrm{d}x+ \int_{s_j^y}^{t_j^y}(1-\hat{v}_j^{\nu,y}(x))|\dot{\hat{v}}_j^{\nu,y}(x)|\,\mathrm{d}x\\
 &\geq 2\int_{\frac{5}{4}\kappa}^{1-\kappa}(1-z)\,\mathrm{d}z = 1-\frac{5}{2}\kappa+\frac{9}{16}\kappa^2=:1-c_\kappa>0
\end{split}
\end{equation*}
for every $j\geq j_0$.

From \eqref{(4.20)} we deduce that, up to subsequences,
\begin{equation}
\chi_{(Q^\nu_{1-\eta}\cap\Pi^\nu)\backslash (N\cup N_j^\kappa)}\to1\quad \mbox{$\mathcal{H}^{d-1}$-a.e. in $Q^\nu_{1-\eta}\cap\Pi^\nu$}
\label{(4.22)}
\end{equation}
so that
\begin{equation*}
\mathop{\lim\inf}_{j\to+\infty}\left(\frac{1}{2}\int_{I_\eta}\frac{(\hat{v}_j^{\nu,y}(x)-1)^2}{\sigma_j} +\sigma_j(\dot{\hat{v}}_j^{\nu,y}(x))^2\,\mathrm{d}x\right)\chi_{(Q^\nu_{1-\eta}\cap\Pi^\nu)\backslash (N\cup N_j^\kappa)}\geq 1-c_\kappa
\end{equation*}
for $\mathcal{H}^{d-1}$-a.e. $y\in (Q^\nu_{1-\eta}\cap\Pi^\nu)$. Finally, from \eqref{(4.10)}, the Fatou's Lemma with \eqref{(4.22)} we obtain
\begin{equation*}
\begin{split}
&\lim_{j\to+\infty}G_{\sigma_j}(v_j,Q^\nu)\geq \mathop{\lim\inf}_{j\to+\infty}\int_{Q^\nu_{1-\eta}} \frac{(\hat{v}_j(x)-1)^2}{\sigma_j}+\sigma_j|\nabla\hat{v}_j(x)|^2\,\mathrm{d}x\\
&\geq \int_{Q^\nu_{1-\eta}}\mathop{\lim\inf}_{j\to+\infty}\left(\frac{1}{2}\int_{I_\eta}\frac{(\hat{v}_j^{\nu,y}(x)-1)^2}{\sigma_j} +\sigma_j(\dot{\hat{v}}_j^{\nu,y}(x))^2\,\mathrm{d}x\right)\chi_{(Q^\nu_{1-\eta}\cap\Pi^\nu)\backslash (N\cup N_j^\kappa)}\,\mathrm{d}\mathcal{H}^{d-1}(y)\\
&\geq (1-c_\kappa) \mathcal{H}^{d-1}(Q^\nu_{1-\eta}\cap\Pi^\nu)=(1-c_\kappa)(1-\eta)\,,
\end{split}
\end{equation*}
whence \eqref{claim2} follows letting $\kappa\to0$ and then $\eta\to0$.
\endproof

\section{The upper limit for the Griffith energy}\label{sec:upperbound}

In this section we prove the $\Gamma$-limsup inequality for the convergence stated in Theorem~\ref{teo:main}. Differently to what done in the previous sections, here we argue for the reference configuration $\Omega$. The constraint $u_\varepsilon \in \mathcal{A}_\delta^\mathrm{Dir}(\Omega; \Rd)$, $v_\varepsilon \in \mathcal{A}_\delta^\mathrm{Dir}(\Omega; \R)$ for the recovery sequence will follow from the  part of the density result Theorem~\ref{thm:density} concerning the treatment of Dirichlet boundary conditions.

\begin{prop}
Assume that $\lim_{\varepsilon\to 0} \frac{\delta}{\varepsilon}=0$, and let $u\in GSBD^2({\Omega})$. Then there exists a sequence $(u_\epsilon,v_\epsilon)\in \mathcal{A}_\delta^\mathrm{Dir}(\Omega; \Rd)\times \mathcal{A}^\mathrm{Dir}_\delta(\Omega; \R)$ such that $(u_\epsilon,v_\epsilon)\to(u,1)$ in measure on ${\Omega}\times{\Omega}$ and
\begin{equation}
\mathop{\lim\sup}_{\epsilon\to0} (E^\mathrm{Dir}_{\lambda,\theta})_\epsilon(u_\epsilon,v_\epsilon)\leq { G}^\mathrm{Dir}_{\lambda,\theta}(u)\,.
\label{claimprop}
\end{equation}
\label{prop:upperbound}
\end{prop}
\begin{proof}
In view of Theorem~\ref{thm:density} and remarks below, by a diagonal argument it is not restrictive to assume that $u\in \mathcal{W}(\Omega;\Rd)$ and that $J_u$ is a closed subset of the hyperplane $\Pi^{e_d}=\{x_d=0\}$, that we denote by $K$.
To fix the notation we argue for $d=3$, the case $d=2$ being analogous.

We recall from \cite[(4.23)-(4.24)]{BBZ} the following fact about the optimal profile problem for the Ambrosio-Tortorelli functional: for fixed $\eta>0$, there exist $T_\eta>0$ and $f_\eta\in C^2([0,+\infty))$ such that $f_\eta(0)=0$, $f_\eta(t)=1$ for $t\geq T_\eta$, $f_\eta'(T_\eta)=f_\eta''(T_\eta)=0$, and 
\begin{equation*}
\int_0^{T_\eta} (f_\eta(t)-1)^2 + (f_\eta'(t))^2\,\mathrm{d}t\leq1+\eta\,.
\end{equation*}
Let  $x=(x',x_d)$ for each $x \in \Rd$, and $K_h:=\{x\in\Pi^{e_d}:\, {\rm dist}(x,K)< h\}$ for every $h>0$. 

Let $T>T_\eta$ and $\gamma_\epsilon>0$ be a sequence such that $\gamma_\epsilon/\epsilon\to0$ as $\epsilon\to0$. We set
\begin{equation*}
\begin{split}
A_\epsilon & :=\{x\in\mathbb{R}^3\colon x'\in K_{\epsilon+\sqrt{3}\delta},\, |x_d|<\gamma_\epsilon+\sqrt{3}\delta\}\,, \\
A'_\epsilon & :=\{x\in\mathbb{R}^3\colon x' \in K_{2\epsilon+\sqrt{3}\delta},\, |x_d| <\gamma_\epsilon+\sqrt{3}\delta+\epsilon T\}\,,\\
B_\epsilon & :=\{x\in\mathbb{R}^3\colon x' \in K_{\epsilon/2},\, |x_d|<\gamma_\epsilon/2\}, \ B'_\epsilon :=\{x\in\mathbb{R}^3\colon x'\in K_{\epsilon}\,,\, |x_d|<\gamma_\epsilon\}\,,
\end{split}
\end{equation*}
and $A_{\epsilon,\delta}:=A_\epsilon\cap\delta\mathbb{Z}^3$, $A'_{\epsilon,\delta}:=A'_\epsilon\cap\delta\mathbb{Z}^3$.
Notice that, for $\varepsilon$ small,
\begin{equation*}
K \subset B_\varepsilon\subset\subset B'_\varepsilon\subset\subset A_\varepsilon\subset\subset A'_\varepsilon\subset\subset \Omega\,,
\end{equation*} 
recalling that $K \subset \Omega$.
Let $\phi_\epsilon$ be a smooth cut-off function between $B_\epsilon$ and $B'_\epsilon$, and set
\begin{equation*}
u_\epsilon(x):=u(x)(1-\phi_\epsilon(x))\,.
\end{equation*}
Since $u\in W^{1,\infty}({\Omega}\backslash J_u;\R^3)$ we have $u_\epsilon\in W^{1,\infty}({\Omega};\R^3)$. Moreover, since $A'_\varepsilon$ is a compact set in $\Omega$ and $u=u_0$ in a neighborhood of $\dod$, also $u_\varepsilon=u_0$ in a neighborhood of $\dod$. 
Note also that, by the Lebesgue Dominated Convergence Theorem, $u_\epsilon\to u$ in $L^1(\Omega;\mathbb{R}^3)$. If $\psi_\epsilon$ is a cut-off function between $K_{\epsilon+\sqrt{3}\delta}$ and $K_{2\epsilon+\sqrt{3}\delta}$, we define
\begin{equation}\label{1412191104}
v_\epsilon(x):=\psi_\epsilon(x') h_\epsilon(x_d) + 1 - \psi_\epsilon(x')\,,
\end{equation}
where the function $h_\epsilon:[0,+\infty)\to\mathbb{R}$ is given by
\begin{equation*}
h_\epsilon(t):=
\begin{cases}
0 & \mbox{ if }t<\gamma_\epsilon + \sqrt{3}\delta\\
f_\eta(\frac{t-(\gamma_\epsilon+\sqrt{3}\delta)}{\epsilon}) &  \mbox{ if }\gamma_\epsilon + \sqrt{3}\delta\leq t\leq\gamma_\epsilon + \sqrt{3}\delta+\epsilon T\\
1 & \mbox{ if } t\geq \gamma_\epsilon + \sqrt{3}\delta+\epsilon T\,.
\end{cases}
\end{equation*}
By construction, $v_\epsilon\in W^{1,\infty}(\Omega)\cap C^0(\Omega)\cap C^2(\Omega\backslash A_\epsilon)$ and $v_\epsilon\to1$ in $L^1(\Omega)$.

We start proving that there exists a sequence $(\bar{u}_\epsilon,\bar{v}_\epsilon)\in \mathcal{A}^{\mathrm{Dir}}_\delta(\Omega;\R^3)\times \mathcal{A}^{\mathrm{Dir}}_\delta(\Omega; \R)$ converging in measure to $(u,1)$ on $\Omega\times\Omega$  such that
\begin{equation}
\mathop{\lim\sup}_{\epsilon\to0} F_\epsilon(\bar{u}_\epsilon, \bar{v}_\epsilon)\leq \sum_{\xi\in S_3}\sigma_{|\xi|}\int_{\Omega}\frac{1}{|\xi|^4}|\langle \mathcal{E}u(x)\xi,\xi\rangle|^2\,\mathrm{d}x\,, 
\label{part1.1}
\end{equation}
and
\begin{equation}
\mathop{\lim\sup}_{\epsilon\to0} F_\epsilon^{\rm div}(\bar{u}_\epsilon, \bar{v}_\epsilon)\leq \int_{\Omega} |{\rm div}\,u(x)|^2\, \mathrm{d}x\,.
\label{part1.2}
\end{equation}

Setting
\begin{equation*}
\bar{v}_\epsilon(x):=\min\{v_\epsilon(\alpha),1\}\,,\quad x\in\alpha+[0,\delta)^3\,, \,\alpha\in\Omega_\delta\,,
\end{equation*}
since $F_\epsilon(\cdot, \bar{v}_\epsilon)\leq F_\epsilon(\cdot,1)$ and $F_\epsilon^{\rm div}(\cdot, \bar{v}_\epsilon)\leq F_\epsilon^{\rm div}(\cdot,1)$, it will be sufficient to prove both \eqref{part1.1} and \eqref{part1.2} for the pair of admissible functions $(\bar{u}_\epsilon,1)$. Notice that $\bar{v}_\epsilon \in \mathcal{A}^\mathrm{Dir}_\delta(\Omega;\R)$ by \eqref{1412191104} and since $A'_\varepsilon$ is a compact subset of $\Omega$.

Let $\xi\in S_3$ be fixed. Define
\begin{equation}
\Omega_\delta^\xi:= \{x\in\mathbb{R}^3\colon [x-\delta\xi,x+\delta\xi] 
\subset\Omega\}\,.
\label{rangec}
\end{equation}
Since $v_\epsilon(\alpha)=0$ for all $\alpha\in A_{\epsilon,\delta}$, let $x\in \Omega_\delta^\xi\backslash A_\epsilon$: by construction, $x\pm\delta\xi\in \Omega_\delta^\xi\backslash B'_\epsilon$ and $u_\epsilon=u$ on $\Omega\backslash B'_\epsilon$. Thus, by using the Fundamental Theorem of Calculus and Jensen's inequality we deduce that
\begin{equation}
\begin{split}
\frac{1}{\delta^2}\int_{ \Omega_\delta^\xi\backslash B'_\epsilon}|D_\delta^\zeta u(x)|^2\,\mathrm{d}x & = \int_{ \Omega_\delta^\xi\backslash B'_\epsilon}\left|\left\langle \frac{u(x+\delta\zeta)-u(x)}{\delta},\frac{\zeta}{|\zeta|^2}\right\rangle\right|^2\,\mathrm{d}x\\
&=\int_{ \Omega_\delta^\xi\backslash B'_\epsilon}\left|\frac{1}{\delta|\zeta|^2}\int_0^\delta \langle \mathcal{E}u(x+t\zeta)\zeta,\zeta\rangle\,\mathrm{d}t\right|^2\mathrm{d}x\\
&\leq \frac{1}{|\zeta|^4} \int_{ \Omega_\delta^\xi\backslash B'_\epsilon}\frac{1}{\delta}\int_0^\delta |\langle \mathcal{E}u(x+t\zeta)\zeta,\zeta\rangle|^2\,\mathrm{d}t\mathrm{d}x\\
&\leq \frac{1}{|\zeta|^4}\int_{ \Omega}|\langle \mathcal{E}u(x)\zeta,\zeta\rangle|^2\,\mathrm{d}x
\end{split}
\label{(4.77)}
\end{equation}
for every $\zeta\in\{\pm\xi\}$.  
Moreover, setting $\Omega_\delta^{\rm div}:=\Omega_\delta^{e_1}\cap\Omega_\delta^{e_2}\cap\Omega_\delta^{e_3}$, by arguing as for \cite[(4.9)]{AFG}, we have that 
\begin{equation}
\frac{1}{\delta}({\rm div}_\delta^{k_1e_1,k_2e_2,k_3e_3}{u})\chi_{\Omega_\delta^{\rm div}\backslash B'_\epsilon}\to {\rm div}\,u \quad\mbox{ in $L^2(\Omega)$,}
\label{(4.88)}
\end{equation}
for every $(k_1,k_2,k_3)\in\{-1,1\}^3$.

For simplicity, we prove \eqref{(4.88)} in the case $(k_1, k_2, k_3)=(1,1,1)$. We first notice that
\begin{equation*}
\begin{split}
&\left\|\frac{1}{\delta}({\rm div}_\delta^{e_1,e_2,e_3}{u})\chi_{\Omega_\delta^{\rm div}\backslash B'_\epsilon}-{\rm div}\,u\right\|_{L^2(\Omega)}^2\\
&=\int_{\Omega_\delta^{\rm div}\backslash B_\epsilon'}\left(\frac{1}{\delta}\int_0^\delta\sum_{k=1}^3\langle(\mathcal{E}u(x+se_k)-\mathcal{E}u(x))e_k,e_k\rangle\,\mathrm{d}s \right)^2\hspace{-0.45em}\mathrm{d}x +\int_{\Omega\backslash(\Omega_\delta^{\rm div}\backslash B_\epsilon')}({\rm div}\,u(x))^2\,\mathrm{d}x\,.
\end{split}
\end{equation*}
Now, since $|\Omega\backslash(\Omega_\delta^{\rm div}\backslash B_\epsilon')|\to0$ as $\epsilon\to0$, with the absolute continuity of the integral, Jensen's inequality and the Cauchy-Schwarz inequality we deduce that
\begin{equation*}
\begin{split}
&\left\|\frac{1}{\delta}({\rm div}_\delta^{e_1,e_2,e_3}{u})\chi_{\Omega_\delta^{\rm div}\backslash B'_\epsilon}{-}{\rm div}\,u\right\|_{L^2(\Omega)}^2 \hspace{-0.5em} \hspace{-0.5em}\leq \int_{\Omega_\delta^{\rm div}\backslash B_\epsilon'}\hspace{-0.3em}\left(\sum_{k=1}^3\frac{3}{\delta}\int_0^\delta|\mathcal{E}u(x{+}se_k){-}\mathcal{E}u(x)|^2\,\mathrm{d}s\hspace{-0.3em}\right)\hspace{-0.3em}\mathrm{d}x{+} o(1)
\end{split}
\end{equation*}
as $\epsilon\to0$. Finally, as a consequence of Fubini's theorem we have
\begin{equation*}
\begin{split}
&\left\|\frac{1}{\delta}({\rm div}_\delta^{e_1,e_2,e_3}{u})\chi_{\Omega_\delta^{\rm div}\backslash B'_\epsilon}{-}{\rm div}\,u\right\|_{L^2(\Omega)}^2 \hspace{-0.5em} \leq \frac{3}{\delta}\int_0^\delta\left(\sum_{k=1}^3\int_{\Omega}|\mathcal{E}u(x{+}se_k){-}\mathcal{E}u(x)|^2\,\mathrm{d}x\hspace{-0.3em}\right)\hspace{-0.3em}\mathrm{d}s{+}o(1)
\end{split}
\end{equation*}
as $\epsilon\to0$, whence \eqref{(4.88)} follows from the continuity of translations in $L^2$.

With the estimates \eqref{(4.77)}, by summing over $\xi\in S_3$ and taking into account Remark~\ref{rem:summation} we infer that
\begin{equation}
\begin{split}
\mathop{\lim\sup}_{\epsilon\to0}\left(\frac{1}{2}\sum_{\xi\in S_3}\sigma_{|\xi|}\int_{ \Omega_\delta^\xi\backslash B'_\epsilon}\frac{1}{\delta^2}\left|D_{\delta,\xi}{u}(x)\right|^2\,\mathrm{d}x\right)&\leq \sum_{\xi\in S_3}\sigma_{|\xi|}\int_{\Omega}\frac{1}{|\xi|^4}|\langle \mathcal{E}u(x)\xi,\xi\rangle|^2\,\mathrm{d}x\,.
\end{split}
\label{contestim1}
\end{equation}
From \eqref{(4.88)}, we deduce that
\begin{equation}
\mathop{\lim\sup}_{\epsilon\to0} \left(\frac{1}{8}\int_{\Omega_\delta^{\rm div}\backslash B'_\epsilon}\frac{1}{\delta^2}\left|{\rm Div}_{\delta}{u}(x)\right|^2\,\mathrm{d}x\right)\leq  \int_\Omega |{\rm div}\,u(x)|^2\, \mathrm{d}x\,.
\label{contestim2}
\end{equation}

Now, we adapt to our case the argument of the proof of \cite[Proposition~4.4]{AFG}, which combined with \eqref{contestim1}-\eqref{contestim2} will give \eqref{part1.1}-\eqref{part1.2}.

For every $y\in(0,1]^3$, we introduce the sequence $T_y^\delta u_\epsilon$ as defined in \eqref{transldiscr} for $d=3$, which satisfies $T_y^\delta u_\epsilon(x)=u_\epsilon(\delta y +\alpha)$ for every $x\in\alpha+(0,\delta]^3$, $\alpha\in\delta\Z^3$.

Now, since for $\alpha\in\delta\Z^3$ and $\xi\in\Z^3$ we have $\delta\lfloor\frac{\alpha}{\delta}\rfloor=\alpha$ and $\delta\lfloor\frac{\alpha+\delta\xi}{\delta}\rfloor=\alpha+\delta\xi$, we get 
\begin{equation}\label{0912191038}
\begin{split}
& \int_{(0,1)^3} \left(F_\epsilon(T_y^\delta u_\epsilon,1)+F_\epsilon^{\rm div}(T_y^\delta u_\epsilon,1)\right)\,\mathrm{d}y\\
 & \leq \frac{1}{2}\sum_{\xi\in S_3}\sigma_{|\xi|}\hspace{-0.7em}\sum_{\alpha\in R_\delta^\xi(\Omega)\backslash A_\epsilon}\hspace{-0.7em}\int_{(0,1)^3}\hspace{-0.7em}\delta\left|D_{\delta,\xi}{u}(\alpha+\delta y)\right|^2\,\mathrm{d}y + \frac{1}{8}\hspace{-0.7em}\sum_{\alpha\in R_\delta^{\rm div}(\Omega)\backslash A_\epsilon}\hspace{-0.5em}\int_{(0,1)^3}\hspace{-0.7em}\delta\left|{\rm Div}_{\delta}{u}(\alpha+\delta y)\right|^2\,\mathrm{d}y\\
&=\frac{1}{2}\sum_{\xi\in S_3}\sigma_{|\xi|}\hspace{-0.7em}\sum_{\alpha\in R_\delta^\xi(\Omega)\backslash A_\epsilon}\hspace{-0.7em}\int_{\alpha+(0,\delta)^3}\frac{1}{\delta^2}\left|D_{\delta,\xi}{u}(y)\right|^2\,\mathrm{d}y +\frac{1}{8}\hspace{-0.7em}\sum_{\alpha\in R_\delta^{\rm div}(\Omega)\backslash A_\epsilon}\hspace{-0.5em}\int_{\alpha+(0,\delta)^3}\frac{1}{\delta^2}\left|{\rm Div}_{\delta}{u}(y)\right|^2\,\mathrm{d}y\\
&\leq \frac{1}{2}\sum_{\xi\in S_3}\sigma_{|\xi|}\int_{ \Omega_\delta^\xi\backslash A_\epsilon}\frac{1}{\delta^2}\left|D_{\delta,\xi}{u}(y)\right|^2\,\mathrm{d}y + \frac{1}{8}\int_{\Omega_\delta^{\rm div}\backslash A_\epsilon}\frac{1}{\delta^2}\left|{\rm Div}_{\delta}{u}(y)\right|^2\,\mathrm{d}y\,,
\end{split}
\end{equation}
whence, with \eqref{contestim1}-\eqref{contestim2}, we infer that
\begin{equation}
\mathop{\lim\sup}_{\epsilon\to0}\int_{(0,1)^3} \left(F_\epsilon(T_y^\delta u_\epsilon,1)+F_\epsilon^{\rm div}(T_y^\delta u_\epsilon,1)\right)\,\mathrm{d}y  \leq  {\mathcal{G}}_{\lambda,\theta}(u)\leq M\,.
\label{(4.13)}
\end{equation}

Moreover, with fixed $\eta>0$, \eqref{(4.13)} implies that the set
\begin{equation*}
\begin{split}
C_\eta^\epsilon:=\biggl\{z\in(0,1)^3:\, F_\epsilon(T_z^\delta u_\epsilon,1)&+F_\epsilon^{\rm div}(T_z^\delta u_\epsilon,1)\\
&\leq \int_{(0,1)^3}  \left(F_\epsilon(T_y^\delta u_\epsilon,1)+F_\epsilon^{\rm div}(T_y^\delta u_\epsilon,1)\right)\,\mathrm{d}y+\eta \biggr\}
\end{split}
\end{equation*}
has strictly positive Lebesgue measure for $\epsilon$ small enough. Indeed, for $\epsilon$ small enough and with \eqref{(4.13)} we deduce that
\begin{equation*}
|(0,1)^3\backslash C_\eta^\epsilon|\leq \frac{\int_{(0,1)^3}  \left(F_\epsilon(T_y^\delta u_\epsilon,1)+F_\epsilon^{\rm div}(T_y^\delta u_\epsilon,1)\right)\,\mathrm{d}y}{\int_{(0,1)^3}  \left(F_\epsilon(T_y^\delta u_\epsilon,1)+F_\epsilon^{\rm div}(T_y^\delta u_\epsilon,1)\right)\,\mathrm{d}y+\eta}\leq \frac{M}{M+\eta}<1\,,
\end{equation*}
so that $|C_\eta^\epsilon|\geq 1-\frac{M}{M+\eta}>0$. Now, as a consequence of Lemma~\ref{lemmatransl}(ii) we deduce that, for every $\epsilon>0$, there exists $z_\epsilon\in C_\eta^\epsilon$ such that $T_{z_\epsilon}^\delta u_\epsilon\to u$ in $L^1$ and
\begin{equation}
F_\epsilon(T_{z_\epsilon}^\delta u_\epsilon,1)+F_\epsilon^{\rm div}(T_{z_\epsilon}^\delta u_\epsilon,1)\leq \int_{(0,1)^3} \left(F_\epsilon(T_y^\delta u_\epsilon,1)+F_\epsilon^{\rm div}(T_y^\delta u_\epsilon,1)\right)\,\mathrm{d}y+\eta\,.
\label{(4.14)}
\end{equation}
Finally, setting $\bar{u}_\epsilon:=T_{z_\epsilon}^\delta u_\epsilon$, with \eqref{(4.13)}-\eqref{(4.14)} we obtain
\begin{equation}
\begin{split}
&\mathop{\lim\sup}_{\epsilon\to0}\left(F_\epsilon(\bar{u}_\epsilon,\bar{v}_\epsilon)+F_\epsilon^{\rm div}(\bar{u}_\epsilon,\bar{v}_\epsilon)\right)\\
&\leq\mathop{\lim\sup}_{\epsilon\to0}\int_{(0,1)^3} \left(F_\epsilon(T_y^\delta u_\epsilon,1)+F_\epsilon^{\rm div}(T_y^\delta u_\epsilon,1)\right)\,\mathrm{d}y+\eta\\
&\leq \sum_{\xi\in S_3}\sigma_{|\xi|}\int_{\Omega}\frac{1}{|\xi|^4}|\langle \mathcal{E}u(x)\xi,\xi\rangle|^2\,\mathrm{d}x+\theta\int_{\Omega} |{\rm div}\,u(x)|^2\, \mathrm{d}x + \eta\,,
\end{split}
\label{elasticuppxi}
\end{equation}
whence the assertion follows letting $\eta\to0$. 
We observe that $\bar{u}_\epsilon \in \mathcal{A}^\mathrm{Dir}_\delta(\Omega;\Rd)$, since $u_\varepsilon=u_0$ in a neighborhood of $\dod$.

We provide now an estimate for $G_\epsilon(\bar{v}_\epsilon)$. Setting,  for $\alpha \in \Omega_\delta$ such that $\alpha+\delta e_k\in\Omega$,
\begin{equation*}
G_\epsilon^{\alpha}(v):=\frac12\delta^3\Bigg(\frac{1}{\epsilon}(v(\alpha)-1)^2 + \epsilon\sum_{\substack{k=1\\}}^3\left(\frac{v(\alpha+\delta e_k)-v(\alpha)}{\delta}\right)^2\Bigg)\,,
\end{equation*}
we have (below we have to restrict the sums over $\alpha \in \Omega_\delta$ such that $\alpha+\delta e_k\in\Omega$, we omit it to ease the notation) 
\begin{equation}
\begin{split}
& \sum_{\alpha\in \Omega_\delta}G_\epsilon^{\alpha}(\bar{v}_\epsilon)  \leq \sum_{\alpha\in  \Omega_\delta}G_\epsilon^{\alpha}({v}_\epsilon)\\
& = \sum_{\alpha\in \Omega_\delta\backslash A'_\epsilon}\frac12\delta^3\left(\frac{1}{\epsilon}(v_\epsilon(\alpha)-1)^2 + \epsilon\sum_{k=1}^3\left(\frac{v_\epsilon(\alpha+\delta e_k)-v_\epsilon(\alpha)}{\delta}\right)^2\right)\\
& + \sum_{\alpha\in (A'_{\epsilon})_\delta\backslash (A_{\epsilon})_{\delta}}\frac12\delta^3\left(\frac{1}{\epsilon}(v_\epsilon(\alpha)-1)^2 + \epsilon\sum_{k=1}^3\left(\frac{v_\epsilon(\alpha+\delta e_k)-v_\epsilon(\alpha)}{\delta}\right)^2\right)\\
& + \sum_{\alpha\in (A_{\epsilon})_{\delta}}\frac12\delta^3\left(\frac{1}{\epsilon}(v_\epsilon(\alpha)-1)^2 + \epsilon\sum_{k=1}^3\left(\frac{v_\epsilon(\alpha+\delta e_k)-v_\epsilon(\alpha)}{\delta}\right)^2\right)\,.
\end{split}
\label{stim0}
\end{equation}

The argument now follows the proof of \cite[Proposition~4.2]{BBZ}, so that we will only recall briefly the main steps.

First, we note that
\begin{equation}
\sum_{\alpha\in \Omega_\delta\backslash A'_\epsilon}G_\epsilon^{\alpha}({v}_\epsilon)=0
\label{stim1}
\end{equation}
since for any $\alpha\in \Omega_\delta\backslash A'_\epsilon$ we have that $v_\epsilon(\alpha+\delta e_k)=v_\epsilon(\alpha)=1$ for every $k=1,2,3$ from the definition of $h_\epsilon$. Then, by exploiting also the regularity of $v_\epsilon$, it can be proved that (for $B_\delta:=B\cap \delta \mathbb{Z}^3$ for every $B\subset \R^3$ Borel)
\begin{equation}
\sum_{\alpha\in (A_{\epsilon})_{\delta}}G_\epsilon^{\alpha}({v}_\epsilon)\leq C \left(\frac{\gamma_\epsilon+\delta}{\epsilon}\right)\mathcal{H}^2(K_{\epsilon+\sqrt{3}\delta})\to0\,.
\label{stim2}
\end{equation}
Indeed, we have that $v_\epsilon(\alpha)=v_\epsilon(\alpha+\delta e_k)=0$, $k=1,2,3$ for every $\alpha\in (A_{\epsilon})_{\delta}$ such that $\alpha+\delta e_k\in (A_{\epsilon})_{\delta}$ for every $k=1,2,3$. This implies that 
\begin{equation*}
\sum_{\{\alpha\in (A_{\epsilon})_{\delta}:\,\, \alpha+\delta e_k\in A_{\epsilon,\delta},\,\forall k=1,2,3\}}G_\epsilon^{\alpha}({v}_\epsilon)=\#(A_{\epsilon})_{\delta}\frac{\delta^3}{\epsilon}\leq C \left(\frac{\gamma_\epsilon+\delta}{\epsilon}\right)\mathcal{H}^2(K_{\epsilon+\sqrt{3}\delta})\,.
\end{equation*}
On the other hand,
\begin{equation*}
\sum_{\{\alpha\in(A_{\epsilon})_{\delta}\colon \alpha+\delta e_k\in (A'_{\epsilon})_\delta \backslash (A_{\epsilon})_{\delta},\,\forall k=1,2,3\}}G_\epsilon^{\alpha}({v}_\epsilon)=\#(\partial A_{\epsilon})_\delta\frac{\delta^3}{\epsilon}\leq C\left(\frac{\delta}{\epsilon}\right)\mathcal{H}^2(K_{\epsilon+\sqrt{3}\delta})\,.
\end{equation*}
Finally, taking into account the fact that $f_\eta$ is a Lipschitz function, we obtain
\begin{equation}
\sum_{\alpha\in (A'_{\epsilon})_\delta\backslash (A_{\epsilon})_{\delta}}G_\epsilon^{\alpha}({v}_\epsilon)\leq \left(1+\eta+C\frac{\delta}{\epsilon}\right)\mathcal{H}^2(K_{\epsilon+\sqrt{3}\delta})+C\delta\,.
\label{stim3}
\end{equation}
Now, collecting the estimates \eqref{stim0}, \eqref{stim1},\eqref{stim2}, \eqref{stim3}  
we deduce that
\begin{equation}
\mathop{\lim\sup}_{\epsilon\to0} G_\epsilon(\bar{v}_\epsilon)\leq (1+\eta)\mathcal{H}^2(J_u\cap\Omega)\,,
\label{upboundjump}
\end{equation}
whence the desired bound follows by the arbitrariness of $\eta>0$.

In addition, \eqref{upboundjump} implies that
\begin{equation*}
\mathop{\lim\sup}_{\epsilon\to0}\frac{1}{\varepsilon}\int_\Omega(\bar{v}_\epsilon(x)-1)^2\,\mathrm{d}x\leq C\,,
\end{equation*}
then $\bar{v}_\epsilon\to1$ in $L^1(\Omega)$.
\end{proof}

\begin{oss}
The argument for the proof of Proposition~\ref{prop:upperbound} shows that an analogous (local version) of the upper bound inequality \eqref{upboundjump} can be obtained also under the assumption that $l:=\displaystyle\lim_{\epsilon\to0}\frac{\delta}{\epsilon}\in(0,+\infty)$. In this case, there exists a constant $C_l>0$ such that
\begin{equation*}
\mathop{\lim\sup}_{\epsilon\to0} G_\epsilon(\bar{v}_\epsilon)\leq (1+C_l)\mathcal{H}^{d-1}(J_u)\,.
\end{equation*}
In particular, this permits to control from above the $\Gamma$-$\limsup$ of $E\Dir\ltms$ through a Griffith-type functional.
\end{oss}

\section{The non-interpenetration constraint} \label{sec:nonimp}
This section contains the proof of the $\Gamma$-convergence approximation in Theorem~\ref{thm:noninter}. The lower inequality relies on the results proven in Section~4. For the upper inequality we employ a density result for couples $(u,v)$, here recalled in Lemma~\ref{le:CCF18}, which has been shown in dimension 2 in \cite{CCF18ARMA} to prove the upper bound in a continuum approximation for the Griffith energy with a linearized non-interpenetration constraint. We give first the proof of Theorem~\ref{thm:noninter}, keeping in the last part of the section the auxiliary results. 
\begin{proof}[Proof of Theorem~\ref{thm:noninter}] 
As a preparation for $(i)$ and $(ii)$ we notice that, since $v \leq 1$, then $F_\varepsilon^{\mathrm{div}}\leq F_\varepsilon^{\mathrm{div},\mathrm{NI}}$ (see \eqref{energiesF2} and \eqref{energiesF2m}), and $E^{\lambda, \theta}_\epsilon\leq (E\NI_{\lambda, \theta})_\varepsilon$.
\medskip\\
\noindent
\textbf{Proof of (ii).} Consider $(u_\varepsilon, v_\varepsilon)_\varepsilon$  with $\sup_{\varepsilon}(E\NI_{\lambda, \theta})_\varepsilon(u_\varepsilon, v_\varepsilon)<+\infty$. In particular, from the previous observations, \eqref{equiboundcomp} holds true. Then by Proposition~\ref{prop:compactness} we have that $u_\varepsilon$ has the same pointwise limit of a suitable function $\bar{u}_\varepsilon$, that satisfies \eqref{compactbound} and $\|\bar{u}_\varepsilon\|_{L^\infty} \leq M$. Therefore (cf.\ \cite{BCDM}) $(u_\varepsilon)_\varepsilon$ converges in every $L^p(\Omega;\Rd)$, $p \in [1, \infty)$, to some $u \in SBD^2(\Omega)$ with $\|u\|_{L^\infty} \leq M$.
\medskip\\
\noindent
\textbf{Proof of (i).}  We argue for $d=3$, the case $d=2$ being analogous.
Fix $(u_\varepsilon, v_\varepsilon)_\varepsilon$ such that $\sup_{\varepsilon}(E\NI_{\lambda, \theta})_\varepsilon(u_\varepsilon, v_\varepsilon)<+\infty$ and $u_\varepsilon\to u$, $u \in SBD^2(\Omega)$. By $(ii)$, $\|u\|_{L^\infty} \leq M$ and $u_\varepsilon \to u$  in every $L^p(\Omega;\Rd)$. In particular, $\mathrm{E} u$ is a measure, with
\begin{equation}\label{Eumeasure}
\mathrm{E} u= \mathcal{E} u \, \mathcal{L}^d + ([u]\odot \nu_u) \hd \res J_u\,. 
\end{equation}  Let us show that $u$ satisfies $\mathrm{Div}^-u=\mathrm{Tr}^-(\mathrm{E}u)\in L^2(\Omega)$.
In fact, let us examine the proof of Proposition~\ref{prop:lowerbound}, with now the control on $F_\varepsilon^{\mathrm{div},\mathrm{NI}}(u_\varepsilon, v_\varepsilon)$ at hand, which improves that on $F_\varepsilon^{\mathrm{div}}(u_\varepsilon, v_\varepsilon)$. (The only difference is that in Section~4 we worked with $\wt\Omega$, now with $\Omega$; anyway, one could as well in this case obtain  the lower limit inequality imposing a Dirichlet datum).
Arguing as in that proof, we introduce the functions $z_\varepsilon$ as in \eqref{interp1}. Then $z_\varepsilon \to u$ in every $L^p(\Omega;\Rd)$, since they converge in measure to $u$ and $\|z_\varepsilon\|_{L^\infty}\leq M$. Moreover, taking the negative part of the scalar functions in \eqref{gradinterp}, we obtain that
\begin{equation}\label{0712192122}
{\rm div}^-z_\epsilon(x)= \frac{1}{\delta^2}({\rm div}_\delta^-)^{k_1e_1,k_2e_2,k_3e_3} u_\epsilon(\alpha) \quad\mbox{ if } x\in\delta \mathcal{Q}^{k_1e_1,k_2e_2,k_3e_3}(\alpha)
\end{equation}
for every $(k_1,k_2,k_3)\in \{-1,1\}^3$, where we recall the definition of 
$
\delta \mathcal{Q}^{k_1e_1,k_2e_2,k_3e_3}(\alpha)$ in \eqref{0712192124}.
Then, arguing as for \eqref{formule4}, we get that
\begin{equation}\label{0712192129}
\int_{\alpha+\delta Q_{2}}({\rm div}^-z_\epsilon(x))^2\,\mathrm{d}x=
\delta(v_\epsilon(\alpha))^2|{\rm Div}^-_{\delta}\,u_\epsilon(\alpha)|^2.
\end{equation}
Summing over $\alpha$ and recalling the control on $F_\varepsilon^{\mathrm{div},\mathrm{NI}}(u_\varepsilon, v_\varepsilon)$, \eqref{energiesF2m}, and \eqref{range}, we infer that 
if $\varepsilon>0$ is small enough then
\begin{equation*}
\|\mathrm{div}^- z_\varepsilon \|_{L^2(\wt \Omega_\eta)} \leq C\,,
\end{equation*} 
for $C>0$ depending on $M$ and $\theta$.
In view of the $L^1$ convergence of $z_\varepsilon$ to $u$, we have that $\mathrm{div}\,z_\varepsilon$ converges in the sense of distributions on $\Omega$ to $\mathrm{Div}\, u=\mathrm{Tr}(\mathrm{E} u)$. Then, arguing as in e.g.\ \cite[Proposition~1.62]{AFP}, we can see that $\mathrm{div}^- z_\varepsilon$ converges weakly in $L^2(\Omega_\eta)$ to a suitable non negative function $f$, with $f \geq \mathrm{Div}^- u$. Then the positive measure $\mathrm{Div}^- u=\mathrm{Tr}^-(\mathrm{E}u)$ is indeed in $L^2(\Omega)$.

Now, computing the negative part of the trace of the identity \eqref{Eumeasure}, we obtain the non-interpenetra\-tion condition $[u]\cdot \nu \geq 0$ $\hd$-a.e.\ on $J_u$, since $\mathrm{Div}^- u$ has no singular part.
We deduce that ${\mathcal{G}}_{\lambda,\theta}\NIM(u,1)={\mathcal{G}}_{\lambda,\theta}(u,1; \Omega) < +\infty$. Then, by Theorem~\ref{teo:main} and $E^{\lambda, \theta}_\epsilon\leq (E\NI_{\lambda, \theta})_\varepsilon$ we conclude (i). 
\medskip\\
\noindent
\textbf{Proof of (iii).} Let $u\etl$, $v\etl$ the functions provided by Lemma~\ref{le:CCF18}, in correspondence to families of $\varepsilon$, $\beta$, $l \in (0,1)$. First we show that for every  $l$, $\beta$
\begin{equation}\label{0812191742}
\begin{split}
& \limsup_{\varepsilon\to 0}
\Bigg[\sum_{\xi\in S_2}\sigma_{|\xi|}\int_{ \Omega_\delta^\xi} \frac{\lambda}{2} \frac{(v\etl)^2}{\delta^2}   \left|D_{\delta,\xi}{u\etl}\right|^2 \,\mathrm{d}x + \frac{\theta}{4}\int_{\Omega_\delta^{\rm div}} \frac{(v\etl)^2}{\delta^2}  \left|{\rm Div}^+_{\delta}{u\etl}\right|^2\,\mathrm{d}x \\& \hspace{3em}+ \frac{\theta}{4}\int_{\Omega_\delta^{\rm div}}\frac{1}{\delta^2}\left|{\rm Div}^-_{\delta}{u\etl}\right|^2\,\mathrm{d}x  \Bigg]
\leq  \\& \limsup_{\varepsilon\to 0} \Bigg[ \int_{\Omega}(v\etl)^2 \bigg(\lambda |\mathcal{E}(u\etl)|^2 + \Big(\frac{\lambda}{2} + \beta\Big) |\mathrm{div}^+\,u\etl|^2 \bigg) \dx \\& \hspace{3em} + \Big(\frac{\lambda}{2} + \beta\Big) \int_\Omega |\mathrm{div}^-\,u\etl|^2 \dx   \Bigg].
\end{split}
\end{equation}
Indeed, in view of the assumption $\lim_{\varepsilon \to 0} \frac{\delta}{\varepsilon^2}=0$ and of Remark~\ref{rem:0812191227}, we have that
\begin{equation*}
\limsup_{\varepsilon\to 0}\frac{1}{\delta}\int_0^\delta \left| \langle \mathcal{E} u\etl(x+s\xi) \xi, \xi \rangle - \langle \mathcal{E} u\etl(x) \xi, \xi \rangle \right|^2 \mathrm{d}s =0
\end{equation*}
uniformly with respect to $\xi \in \R^2 \sm\{0\}$, $\beta$, $l$, and $x$ with $\dist(x,\partial \Omega)> \delta|\xi|$.
Therefore
\begin{equation*}
\begin{split}
\limsup_{\varepsilon\to 0}\frac{1}{\delta^2}\int_{ \Omega_\delta^\xi}(v\etl)^2|D_\delta^\zeta u\etl(x)|^2\,\mathrm{d}x & =\limsup_{\varepsilon\to 0}\int_{ \Omega_\delta^\xi}\left|\frac{v\etl}{\delta|\zeta|^2}\int_0^\delta \langle \mathcal{E}u(x+t\zeta)\zeta,\zeta\rangle\,\mathrm{d}t\right|^2\mathrm{d}x\\
&=\limsup_{\varepsilon\to 0} \frac{1}{|\zeta|^4}\int_{ \Omega_\delta^\xi}(v\etl)^2|\langle \mathcal{E}u\etl(x)\zeta,\zeta\rangle|^2\,\mathrm{d}x
\end{split}
\end{equation*}
for every $\zeta\in\{\pm\xi\}$. Similarly, one estimates
\begin{equation*}
\begin{split}
&\limsup_{\varepsilon \to 0} \int_{\Omega_\delta^{\rm div}}\frac{1}{\delta^2}  |(\mathrm{div}_\delta^-)^{k_1e_1, k_2e_2} u\etl|^2 \dx 
\\&
= \limsup_{\varepsilon\to 0} \int_{\Omega_\delta^{\rm div}} \Bigg( \frac{1}{\delta}\int_0^\delta \langle \mathcal{E}u\etl(x+t(k_1e_1)) e_1, e_1\rangle + \langle \mathcal{E}u\etl(x+t(k_2e_2)) e_2, e_2\rangle \mathrm{d}t \Bigg)^2  \dx  
\\&= \limsup_{\varepsilon\to 0} \int_{\Omega_\delta^{\rm div}} \Big(  \langle \mathcal{E}u\etl(x) e_1, e_1\rangle + \langle \mathcal{E}u\etl(x) e_2, e_2\rangle  \Big)^2 \dx = \limsup_{\varepsilon\to 0} \int_{\Omega_\delta^{\rm div}} (\mathrm{div}^-u\etl)^2\dx
\end{split}
\end{equation*}
and then obtains \eqref{0812191742}.

From \eqref{0812191742} we pass to an estimate on $\lambda\sum_{\xi\in S_2}\sigma_{|\xi|}F_\epsilon^\xi(u\etl,v\etl)+\theta {F}^{\mathrm{div}, \mathrm{NI}}_\epsilon(u\etl,v\etl)$ by arguing as in the proof of Proposition~\ref{prop:upperbound}.  We thus consider  for every $y\in(0,1]^2$ the functions $T_y^\delta u\etl$, $T_y^\delta v\etl$  as defined in \eqref{transldiscr} for $d=2$ and $u= u\etl$, $v\etl$. By the definition of the operator $T_y^\delta$, arguing similarly to \eqref{0912191038} we deduce
\begin{equation}\label{0912191210}
\begin{split}
& \int_{(0,1)^2} \left(\lambda F_\epsilon(T_y^\delta u\etl,T_y^\delta v\etl)+\theta F_\epsilon^{\rm div, \mathrm{NI}}(T_y^\delta u\etl,T_y^\delta v\etl)\right)\,\mathrm{d}y\\
 &\leq \sum_{\xi\in S_2}\sigma_{|\xi|}\int_{ \Omega_\delta^\xi} \frac{\lambda}{2} \frac{(v\etl)^2}{\delta^2}   \left|D_{\delta,\xi}{u\etl}\right|^2 \,\mathrm{d}x + \frac{\theta}{4}\int_{\Omega_\delta^{\rm div}} \frac{(v\etl)^2}{\delta^2}  \left|{\rm Div}^+_{\delta}{u\etl}\right|^2\,\mathrm{d}x \\& \hspace{3em}+ \frac{\theta}{4}\int_{\Omega_\delta^{\rm div}}\frac{1}{\delta^2}\left|{\rm Div}^-_{\delta}{u\etl}\right|^2\,\mathrm{d}x\,.
\end{split}
\end{equation}
Following the very same argument as that to get \eqref{elasticuppxi}, applied to $(T_y^\delta u\etl,T_y^\delta v\etl)$ in place of $T_y^\delta u_\epsilon$, for $\eta>0$ fixed we infer that for every $l$, $\beta$, $\varepsilon$ there exists $z\etl \in (0,1)^2$ such that, setting $\bar{u}\etl:= T_{z\etl}^\delta u\etl$, $\bar{v}\etl:=  T_{z\etl}^\delta v\etl$, it holds that 
\begin{equation}\label{0912191215}
\lim_{\varepsilon\to 0}\|(\bar{u}\etl,\bar{u}\etl)-(u\etl, v\etl)\|_{L^1}=0
\end{equation}
 and 
\begin{equation}\label{0912191211}
\begin{split}
&\limsup_{\varepsilon\to 0} \left(\lambda F_\epsilon(\bar{u}\etl,\bar{v}\etl)+\theta F_\epsilon^{\rm div, \mathrm{NI}}(\bar{u}\etl,\bar{v}\etl)\right) \\& 
\leq \limsup_{\varepsilon \to 0} \int_{(0,1)^2} \left(\lambda F_\epsilon(T_y^\delta u\etl,T_y^\delta v\etl)+\theta F_\epsilon^{\rm div, \mathrm{NI}}(T_y^\delta u\etl,T_y^\delta v\etl)\right)\,\mathrm{d}y + \eta\,.
\end{split}
\end{equation}

As for the Modica-Mortola term, we first introduce a variant of $G_\varepsilon$ obtained by replacing $\alpha$ by $\alpha+\delta y$ in the expression of $G_\varepsilon$ \eqref{energiesG}, namely for every $v\colon \Omega\to \R$ measurable we set
\begin{equation*}
G^y_\epsilon(v):=\frac12\sum_{\alpha\in \Omega_\delta}\delta^2\Bigg(\frac{1}{\epsilon}(v(\alpha + \delta y)-1)^2 + \epsilon\sum_{\substack{k=1\\ \alpha+2 \delta e_k\in \Omega}}^2\left(\frac{v(\alpha+\delta (y+e_k))-v(\alpha+\delta y)}{\delta}\right)^2\Bigg).
\end{equation*}
Now we may argue exactly as done in \cite[Proposition~4.2, Step~2]{BBZ} and in the last part of the proof of Proposition~\ref{prop:upperbound}, 
with $\alpha+ \delta y$ in place of $\alpha$ (that is, the functions are evaluated in $\alpha+\delta y$ instead of $\alpha$, inside each cube $\alpha+\delta [0,1)^2$) and the role of $K$, $h_\varepsilon$ played now by $\Gamma^\beta$, $\gamma(\frac{\cdot}{\varepsilon})$ from (i) in Lemma~\ref{le:CCF18} (notice that we use the regularity of $\Gamma^\beta$ to control its discretized neighborhoods).
We then obtain that for every $y \in [0,1)^2$
\begin{equation}\label{0912191152}
\limsup_{\varepsilon\to 0} G^y_\epsilon(v\etl) \leq \limsup_{\varepsilon \to 0}  \frac12 \int_\Omega\bigg( \frac{(v\etl-1)^2}{\varepsilon} + \varepsilon |\nabla v\etl|^2 \bigg) \dx \,.
\end{equation}
Moreover, notice that for every $y \in [0,1)^2$
\begin{equation}\label{0912191209}
\limsup_{\varepsilon\to 0} G^y_\epsilon(v\etl)=\limsup_{\varepsilon\to 0} G_\epsilon(T_\delta^y v\etl)
\end{equation}

Let us choose $l$, $\beta$, $\eta$ in dependence on $\varepsilon$, vanishing as $\varepsilon\to 0$, and denote by $\bar{u}_\varepsilon$, $\bar{v}_\varepsilon$ the corresponding $\bar{u}\etl$, $\bar{v}\etl$ (before we omit the further dependence on $\eta$). By collecting \eqref{0912191210}, \eqref{0912191211},   \eqref{0912191152}, \eqref{0912191209} and (iv) in Lemma~\ref{le:CCF18}, we eventually deduce that
\begin{equation*}
\limsup_{\varepsilon \to 0} (E\NIM_{\lambda, \theta})_\epsilon(\bar{u}_\varepsilon, \bar{v}_\varepsilon) \leq {\mathcal{G}}_{\lambda,\theta}\NIM(u,1)\,,
\end{equation*}
and \eqref{0912191215} with (ii) in Lemma~\ref{le:CCF18} give that $\bar{u}_\varepsilon$, $\bar{v}_\varepsilon$ converge to $u$ and 1.
This concludes the proof of (iii).
\end{proof}

We recall the following result, which is a direct outcome of \cite{CCF18ARMA}.
\begin{lem}\label{le:CCF18}
Let $d=2$ and $u\in SBD^2(\Omega) \cap L^\infty(\Omega;\Rd)$ with $\mathrm{Div}^-u=\mathrm{Tr}^-(\mathrm{E}u) \in L^2(\Omega)$. Then, for every families of parameters $\varepsilon$, $\beta$, $l \in (0,1)$ there exist functions $v\etl \in C^\infty(\Omega; [0,1])$, $u\etl \in C^\infty(\Omega;\Rd)$ such that
\begin{itemize}
\item[(i)] for every $\beta>0$ there exists a set $\Gamma^\beta$, which is a finite union of $C^1$ hypersurfaces and of at most $C\beta/(\varepsilon l)$ isolated points (for $C>0$ a universal constant), such that $\mathcal{H}^1(J_u \triangle \Gamma^\beta) < \beta^2$ and $v\etl$ has the form
\begin{equation*}
v\etl=\gamma\Bigg( \frac{(\dist(x, \Gamma) - 16 \sqrt{2} \varepsilon l )^+ }{\varepsilon} \Bigg)
\end{equation*}
where $\gamma$ is a smooth scalar function with $\gamma(t)\in [0,1]$, $\gamma(0)=0$, $\lim_{t\to +\infty}\gamma(t)=1$. In particular,
for every $\beta$, $l$, it holds $v\etl \to 1$ in $L^2(\Omega)$ as $\varepsilon\to 0$;
\item[(ii)] for every $l$, it holds $\limsup_{\beta\to 0} \limsup_{\varepsilon\to 0}\|u\etl - u\|_{L^2}=0$;
\item[(iii)] $u\etl=\varphi_{\varepsilon l} \ast \widetilde{u}\etl$, for a suitable $\widetilde{u}\etl$ in $L^\infty(\Omega;\Rd)$ with $\|\widetilde{u}\etl\|_{L^\infty}\leq \|u\|_{L^\infty}$ and $\varphi_{\varepsilon l}=(\varepsilon l)^{-2}\varphi(\frac{\cdot}{\varepsilon l})$ for $\varphi\in C_c^\infty(B_{1/2})$ with $\int \varphi \dx=1$ a given radially symmetric mollifier;
\item[(iv)] it holds that \begin{equation*}
\begin{split}
\limsup_{l\to 0} &\limsup_{\theta\to 0} \limsup_{\varepsilon\to 0} \Bigg[ \int_{\Omega}(v\etl)^2 \bigg(\lambda |\mathcal{E}(u\etl)|^2 + \Big(\frac{\lambda}{2} + \beta\Big) |\mathrm{div}^+\,u\etl|^2 \bigg) \dx \\& + \Big(\frac{\lambda}{2} + \beta\Big) \int_\Omega |\mathrm{div}^-\,u\etl|^2 \dx +  \frac12 \int_\Omega\Big( \frac{(v\etl-1)^2}{\varepsilon} + \varepsilon |\nabla v\etl|^2 \Big) \dx  \Bigg]  \leq {\mathcal{G}}^{\lambda,\theta}(u)\,.
\end{split}
\end{equation*}
\end{itemize}
\end{lem}
\begin{proof}
Properties (i), (ii), (iii) are clear from the construction for the $\limsup$ inequality for \cite[Theorem~1]{CCF18ARMA},  in \cite[Subsections~3.1 and 3.2]{CCF18ARMA}. In particular, for (i) see (with the numeration in \cite{CCF18ARMA}) the definition of $v_\varepsilon^l$ at the beginning of Subsection~3.1 and (17), for (ii) the very last sentence of Section~3, and for (iii) the definition of $u_\varepsilon$ below (24), where $u$ has to be replaced by $u_I$, as explained below (34). 

As for (iv), this is a consequence of (18) for the Modica-Mortola part in $v$ (with a minor modification since the Modica-Mortola term here is slightly different), of (27), that states that $\mathcal{E}u\etl$ is a good approximation of $\mathcal{E}u$ where $v\etl \neq 0$ (then one can treat separately $\mathcal{E}u\etl$ and $\mathrm{div}^+u\etl$, as we did), and of (36)-(37) for the treatment of $\mathrm{div}^-u\etl$.
\end{proof}

\begin{oss}\label{rem:0812191227}
From (iii) it follows that 
\begin{equation*}
\|\nabla u\etl\|_{W^{1,\infty}(\Omega_{\varepsilon l})} \leq C (\varepsilon l)^{-2}
\end{equation*} 
for every fixed $\beta$, $l$, $\varepsilon$, $C$ depending only on $\|u\|_{L^\infty}$ and $\varphi$, and $\Omega_{\varepsilon l}:=\{x \in \Omega\colon \dist(x, \dom)>\varepsilon l\}$. In fact, for $x$, $y \in \Omega_{\varepsilon l}$
\begin{equation*}
\partial_i u\etl(x+y) -\partial_i u\etl(x)=\int_{\R^2} \big(\partial_i \varphi_{\varepsilon l}(x+y-z) - \partial_i \varphi_{\varepsilon l}(x-z)\big) u(z) \, \mathrm{d}z\,.
\end{equation*}
We deduce the claim by noticing that $\|\nabla \varphi_{\varepsilon l}\|_{W^{1,\infty}}\leq (\varepsilon l)^{-4}\|\nabla \varphi\|_{W^{1,\infty}}$ and that the above integral is indeed computed on the set $B_{\varepsilon l /2}(x+y) \cup B_{\varepsilon l/2}(x)$, with area $C(\varepsilon l)^2$.
\end{oss}

\section*{Acknowledgements}
V. Crismale has been supported by the Marie Sk\l{}odowska-Curie Standard European Fellowship No 793018, project \emph{BriCoFra}.

G. Scilla and F. Solombrino have been supported by the Italian Ministry of Education, University and Research through the Project “Variational methods for stationary and evolution problems with singularities and interfaces” (PRIN 2017).

\bibliographystyle{siam}

\bibliography{references}

\Addresses

\end{document}